\documentclass[11pt]{amsart}   	
\usepackage[top=1in,bottom=1in,left=1in,right=1in]{geometry}                		

\usepackage{xcolor}

\usepackage[unicode=true,colorlinks=true,linkcolor=blue,urlcolor=blue, citecolor=blue]{hyperref}
\usepackage{multirow}
\usepackage{array}

\usepackage{cellspace}
\usepackage{todonotes}

\usepackage{enumerate}
\usepackage{amssymb,amsfonts,amsthm,amsmath, mathrsfs}
\usepackage{thmtools}

\usepackage{graphicx}
\usepackage{xypic}
\usepackage{hhline}

\usepackage{colordvi}
\usepackage{caption}

\usepackage{amscd}
\usepackage{xcolor}

\usepackage{rotating}

\numberwithin{equation}{section}
\declaretheorem[style=plain,parent=section]{theorem}
\declaretheorem[style=plain,sibling=theorem]{corollary}
\declaretheorem[style=plain,sibling=theorem]{lemma}

\declaretheorem[style=plain,sibling=theorem]{proposition}

\declaretheorem[style=definition,sibling=theorem]{definition}

\declaretheorem[style=definition, qed=\hfill $\diamond$, sibling=definition]{example}
\declaretheorem[style=remark,sibling=theorem]{remark}

\newcommand{\Trop}{\text{Trop}}
\newcommand{\trop}{\text{trop}}

\newcommand \mult{\operatorname{mult}}

\newcommand{\RR}{\mathbb{R}}
\newcommand{\CC}{\mathbb{C}}

\newcommand{\cP}{\mathcal{P}}
\newcommand{\cQ}{\mathcal{Q}}

\newcommand{\sC}{\mathscr{C}}

\newcommand{\pr}{\mathbb{P}}
\newcommand{\PS}{\CC\{\!\{t\}\!\}}

\newcommand{\ww}{\omega}

\newcommand{\TP}{\mathbb{T}}
\newcommand{\TPr}{\TP\pr}

\newcommand{\K}{\ensuremath{\mathbb{K}}}

\newcommand{\resK}{\ensuremath{\widetilde{\K}}}

\newcommand{\Star}{\ensuremath{\operatorname{Star}}}

\newcommand{\Dn}[1]{\ensuremath{D_{#1}}}

\newcommand {\dunion}{\,\mbox {\raisebox{0.25ex}{$\cdot$} \kern-1.83ex $\cup$}
  \,}

\newcommand{\sextic}{\ensuremath{f}} 
\newcommand{\lenl}{\ensuremath{l}} 
\newcommand{\tclass}{\ensuremath{\Omega}} 
\newcommand{\bclass}{\ensuremath{\tclass^b}} 
\newcommand{\bclassns}{\ensuremath{\bclass_{ns}}} 
\newcommand{\tR}{\ensuremath{\mathscr{T}}} 

\makeatletter
\@namedef{subjclassname@2020}{%
 \textup{2020} Mathematics Subject Classification}
\makeatother

\title[A partition theorem for tropical tritangent classes]{A lifting partition theorem for tropical tritangent classes to smooth space sextic curves}

\author[M.A.~Cueto, H.~Markwig and Y.~Ren]{Mar\'ia Ang\'elica Cueto, Hannah Markwig and Yue Ren}
\date{\today} 


\begin{document}

\begin{abstract}
  The set of tritangent planes to smooth tropical space sextic curves has 15 connected components, recording continuous displacements of planes preserving the tritangency condition. These 15 tritangent classes are polyhedral complexes in $\RR^3$, and each of them contains the tropicalization of precisely eight tritangent planes to any smooth space sextic curve with the given tropicalization. 

  Prior joint work of the authors with Len confirms that each tropical tritangent plane has 0, 1, 2, 4 or 8  lifts to classical tritangent planes defined over the algebraic closure of the field over which the original algebraic curve is defined. Our main theorem states that when the input classical curve is generic, then only six out of the ten possible partitions of 8 into powers of 2 arise from lifting multiplicities of tritangent classes. Furthermore, we show that these partitions are completely determined by the dimension of a suitable connected subcomplex of the class  and the existence of a member with a tropical tangency  of a predetermined combinatorial type.
\end{abstract}

\maketitle

\section{Introduction}\label{sec:introduction}

The classical question of counting the number of tritangent planes to a smooth,   space sextic complex curve $C$ dates back to work of Clebsch and Coble~\cite{Clebsch, Coble}. By working over a non-Archimedean, non-trivially valued algebraically  closed field $\K$ with residue field of characteristic not 2 or 3 (such as $\PS$) these questions can be addressed using tropical geometry methods~\cite{CLMR25Lifting,HL17}.

 Even though the tropical count of tritangent planes to these curves is infinite, the collection of all tropical tritangent planes to a given smooth tropical space sextic curve $\Gamma$ can be grouped into 15 equivalence  classes, corresponding to continuous translations of these planes within this collection~\cite{HL17}. Jensen and Len's work~\cite{JL18} confirms that each tritangent class contain the tropicalization of exactly eight classical tritangents, which recovers the classical count.
 
In~\cite{CLMR25Lifting}, the present authors and Len revisited this question from a computational perspective with three goals in mind: determine which tritangent planes to $\Gamma$ lift to classical tritangent planes, decide their  lifting multiplicities and provide an algorithm to compute these lifts. 
In particular,~\cite[Theorem 1.2]{CLMR25Lifting} shows that these individual lifting multiplicities are  either 0, 1, 2, 4, or 8. In turn, these numbers can be determined by the local structure of $\Gamma$ at each tropical tangency point.

Since each tritangent class has eight lifts in total, it is natural to ask which combinations of non-zero lifting multiplicities can occur within a tritangent class to a fixed  $\Gamma$. This information, which we refer to as the \emph{lifting partition} of the input class, is recorded as a quadruple, indicating the number of members with lifting multiplicities 1, 2, 4 and 8, respectively.

Our first main result gives a precise answer to this question, under mild genericity assumptions.

\begin{theorem}   \label{thm:partitionThm}
Assume that the input classical curve is general relative to its tropicalization  $\Gamma$. Then, there are six possible values for the lifting partition of a given tritangent class to $\Gamma$, namely $(8,0,0,0)$, $(4,2,0,0)$,  $(0,2,1,0)$, $(0, 4, 0, 0)$, $(0,0,2,0)$ or $(0,0,0,1)$.
\end{theorem}

The proof of this statement is inspired by prior works of the first two  authors and Len on effective methods for lifting tropical bitangent lines to smooth plane tropical quartic curves~\cite{CM20, LM17}. In order to perform a systematic continuous search of liftable tritangents within a given class, it becomes essential to know the structure of these sets. Our results confirm that, unlike what occured in the quartic case, a complete classification of these classes is not need to accomplish our task for sextic curves.

Basic duality between $\RR^3$ and the space of non-degenerate tropical planes in $\RR^3$ identifies each such plane with the location of its unique vertex (up to sign).
Our second main result  describes the combinatorics underlying each tritangent class, viewed as subsets of $\RR^3$:

\begin{theorem} \label{thm:polyhedralStructure} Each tritangent class $\tclass$ to $\Gamma$ is a connected rational polyhedral complex in $\RR^3$. In general, these sets do not admit a coarsest structure as polyhedral complexes. 
  \end{theorem}

As in the quartic case, the genericity of the input sextic curve relative to $\Gamma$ ensures that not all cells of a given tritangent class $\tclass$ to $\Gamma$ contribute liftable members. For this reason, after refining the given structure on $\tclass$ if necessary, we can peel off non-contributing cells and build a connected bounded subcomplex $\bclass \subseteq \tclass$ with the same lifting partition as $\tclass$ but a simpler structure. Most notably, when $\Gamma$ has generic edge lengths, this subcomplex yields a refined version of~\autoref{thm:partitionThm}:

\begin{theorem}\label{thm:refinedPartitionGeneric}
Assume the input classical curve is generic relative to $\Gamma$. If $\Gamma$ has generic edge lengths, then the lifting partition of a tritangent class $\tclass$ to $\Gamma$ is restricted by the dimension of the associated bounded complex $\bclass$. The precise conditions are given in~\autoref{tab:classificationPartitionsGeneric}. No length restrictions need to be imposed on $\Gamma$ if $\tclass$ is either discrete or full-dimensional. \end{theorem}

            \begin{table}[tb]
              \begin{tabular}{|c|c||c|c|}
                \hline
                $\dim \bclass$ & Lifting partition of $\tclass$ &
                $\dim \bclass$ & Lifting partition of $\tclass$\\
               \hline\hline
                $3$ & $(8,0,0,0)$&
                $2$ & $(4,2,0,0)$ or $(0,4,0,0)$ \\\hline
                $1$ & $(0,2,1,0)$ or $(0,0,2,0)$ &
                $0$ & $(0,0,0,1)$\\
                \hline
              \end{tabular}
              \caption{Classification of possible lifting partitions of a tritangent class to a tropical smooth sextic curve $\Gamma$  in $\TPr^3$ with generic edge lengths.\label{tab:classificationPartitionsGeneric}}
            \end{table}

A similar result can be stated with no length restrictions when $\tclass$ has arbitrary dimension, and we do so in~\autoref{thm:refinedLiftingPartition}. In such cases, the role of $\bclass$ is replaced by a second connected subcomplex $\bclassns \subseteq \bclass$, which we call the  non-special bounded complex. Its construction is described in~\autoref{sec:non-special-bounded}. \autoref{thm:comparingDimensions} confirms that both subcomplexes have the same dimension for generic edge lengths, allowing us to obtain~\autoref{thm:refinedPartitionGeneric} as a direct corollary to~\autoref{thm:refinedLiftingPartition}. \autoref{thm:dimComparison} provides the list of 11 possible tuples recording the dimensions of these three complexes, confirming that the roles of $\bclassns$ or $\bclass$ in both versions of the refined lifting partition theorem cannot be played by  $\tclass$.

The rest of the paper is organized as follows. In~\autoref{sec:preliminaries}, we review the construction of tritangent classes to a given curve $\Gamma$, the classification of local tangencies, and the main lifting results of~\cite{CLMR25Lifting}. In particular, we review how Clebsch and Coble's classical question and its tropical counterpart can be reduced to the search of tritangent $(1,1)$-curves to smooth $(3,3)$-curves in $\pr^1\times \pr^1$ via the classical Segre embedding. \autoref{ex:realizableLifts} confirms all tuples listed in~\autoref{thm:partitionThm} can be realized for convenient choices of $\Gamma$.

In~\autoref{sec:local-moves} we present a proof of the first claim in~\autoref{thm:polyhedralStructure} by describing local moves that preserve a given tangency on $\Gamma$. \autoref{sec:polyhedra-bounded-complex} is devoted to the construction of the bounded complex $\bclass$ associated to a given tritangent class $\tclass$. \autoref{cor:bcomplexDeformation} confirms the connectivity of $\bclass$ by realizing its support as a strong deformation retract of the support of $\tclass$. In~\autoref{sec:proof-lift-part-3d} we present a proof of~\autoref{thm:refinedPartitionGeneric} when the bounded subcomplex $\bclass$ is full-dimensional. \autoref{rem:coarsestStructure} provides examples of tritangent classes with no  coarsest polyhedral complex structure, confirming the second claim in~\autoref{thm:polyhedralStructure}.

          {Sections}~\ref{sec:comb-bound-compl} through~\ref{sec:proof-lift-part-1d-or-less} are devoted to the proof of~\autoref{thm:refinedPartitionGeneric} beyond dimension 3. \autoref{tab:combClassificationGenericPerDim} summarizes the content of~\autoref{lm:4valentSingleton} and~\autoref{thm:dim210}, describing all possible generic members of top-dimensional cells of the complex $\bclass$. As expected, each of them will be the starting point in the search of liftable members within $\bclass$.
In~\autoref{sec:non-special-bounded} we introduce the non-special bounded complex $\bclassns$, show its connectedness (\autoref{thm:pathConnectednessOfNSComplexes}), and determine under which conditions both complexes $\bclassns$ and $\bclass$ have the same dimension (\autoref{thm:comparingDimensions}). The main result in this section is \autoref{thm:refinedLiftingPartition}, which generalizes~\autoref{thm:refinedPartitionGeneric}.~\autoref{sec:technicalLemmasDim1-2} complements the findings from~\autoref{tab:combClassificationGenericPerDim} by describing, through a series of combinatorial lemmas,  which trajectories within each complex $\bclassns$ must be pursued to find liftable members. Finally, 
{Sections}~\ref{sec:proof-lift-part-2d-or-less} and~\ref{sec:proof-lift-part-1d-or-less} establish the proof of the general lifting partition theorem when the complex $\bclassns$ is not full-dimensional.

The paper concludes with  \autoref{sec:comp-dimens-tclass}, in which we present a full classification of the possible tuples recording the dimensions of the complexes $\bclassns$, $\bclass$ and $\tclass$. Examples realizing all these tuples are provided, confirming the key role played by the dimension of the subcomplex $\bclassns$ in determining the lifting partition of each tritangent class.

\section*{Acknowledgments}

This project started during a week-long Research in Pairs stay at The Oberwolfach Institute for Mathematics (Germany). The authors thank the institute and its staff for their hospitality and  excellent working conditions. In addition, we thank Yoav Len for fruitful discussions during the development of this article. MAC acknowledges the Universit\`a degli Studi di Parma (Italy) for hosting her during the final stages of this work.

MAC was supported by NSF Standard Grants DMS-1700194 and DMS-1954163 (USA), as well as INdAM Research Project E53C23001740001 (Italy). HM was supported by the DFG project MA 4797/9-1 and YR was supported by UKRI Future Leaders Fellowship MR/Y003888/1 and the EPSRC grant EP/Y028872/1.

\medskip

\section{Tropical tritangent class, local tangencies and their lifting multiplicities}\label{sec:preliminaries}

Throughout, we let $C$ be a smooth sextic curve in $\pr^3$. 
Being a non-hyperelliptic genus four canonical curve, $C$ is the complete intersection of a (unique) quadric and a cubic surface in $\pr^3$. When $C$ is generic, the quadric surface is smooth, thus, up to an automorphism of $\pr^3$, we may fix it to be the standard Segre quadric $V(x_0x_1-x_2x_3)$.
Under the standard Segre embedding
\[ \varphi\colon \pr^1\times \pr^1 \to \pr^3\qquad \varphi(([s_0:s_1], [t_0:t_1])) = [s_0t_0:s_1t_0: s_0t_1: s_1t_1],\] a plane in $\pr^3$ corresponds to a $(1,1)$-curve in $\pr^1\times \pr^1$, and the space sextic yields a $(3,3)$-curve in the same ambient space. In particular, tritangent planes to the input smooth space sextic curve become  tritangent $(1,1)$-curves $V(\ell)$ to a smooth $(3,3)$-curve $V(\sextic)$ in $\pr^1\times \pr^1$. Throughout, we assume the Newton polytopes of both $\ell$ and  $\sextic$ are squares, of side lengths one and  three, respectively.

Following the \emph{max} convention in Tropical Geometry, we let $\Lambda$ and $\Gamma$ denote the tropicalization in $\TPr^1\times \TPr^1$ of  $V(\ell)$ and $V(\sextic)$, respectively. We will always assume $\Gamma$ is  smooth, i.e., dual to a unimodular triangulation of the square of side length three. For an overview of the basic constructions in Tropical Geometry relevant to our work, we refer the reader to~\cite{MSBook, Mi06}.

\begin{remark}\label{rm:actionBySn} The  dihedral group $\Dn{4}$ of order eight records automorphisms of the unit square preserving its boundary. This action extends to the space of smooth tropical curves in $\TPr^1\times \TPr^1$ of bidegree $(a,a)$ for any positive integer $a$. Concrete formulas appear in~\cite[Table 2]{CLMR25Lifting}. 
\end{remark}

\begin{figure}[t]
 \includegraphics[scale=0.35]{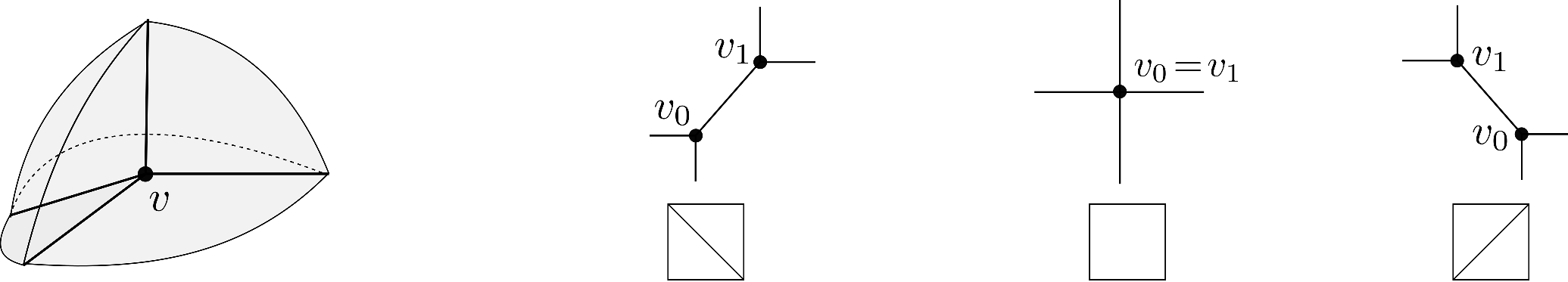}
  \caption{Correspondence between tropical planes and $(1,1)$-curves under the tropical Segre embedding.\label{fig:Planesvs11Curves}}
  \end{figure}

As shown in \autoref{fig:Planesvs11Curves}, a generic tropical plane in $\TP\pr^3$ is completely determined by the location of its unique vertex in $(\TP\pr^3)^\circ \simeq \RR^3$, which we label by $v$. In turn, a tropical $(1,1)$-curve $\Lambda$ in $\TPr^1\times \TPr^1$ is parameterized by the location of its two vertices in $(\TPr^1\times \TPr^1)^{\circ}\simeq \RR^2$, called $v_0$ and $v_1$. Note that $\Lambda$ may have a single vertex, in which case we set $v_0=v_1$.

Since the slope of the unique edge of $\Lambda$ has two fixed values, namely 1 or -1,  there is a  linear dependence between its endpoints. Thus, we may choose an alternative way to parameterize such curves, namely, by recording  the location of its lowest vertex (called $v_0$) and the signed length $\lenl$ of its unique edge. We incorporate signs to distinguish between the three possible combinatorial types of such curves. A negative length corresponds to a curve with edge of slope $-1$, whereas a positive one has slope 1. The 4-trivalent case corresponds to $\lenl=0$, when $v_0=v_1$.

Our next result shows confirms that  this choice is compatible with the  \emph{tropical Segre embedding}  $\trop(\varphi)\colon \TPr^1\times \TPr^1 \to \TPr^3$ given by left multiplication with the matrix
\[\begin{pmatrix} 1 & 1& 0 & 0\\
0 & 1 & 1 & 0\\
1 & 0 & 0 & 1\\
0 & 0 & 1 & 1
\end{pmatrix}.
\]

\begin{proposition} The tropical Segre embedding induces an isomorphism between generic tropical planes in $\TPr^3$  and generic tropical (1,1)-curves in $\TPr^1\times \TPr^1$. More precisely, it corresponds to a bijective piecewise linear map
\begin{equation}\label{eq:TropSegre}
  \psi\colon \RR^3\to \RR^3 \qquad (x,y,z) \mapsto \begin{cases} (x,z-x, z-x-y) & \text{ if } x+y\geq z, \\ (x,y, z-x-y) &\text{ if }   x+y\leq z.
  \end{cases}
\end{equation}
  \end{proposition}

\begin{proof} Following our earlier discussion, we identify $\RR^3$ with $(\TPr^3)^{\circ}$ and $(\TPr^1\times \TPr^1)^{\circ} \times \RR$, respectively. Thus, we view $v$ and $(v_0,\lenl)$ in $\RR^3$. We must show that the map $v\mapsto (v_0,\lenl)$ matches~\eqref{eq:TropSegre}.
  
  A plane $\pi$ in $\pr^3$ with generic tropicalization is determined by the vanishing of a linear equation   $a +b\, y + c\, z+ d\, w = 0$ with four non-zero terms. In turn, the (1,1)-curve obtained as the pullback of $\pi$ under the Segre embedding has defining equation $a + b \,s+ c\, t+ d \, st = 0$  for $([1:s], [1:t])\in \pr^1\times \pr^1$.

  The tropicalization of the plane $\pi$ in $\RR^3$ is given by the corner locus of the function $\max\{A, B+y, C + z, D +w\}$ where $A, B, C$ and $D$ correspond to the negative valuation of the coefficients $a, b, c$ and $d$, respectively. In particular, the unique vertex of this tropical plane equals $v=(A-B, A-C, A-D) \in \RR^3$. In turn, the tropicalization of the (1,1)-curve (denoted by $\Lambda$) is the corner locus of the function $\max\{A, B+s, C + t, D +s+t\}$ in $\RR^2$. Its combinatorial type depends on the sign of $A+D - (C+B)$, i.e. on the sign of the expression $x+y-z$  specialized at $v$.

  In what follows, we determine the values of $(v_0, \lenl)$  for each case. 
  First, assume $A+D> C+B$. In this situation, the edge of $\Lambda$ has slope -1, and its vertices have coordinates $v_0=(A-B, B-D)$, $v_1= (C-D, A-C)$. In particular, its unique edge has length $A+D - (C+B)> 0$. Thus, the signed length of this edge is $\lenl = (C+B)-(A+D)$. 
  
  In turn, if $A+D<C+B$, the edge of $\Lambda$ has slope 1, with vertices $v_0=(A-B, A-C)$, and $v_1=(C-D, B-D)$. Furthermore, the length of the edge of $\Lambda$ is $\lenl=B+C-A-D> 0$.  Finally, if $A+D=C+B$, then $v_0=v_1=(A-B, B-D)$ and $\lenl=0$.
  
  In all three cases, the formula seen in~\eqref{eq:TropSegre} matches the correspondence $v\mapsto (v_0, \lenl)$. The construction confirms it is bijective map.
\end{proof}

The next definition appeared in~\cite[Section 2]{CLMR25Lifting} and it is inspired by the classical notion of tangency between two plane curves. Note that set-intersections of tropical curves must be replaced by their stable analogs (see, e.g.,~\cite{jen.yu:16} and references therein).

\begin{definition}\label{def:tropicalTritangents} Let $\Gamma$ and $\Lambda$  be two tropical curves in $\TPr^1\times \TPr^1$ of bidegrees $(3,3)$ and $(1,1)$, respectively. Assume that $\Gamma$ is smooth. Then, we say that $\Lambda$ is \emph{tritangent} to $\Gamma$ if any of the following three conditions hold:
  \begin{enumerate}[(i)]
  \item $\Lambda\cap \Gamma$ has three components, each with total stable intersection multiplicity 2, or
    \item $\Lambda\cap \Gamma$ has two components, of total stable intersection multiplicities 2 and 4, respectively, or
\item $\Lambda \cap \Gamma$ is connected, and its stable intersection multiplicity is 6.
  \end{enumerate}
\end{definition}

To simplify the construction of  curves $\Lambda$ tritangent to  $\Gamma$ it is convenient to predetermine the  possible tangencies that can occur on $\Gamma$. The following statement, which originally appeared in \cite[Theorem 3.1]{CLMR25Lifting}, provides a complete classification of such local tangencies.

\begin{theorem}\label{thm:classificationRealizableLocalTangencies} Up to $\Dn{4}$-symmetry, there are 38 possible local tangencies between a $(1,1)$-tropical curve and a smooth $(3,3)$-tropical curve in $\TPr^1\times \TPr^1$. They are depicted in~\autoref{fig:classificationLocalTangencies}. 
  \end{theorem}

\begin{figure}[tb]
\includegraphics[scale=0.33]{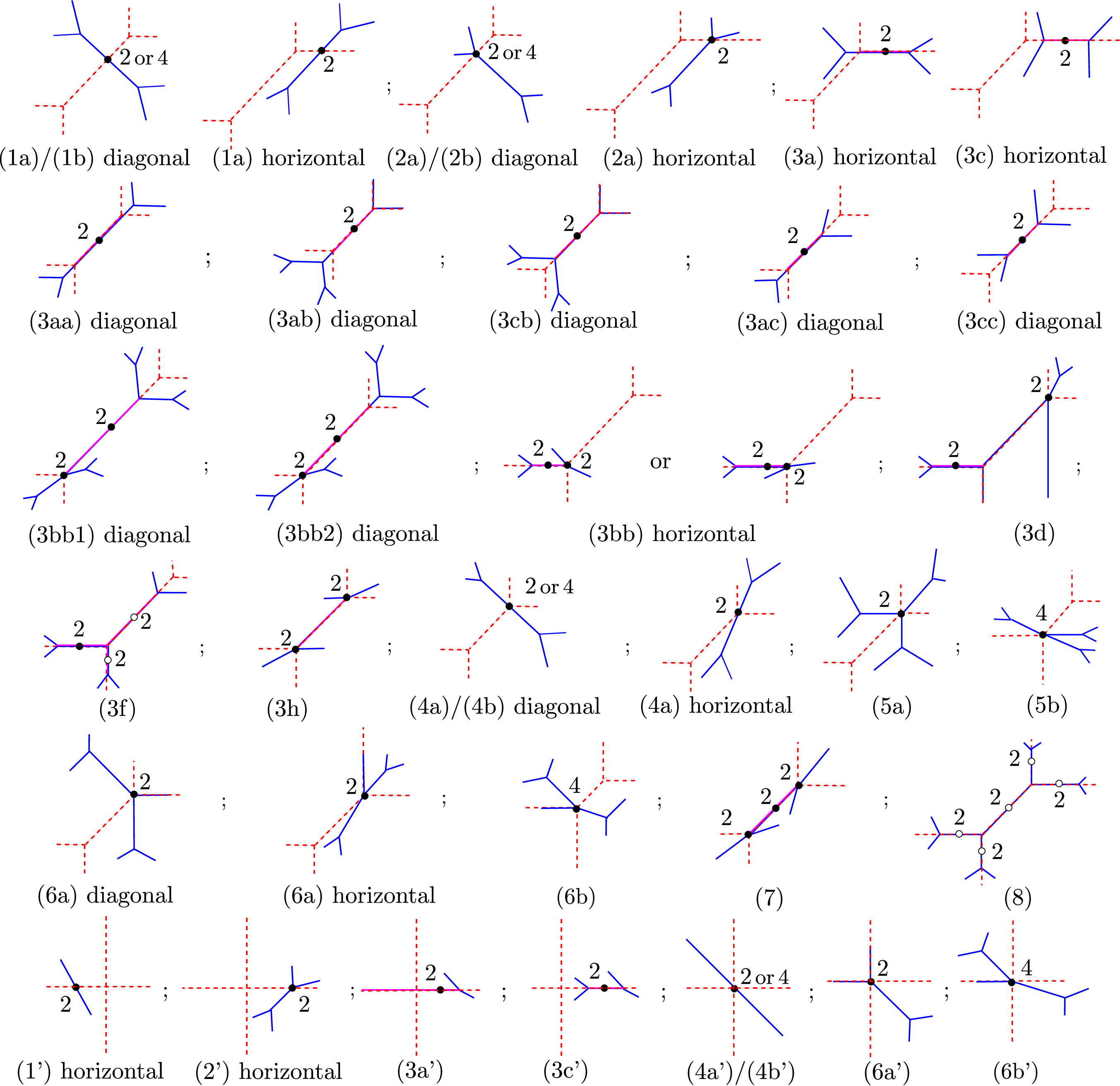}
  \caption{Representatives of all possible local tangency types between a $(1,1)$-tropical curve $\Lambda$ (in dashed lines) and a smooth $(3,3)$-tropical curve $\Gamma$, under the action of $\Dn{4}$. The `a' and `b' labels for types (1), (2), (4), (5) and (6), distinguish between tangencies of multiplicity two and four, respectively. The black dots show the location of the tropical tangency point. The numbers adjacent to each dot indicate  multiplicities values. Unfilled dots in (3f) and (8) denote potential tangencies, and their total multiplicities are four and six, respectively.\label{fig:classificationLocalTangencies}}
  \end{figure}

\begin{definition}\label{def:lift} Let $\Lambda$ be a tritangent curve to $\Gamma$, and
denote by $P$, $P'$ and $P''$ the corresponding tangency points (which are allowed to coincide). We say that $(\Lambda,P,P',P'')$ \emph{lifts} to a tritangent tuple $(\ell,p,p,p'')$ to $V(\sextic)$ if $V(\ell)$ is tangent to $V(\sextic)$ at the points $p, p$ and $p''$, with 
  \[\Trop\, V(\ell) = \Lambda, \quad \Trop\, p = P, \quad \Trop\, p' = P' \quad\text{and} \quad \Trop\, p'' = P''.
  \]
   The \emph{lifting multiplicity} of $\Lambda$ equals the number of such tritangent tuples.
\end{definition}

As we mentioned in~\autoref{sec:introduction}, the existence of infinitely many tritangent planes  to a given $\Gamma$ is a common phenomenon. Continuous displacement of the vertices of these planes preserving the tritangent condition produce equivalence classes of tritangent planes. Under the tropical Segre map, these movements correspond to continuous compatible translations of the vertices $v_0$ and $v_1$ of each $\Lambda$, or equivalently, continuous displacements of $v_0$ followed by piece-wise linear maps on the signed length component with slopes in $\{0, \pm 1\}$. The next definition summarizes these facts:

\begin{definition} A~\emph{tropical tritangent class} 
  of $\Gamma$ is a connected component of the subset of $\RR^3$ consisting of all points $(v_0,\lenl)$ encoding all smooth tropical $(1,1)$-curves that are tritangent to $\Gamma$. We label these sets by $\tclass$.
\end{definition}

The following result from~\cite[Theorem 5.2]{HL17} describes how the classical count of tritangent planes to $V(\sextic)$ can be recovered from the tritangent classes of $\Gamma$:

\begin{theorem}\label{thm:countingTritangentClassesAndLifts}
  Each $\Gamma$ has exactly 15 tritangent classes. Furthermore, each of them contains the tropicalization of exactly eight tritangent curves $V(\ell)$ to $V(\sextic)$ in $\pr^1\times \pr^1$.
\end{theorem}

\noindent These numbers are deduced from the classical proof of Coble-Clebsch formula and the computation of tropical theta characteristics on metric graphs via Zharkov's Algorithm~\cite{Zha10}.

By analogy with prior work of the first two authors on bitangent classes to smooth tropical plane quartics~\cite{CM20}, it is natural to ask if the location and lifting multiplicity of the liftable members of $\tclass$ can be detected solely from $\sextic$ and the combinatorial type of $\Gamma$.
The proof of~\autoref{thm:partitionThm} yields an affirmative answer to this question.  Most notably, we can do so  without providing a  full classification of all tropical tritangent classes to $\Gamma$. Note that obtaining such a list would not be a simple task, given the high number of local tangency types provided in~\autoref{thm:classificationRealizableLocalTangencies} and the existence of $5\,941$ possible combinatorial types for $\Gamma$ (up to $\Dn{4}$-symmetry).

Prior work of Len and the second author~\cite{LM17} shows how tropical tangencies  between two curves can often be used to compute lifting multiplicities, by solving local versions of the classical system $\ell=\sextic = \det(\nabla \ell, \nabla \sextic)=0$ determining tangencies between the input curves, with respect to the tangency points.

Our companion paper with Len~\cite{CLMR25Lifting} carries out this program successfully for the problem at hand, with mild genericity assumptions on the classical curve $V(\sextic)$, which we now make precise:

\begin{definition}\label{def:fgenericRelToGamma} Given $\Gamma$, we say that $\sextic$ is \emph{generic relative to $\Gamma$} if the following conditions hold:
\begin{enumerate}[(i)]
  \item the polynomial  $\sextic$  defines a smooth curve in $\TPr^1\times \TPr^1$ with $\Trop V(\sextic) = \Gamma$,
\item the curve $V(\sextic)$ has no hyperflexes,
\item if two tropical tangencies occur on the relative interior of the same leg or edge of $\Lambda$, the corresponding local systems are inconsistent,
  \item if $\Gamma$ is trivalent, and  $\Lambda$ has a vertex whose two adjacent legs contain no tangency points, then the $9\times 9$ system of local equations defined by the tangency points between $\Gamma$ and $\Lambda$ has no solutions in the $9$-dimensional torus over the residue field $\resK$.
  \end{enumerate}
\end{definition}

Our next statement, extracted from~\cite[Theorem 1.2]{CLMR25Lifting} confirms that the lifting multiplicity of a given $\Lambda$ can be obtained from the number of solutions of each local system.  The explicit computation of this local data, which we refer to as \emph{local lifting multiplicities}, was a main contribution of~\cite{CLMR25Lifting}. 

\begin{theorem}\label{thm:totalLifting} Let $\Gamma$ be a smooth tropical $(3,3)$-curve in $\TPr^{1}\times \TPr^1$, and assume that $\sextic$ is generic relative to $\Gamma$.  Then, any  $(1,1)$-curve $\Lambda$ tritangent to $\Gamma$ has 0, 1, 2, 4 or 8 lifts to a tritangent curve $V(\ell)$ to $V(\sextic)$. Furthermore, when non-zero, this number agrees with the product of the local lifting multiplicities of the corresponding tropical tangencies listed in~\autoref{tab:LiftingMultiplicities}.
\end{theorem}

Our next two remarks given more precisions regarding the  values collected in~\autoref{tab:LiftingMultiplicities}.

\begin{table}[t]\begin{center}
              \begin{tabular}{|c|c|}
                \hline Multiplicity & Tangency types \\
                \hline
                0 &  (1),  (2b), (3ab), (3cb), (3bb), (3bb1), (3bb2), (4b), (7); (1'),  
                (3a')  \\
                1 & (2a), (4a) h/v,   (5b), (6a) h/v,  (6b); (2a'),  (4b'),  (6b')\\
                 2 & (3a), (3c),  (3aa), (3ac), (3cc), (3d), (3h), (5a); (3c'),  \\
                4 & (3f) \\
                8 & (8) \\
                $\mu$ & (4a) diagonal, (6a) diagonal ; (4a') , (6a')\\
                \hline
              \end{tabular}
            \end{center}
  \caption{Possible values for the lifting multiplicities of all 38 local tangency types, following the notation of~\autoref{fig:classificationLocalTangencies}. Trivalent types and $4$-valent ones are separated by `;'. The multiplicities on the last row depend on the quantity $\mu$ from~\eqref{eq:mu}, whose value is either 1 o 2. For types (2), (4) (5) and (6), we use the labels `a' and `b' to distinguish between tangencies of multiplicity two and four, respectively.\label{tab:LiftingMultiplicities}}
          \end{table}

\begin{remark}\label{rm:4a6a4ap6ap}
  As was highlighted in~\cite[Table 1]{CLMR25Lifting}, the lifting multiplicities of tangency types (4a), (6a), (4a') and (6a') are not fixed. Their value depends both on whether the tangency is diagonal or not, as well as on the location of the complementing tangency points on $\Lambda$.

  When the type is either (4a) or (6a) and of  vertical or horizontal nature, its lifting multiplicity equals 1 by our findings from~\cite[Propositions 4.3 and 4.5]{CLMR25Lifting}. On the contrary,~\cite[Propositions 4.3 and 7.2]{CLMR25Lifting} confirm that in the remaining cases, the local lifting multiplicity  equals the quantity $\mu$ seen in~\autoref{tab:LiftingMultiplicities}, which we know define.

 We let  $P$ be a tangency on $\Lambda$ of type (4a')/(6a') or (4a)/(6a) and of diagonal nature. Since we  know that  $P$ is contained on a edge of $\Gamma$ of slope $\pm 1$, we can consider the two halfspaces  $H^{\pm}$ of $\RR^2$ determined by the line through $P$ spanned by this edge. We let $P'$ and $P''$ be the tangencies between $\Lambda$ and $\Gamma$ complementing $P$ and define
\begin{equation}\label{eq:mu}
  \mu:=\max\{|H^+\cap \{P',P''\}|,|H^-\cap \{P',P''\}|\}.
\end{equation}
Thus,  $\mu=1$ if $P'$ and $P''$ lie on different halfspaces. Otherwise, we have $\mu=2$.
\end{remark}
  
\begin{remark}\label{rm:tropicalGenericity}
  We should note that the tropical modification techniques used to determine the values seen \autoref{tab:LiftingMultiplicities} for almost all tangency types arising from a non-discrete connected component of $\Lambda\cap \Gamma$ have no requirements on $\Gamma$ beyond its smoothness.
  The only exceptions correspond to types (3f) and (8), known as the star-shape and tree-shape intersections. In these two cases, the techniques require the lengths of the edges in the corresponding connected component of $\Lambda\cap \Gamma$ be mildly generic. We refer to~\cite[Remark 2.10]{CLMR25Lifting} for the precise conditions.

  However,~\autoref{thm:countingTritangentClassesAndLifts} combined with the explicit computations of the lifting partitions for those classes containing members with these tangency types that we carry  in~\autoref{sec:proof-lift-part-1d-or-less}  (see~\autoref{thm:partitionDim0} and~\autoref{pr:3fOr3a3c}), confirm that the values for the exceptional types listed in the table are valid even when these mild genericity conditions are not satisfied by $\Gamma$. This phenomenon is analogous to the one observed by Geiger in the case of tropical bitangent lines to smooth plane quartics~\cite{gei:25}.
   \end{remark}

\autoref{thm:totalLifting} confirms the first roadblock that needs to be overcome before proving~\autoref{thm:partitionThm}: we must determine the largest connected subset of each class where the product formula for the lifting multiplicity of its member holds. These sets will be denoted by $\bclassns$. Its construction will be done in two-steps, which we carry out in 
        {Sections}~\ref{sec:polyhedra-bounded-complex} and~\ref{sec:non-special-bounded}. It is  dependent on the  genericity of $\sextic$ relative $\Gamma$. As byproduct of the construction, we will confirm that  both sets are path-connected. This will allow us to fix a base point in each set and search for liftable members by prescribed paths. The determination of our initial points will be accomplished in~\autoref{sec:comb-bound-compl}, whereas the choice of trajectories is the subject of our efforts in~\autoref{sec:technicalLemmasDim1-2}.

              \begin{figure}[t]
                \includegraphics[scale=0.85]{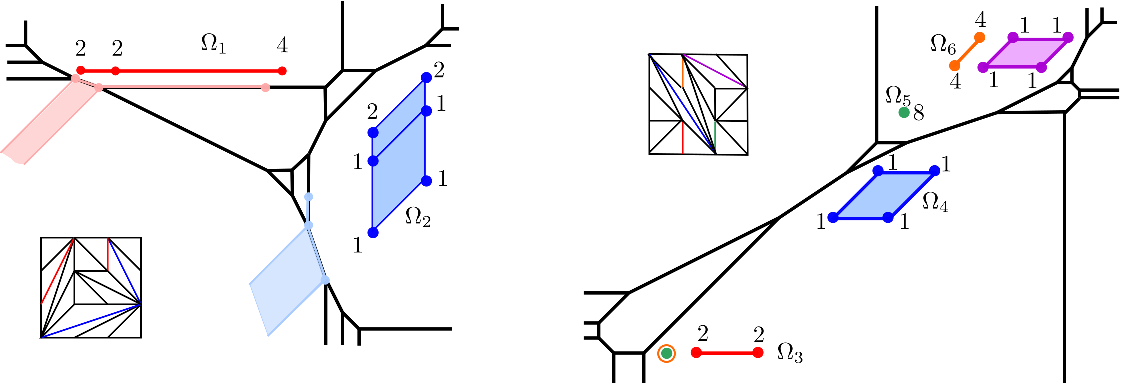}
                \caption{Examples of tritangent classes (with multiplicities) realizing all six lifting partitions from~\autoref{thm:partitionThm}. \label{fig:examplesLift}}
              \end{figure}

              \smallskip
              
              We end the present section by providing examples realizing the lifting partitions predicted  by~\autoref{thm:partitionThm}. For each of them, the corresponding subsets $\bclassns$ mentioned above agree with the union of the bounded components of each $\tclass$.

              \begin{example}\label{ex:realizableLifts}  Let $\Gamma$ and $\Gamma'$ be the two tropical curves seen on the left and right of~\autoref{fig:examplesLift}, respectively. For each of them, we depict a subset of their 15 tritangent classes.
                Members for each of these classes are trivalent, and have a slope one edge. Thus, we can record these classes as pairs of colored subsets in the plane marking the location of the vertices for each of these members. Lighter colored regions indicate the location of lower vertices. Labels indicate lifting multiplicities. 

                  The left picture presents two  tritangent classes to $\Gamma$, labeled $\tclass_1$ and $\tclass_2$, with lifting partitions $(0,2,1,0)$ and  $(4,2,0,0)$, respectively. In turn, the right picture shows four tritangent classes to $\Gamma'$, denoted by $\tclass_3, \ldots, \tclass_6$, which can be recovered from the six colored polyhedra seen in the figure as follows. First, the top left parallelogram (unlabeled and colored in purple), records the position of the top vertex of members of the classes $\tclass_3$ and $\tclass_4$. It can be paired with either the  segment (in red) or the bottom parallelogram (in blue), to yield these two classes. Their  lifting partitions are  $(0,4,0,0)$ and $(8,0,0,0)$, respectively.

                Similarly, the (green) vertex  and the slope one segment (in orange) record the location of top-vertices of members of the classes  $\tclass_5$ and $\tclass_6$, respectively, whose lifting partitions equal  $(0,0,0,1)$ and $(0,0,2,0)$. The location of the lower vertex for all members in these two classes is unique, and corresponds to the unlabeled vertex featured at the bottom, and colored in green and orange.
              \end{example}

\section{Local moves and combinatorial structures on tritangent classes}
\label{sec:local-moves}
In order to prove the lifting partition theorem for tritangent classes to a given smooth $(3,3)$-curve $\Gamma$ in $\TPr^1\times \TPr^1$, we must first understand their structure. This will be achieved by describing their local behavior. We remark that all results in this section are purely combinatorial: they  rely solely on the duality between $\Gamma$ and unimodular triangulations of the standard square of side length three. The metric structure on $\Gamma$ and the genericity conditions from \autoref{def:fgenericRelToGamma} play no role.

Throughout, we let $\Lambda$ be a tropical $(1,1)$-curve in $\TPr^1\times \TPr^1$ tangent to $\Gamma$. We let $P$ be the corresponding tangency, i.e., the tuple of tangency points on the component of $\Lambda\cap \Gamma$ with even stable intersection multiplicity corresponding to the tangency condition.
Any move on $\Lambda$ will shift the point   $P$.
In what follows we discuss how to to perturb the point $(v_0,\lenl)$  in $\RR^3$ corresponding  to $\Lambda$ while mantaining the tangency condition at the translation of $P$.

We restrict such movements to a bounded neighborhood of the input point in $\RR^3$, and  let $\Xi$ be the region of $\RR^3$ recording such local moves. As we will see throughout the various results presented in this section,  the nature of $\Xi$ depends on three different conditions, namely the valency of $\Lambda$, the transversality (or not) of the tangency  $P$, and whether or not $P$ contains a vertex of $\Gamma$. An additional role is played by the relative position of each vertex of $\Lambda$ with respect to $\Gamma$.

The next proposition describes $\Xi$ for each tangency point $P$.

\begin{proposition}\label{pr:localMovesTritangents} Let $\Lambda$ be a tropical tritangent $(1,1)$-curve to $\Gamma$ and let $P$ be one of the corresponding tangencies. The relative position of $P$ within $\Lambda$ and $\Gamma$ restricts how to move the point $(v_0,\lenl)$ encoding $\Lambda$ so that the corresponding translation of $P$ remains a tangency. \autoref{fig:localMoves} describes these moves.
\end{proposition}

\begin{proof} We follow the same strategy as in the proof of~\cite[Lemma 3.2]{CM20}, proceeding by a case-by-case analysis according to the nature of the tangency type of $P$.  Up to $\Dn{4}$-symmetry, there are 38 cases to consider, as shown in~\autoref{fig:classificationLocalTangencies}. Nonetheless, for several of them the combinatorial structure of their local moves agree, leading to only 15 possibilities which we depict in~\autoref{fig:localMoves}. 

  Three remarks are in order. First, we are only concerned with mantaining the tangency condition of the translation of $P$ even though the translated remaining tangencies between $\Lambda$ and $\Gamma$ may cease to be. Second, we always assume that $\lenl\geq 0$, and require our local moves preserve its sign if $\lenl>0$. As we mention earlier,  we only consider bounded local moves, even though in some cases, we can continue translating the input point $(v_0,\lenl)$ along an infinite direction. Such a situation arises, for example, for types (1a) and (1b). Indeed, for the $\Dn{4}$-representatives of this tangency type seen in~\autoref{fig:classificationLocalTangencies}, we can move $v_0$ away from $v_1$, while fixing both $v_1 = v_0 + \lenl (1,1)$ and $P$. 

  \begin{figure}[t]
  \includegraphics[scale=0.36]{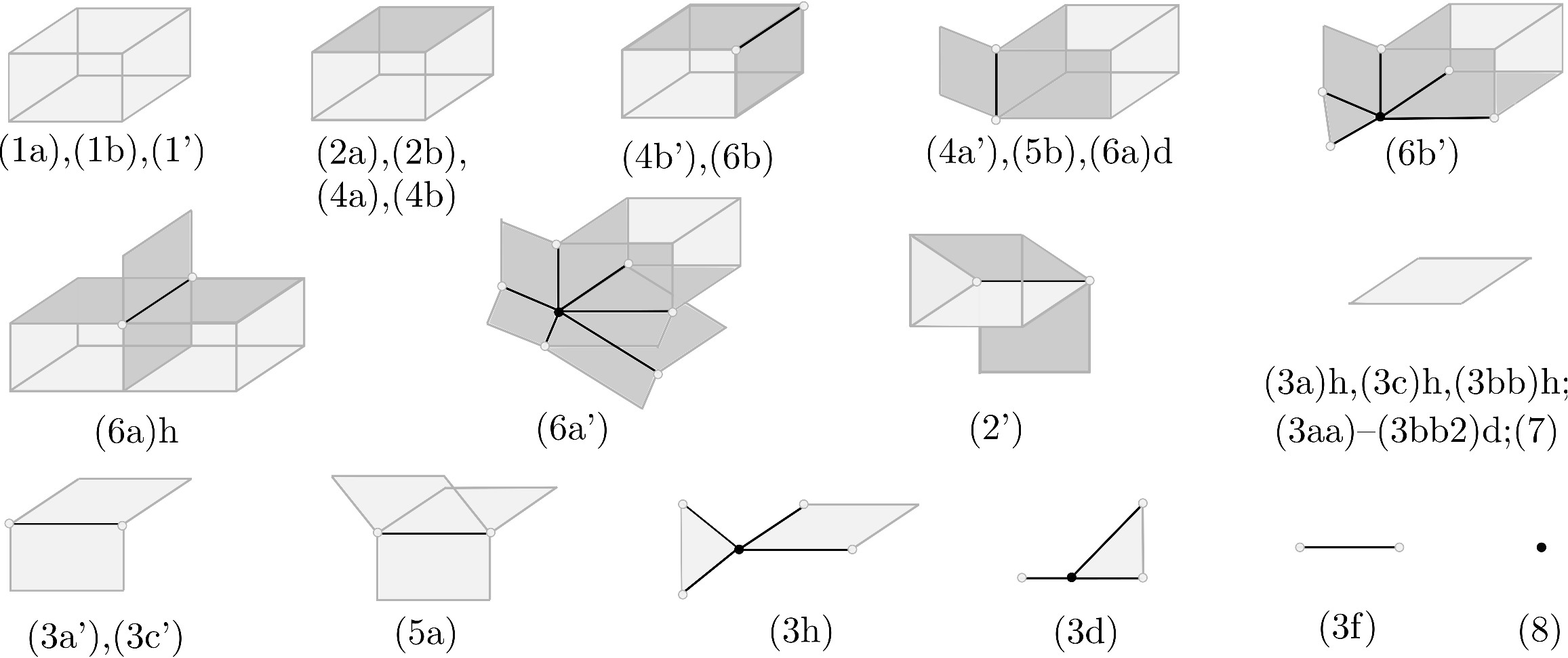}
  \caption{Local moves for all 38 local tangencies depicted in~\autoref{fig:classificationLocalTangencies}, ordered by dimension.  Each local move  consists of a union of relative open polyhedra in $\RR^3$. 
    In particular, only  black filled dots, black edges and dark gray parallelograms are part of these local moves. The point $(v_0,\lenl)$ representing the  curve $\Lambda$ with the given local tangency is the center of mass of the lowest dimensional cell in each local move.\label{fig:localMoves}}
\end{figure}

          {Lemmas}~\ref{lm:OpenCellsTrivalent} and~\ref{lm:OpenCellsFour-valent} below discuss the case when $P$ is  not a vertex of $\Lambda$ and its type is not (3d), depending on the valency of $\Lambda$. In turn,~\autoref{lem:Type4And4pCells} characterizes local moves for tangency types (4a), (4b), (4a') and (4b').

  To finish, we must determine $\Xi$ for the remaining eight  local tangency types. 
 The possible displacements of both vertices of $\Lambda$ for each case can be seen in   \autoref{fig:PositionVerticesMixedDimension}. The roman labels on each picture show how to combine  the areas describing the local movement of each individual vertex of $\Lambda$ to obtain a valid curve $\Lambda'$ tangent to $\Gamma$ at the translation of $P$.
In what follows, we briefly explain how we obtain our findings. Notice that in all these cases, each $v_i$ is either outside $\Gamma$ or it is a vertex of this curve. We treat both situations separately.

  First, assume  one of the vertices of $\Lambda$, say $v_i$, is not in $\Gamma$. Then, $\Lambda$ is trivalent and $P$ corresponds to a local tangency type (5a) or (6a) when $i=0$, and either (5b) or (6b) for $i=1$. In principle, for these four types, we can move $v_i$ within an open neighborhood $U$. For the first three cases, we may choose $U$ to be the join of two open parallelograms along a common slope one open edge  with center of mass $v_i$ in their closure. However, 
  for type (6b), the remaining vertex of $\Lambda$  further restricts the movement of $v_1$. More precisely, the set $U$ becomes the intersection of an open parallelogram with center of mass $v_1$ and one of the two closed halfspaces determined by the unique slope one line though $v_1$. This situation is similar to the one occurring for type (2) tangencies, which is treated in \autoref{lm:OpenCellsTrivalent}.  \autoref{fig:PositionVerticesMixedDimension} shows the compatible movement on the remaining vertex of $\Lambda$ for each of the regions described above.

  Finally, we discuss the case when all vertices of $\Lambda$ are in the connected component of $\Gamma\cap \Lambda$ containing $P$. To determine the displacements for each vertex $v_i$, we must look for valid locations to move $v_i$ within each face of a complete fan centered at $v_i$. In the trivalent case, the rays of this fan are obtained from the union of those in $\Star_{\Gamma}(v_i)$ and the negatives of those in $\Star_{\Lambda}(v_i)$. In turn, if $\Lambda$ is $4$-valent, we replace the latter by the eight rays of the $B_2$ arrangement. In both cases, a non-transverse tangency along a ray of $\Star_{\Gamma}(v_i)$ further restricts the movement to a line. This happens for the vertex $v_0$ in type (3d).

  If the tangencies are proper, the valid locations to displace $v_i$ within a 2-dimensional cell of such a fan can be further restricted to a polygon (either a parallelogram or a triangle). We choose this option to ensure that $\overline{\Xi}$ is a join of polyhedra. The set $\Xi$ is obtained by removing appropriate faces from each polytope in $\overline{\Xi}$, namely, those not containing $(v_0,\lenl)$.
  
  In all but one case, namely type  (6b'), the corresponding polyhedra are either parallelograms or parallelepipeds. In turn, for type (6b') the set $\overline{\Xi}$ is  obtained by joining a parallelepiped and a triangle to a parallelogram with vertex $(v_0,0)$ and edge directions $(-1,0,0)$ and $(-1,-1,1)$, along these two edges.
\end{proof}

\begin{figure}
  \includegraphics[scale=0.38]{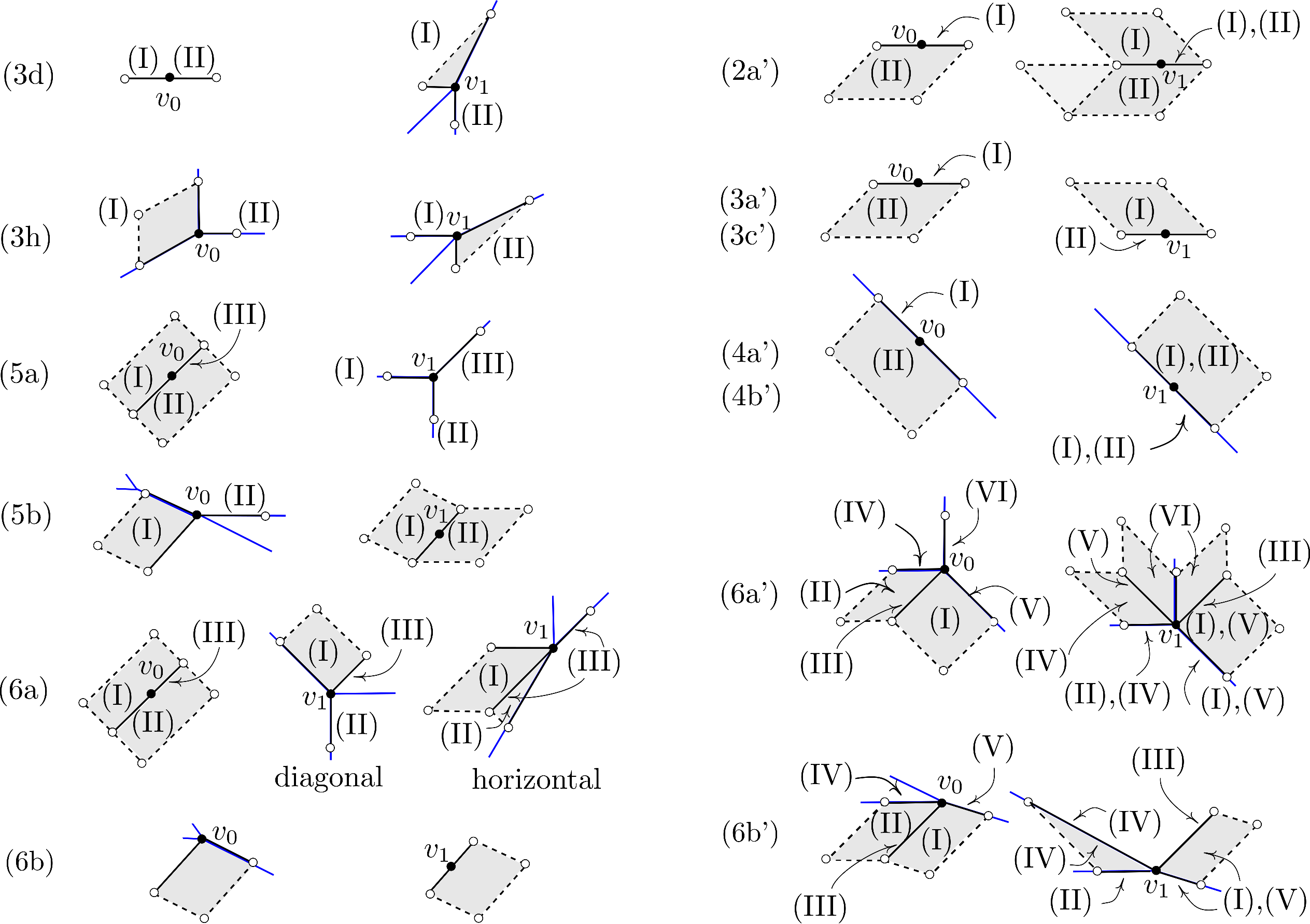}
  \caption{Possible displacements for the vertices $v_0$ and $v_1$ of $\Lambda$ for local moves of mixed-dimension. Relevant edges of $\Gamma$ are indicated in blue.\label{fig:PositionVerticesMixedDimension}}
\end{figure}

The next technical lemmas are used to simplify the proof of~\autoref{pr:localMovesTritangents}. For both statements, we let $(v_0,\lenl)$ be a point in $\RR^3$ corresponding to a tropical $(1,1)$-curve $\Lambda$ tangent to $\Gamma$ at a point $P$. Furthermore, we assume $P$ is not a vertex of $\Lambda$. We let ${\Upsilon}_P$ be the connected component of $\Lambda\cap \Gamma$ containing $P$, and define the following quantity:
\begin{equation}\label{eq:kdim}
  k_0:=\text{number of non-transverse tangency points in } {\Upsilon}_P \text{ that are not vertices of }\Gamma.
  \end{equation}

The first lemma covers  local tangencies types  that are not (3d), namely, (1a)/(1b), (2a)/(2b), (3f), (7) and (8), the  horizontal tangencies of type (3a), (3c) and (3bb), as well as the diagonal tangencies (3aa) through (3bb2) listed in~\autoref{fig:classificationLocalTangencies}.
\begin{lemma}\label{lm:OpenCellsTrivalent}
  Let  $(v_0,\lenl), P$,  $\Lambda$ and $k_0$ be as above with $\lenl>0$. Assume the local tangency type of $P$ is not (3d). Then, the local moves of $(v_0,\lenl)$  agree with  a $(3-k_0)$-dimensional parallelepiped 
  in $\RR^3$ with either all or all but one of its facets removed.
\end{lemma}

\begin{proof} 
We determine $\Xi$ by analyzing the possible perturbations of $v_0$ and $\lenl$ separately. The non-transverse tangencies in ${\Upsilon}_P$ lying in the relative interior of edges of $\Gamma$ impose restrictions on the movement of $v_0$. The number of such tangencies is precisely  the quantity $k_0$. 

  If $k_0 = 0$, then we are in the presence of a type (1) or (2) tangency. In the first case, we can move $v_0$ and $\lenl$ independently in small open neighborhoods of $\RR^2$ and $\RR_{>0}$, respectively. More precisely, we can find $\varepsilon, \varepsilon', \varepsilon''>0$ small enough to ensure that the set $\Xi$ is the interior of the parallelepiped with center of mass $(v_0,\lenl)$, whose edges have directions $(1,1,0), (-1,1,0)$ and $(0,0,1)$, of lengths $2\varepsilon, 2\varepsilon'$ and $2\varepsilon''$, respectively. 

  Similarly, for a type (2) tangency, the movements of $v_0$ and $\lenl$ will also be independent of each other. To describe $\Xi$ we may assume, upon $\Dn{4}$-symmetry, that $v_0$ is adjacent to the leg or edge of $\Lambda$ containing $P$.We let $e$ be its direction. While for any point $(v_0',\lenl') \in \Xi$, the parameter $\lenl'$ can take any value within an open segment centered at $\lenl$ of length $2\varepsilon''$, the point $v_0'$ will lie in a parallelogram with all but one edge removed. This edge will have center of mass $v_0$ and  direction $e$. Thus, the set $\Xi$ becomes a parallelepiped with all but one facet removed. We only preserve the relative interior of the facet with center of mass $(v_0,\lenl)$.

  In turn, if $k_0=1$, then $\Upsilon_P$ contains an edge  with a non-transverse tangency in its relative interior. Without loss of generality, we assume $v_0$ is adjacent to the edge or leg of $\Lambda$ containing $P$. 
  We let $e$ be the direction of this edge and $P_0$ be the corresponding non-transverse tangency in $\Upsilon_P$. This condition forces $v_0$ to move within the open segment with endpoints $v_0\pm \varepsilon e$ for some $\varepsilon>0$ small enough. The length parameter $\lenl'$ can vary independently within an interval $(\lenl -\varepsilon', \lenl+\varepsilon')$ for some positive quantity $\varepsilon'\ll 1$.
Thus, $\Xi$ becomes the open parallelogram in $\RR^2$ with center of mass $(v_0,\lenl)$ and edges with directions $(e,0)$ and $(0,0,1)$, of lengths $2\varepsilon$ and $2\varepsilon''$, respectively. 

Next, assume $k_0=2$. Our assumptions and~\autoref{thm:classificationRealizableLocalTangencies} confirm that we are in the presence of a type (3f) tangency. Next, we describe $\Xi$ for a $\Dn{4}$-representative with $v_0\in \Upsilon_P$. Since the directions of the two edges of $\Upsilon_P$ carrying the non-transverse tangency are linearly independent, the vertex $v_0$ must remain in  place if the tangency conditions were to be preserve. However, the length parameter is free to take any value in an open interval of $\RR_{>0}$ centered at $\lenl$. Thus, the set $\Xi$ is an open segment in $\RR^3$ with center of mass $(v_0,\lenl)$.

Finally, if $k_0=3$, the same classification result ensures our tangency is of type (8). The three possible distributions of points in $\Lambda$ (up to symmetry) are seen at the bottom of~\cite[Figure 4]{CLMR25Lifting}. In this case, the directions of the edges or legs of $\Lambda$ carrying the three non-transverse tangencies prevent $v_0$ and $\lenl$ from moving. We obtain a $0$-dimensional local movement, as claimed.
\end{proof}


Our next lemma determines the local moves for tangency types (1'), (2'), (3a') and (3c'):
\begin{lemma}\label{lm:OpenCellsFour-valent}    Let  $(v_0,\lenl), P$,  $\Lambda$ and $\Upsilon_P$ be as above.  Assume that $\lenl=0$ and that $P$ is not a vertex of $\Gamma$. Then, the local moves  of $(v_0,0)$ agree with either
  \begin{enumerate}[(i)]
  \item   the interior of a parallelepiped (for type (1')),
  \item a join of two parallelograms with three facets removed, glued along a common edge (for types (3a') and (3c')), or
    \item a join of a parallelogram and a triangular prism with all facets removed except for one (a parallelogram), glued along a common edge (for type (2')).
    \end{enumerate}
    In all cases, the point $(v_0,0)$ is the center of mass of the lowest dimensional component of the movement.
\end{lemma}

\begin{proof} As usual,  we let $\Xi$ be the bounded region determined by the local moves of $(v_0,0)$. Our assumptions of $v_0, P$ and $\Lambda$ combined with~\autoref{thm:classificationRealizableLocalTangencies} ensure that $P$ is of types (1'), (2'), (3a') or (3c'). In particular, $\Upsilon_P$ contains a single tangency point, namely $P$.  We let $e$ be the edge or leg of $\Gamma$ carrying the tangency $P$. Exploiting $\Dn{4}$-symmetry, we  suppose that $P$ lies on a horizontal leg of $\Lambda$, picking the representatives for each type depicted in~\autoref{fig:classificationLocalTangencies}. 

 The local movement of $(v_0,0)$ depends on the tangency type of $P$. We treat each case separately, in increasing order of complexity.   Note that since $\Lambda$ is $4$-valent, a $(1,1)$-curve $\Lambda'$ corresponding to a point $(v_0',\lenl')$ in the local movement of $(v_0,0)$ has no combinatorial restrictions. In particular, the length entry $\lenl'$ can have arbitrary sign.

  First, assume $P$ has type (1'), so it lies on the negative horizontal leg of $\Lambda$.  First, we determine the possible points with $\lenl'=0$. Since $P$ lies in the relative interior of the edge $e$, a simple inspection reveals that we can find a bounded open neighborhood $D$ of $v_0$ in $\RR^2$ so that any point $(v_0',0)\in \overline{D}\times \{0\}$ produces a curve $\Lambda'$ with a tangency of type (1') along the edge $e$ of $\Gamma$ and the negative horizontal leg of $\Lambda'$. 
  Furthermore, we may assume $D$ is the interior of a parallelogram with center of mass $v_0$ and edge directions $(1,1)$ and $(-1,1)$. By compactness of $\overline{D}$ we can next pick $\varepsilon>0$ small enough so that any point $(v_0',\lenl')\in \overline{D}\times (-\varepsilon, \varepsilon)$ preserves this tangency type. This proves our claim for type (1').

  Next, assume $P$ is of type (3a') or  (3c'). Let $e$ be the leg, respectively, edge of $\Gamma$ containing $P$. By our choice of representatives, $P$ lies on the positive horizontal leg of $\Lambda$ and $e$ is also horizontal. For any $\Lambda'$ corresponding to a point $(v_0',\lenl')$ in the region $\Xi$,  one of the vertices of $\Lambda'$ (call it $v'$) must lie in a bounded open  segment $D$ of the horizontal line in $\RR^2$ through $P$  containing $v_0$ as its center of mass. Furthermore, the  tangency point obtained by translating $P$ must lie on the positive horizontal leg of $\Lambda'$, so the sign of $\lenl'$ determines the nature of $v'$.  Note that any point in $D\times\{0\}$ will be in $\Xi$. Next,  we show that the local moves for $(v_0',0)$ are obtained by gluing two open parallelograms along the edge $D\times \{0\}$ contained in both of their closure.

 If $\lenl'<0$, then $v'=v_0'$ is restricted to $D$. Thus, we can find $0<\varepsilon\ll 1$ so that the  the open parallelogram $D\times (-\varepsilon, 0)$ is contained in $\Xi$ . On the contrary, if $\lenl'>0$, then $v'=v_1'$ and $\lenl'$ is uniquely determined (in a linear way) by $D$ and $v_0'$. Thus, we need only describe the possible values of $v_0'$.
 A simple inspection reveals that we can find $0<\varepsilon\ll 1$ so any point in the open  parallelogram in $\RR^2$ determined by the edges $D$ and $D-\varepsilon(1,1)$ is a valid location for $v_0'$. Thus, the relative interior of the parallelogram in $\RR^3$ with parallel edges $D\times \{0\}$ and $(D\times \{0\})+\varepsilon(-1,-1,1)$ is a subset of $\Xi$. The gluing of both parallelograms along $D\times \{0\}$ lies in $\Xi$, as desired.

 Finally, assume  $P$ is of type (2'), so it lies in the positive horizontal leg of $\Lambda$. Our choice of $\Dn{4}$-representative places $P$ on the top endpoint of the edge $e$ of $\Gamma$ responsible for this horizontal tangency. In particular, by shrinking $\Xi$ further  we may assume that the translation of $P$ corresponding to any point of $\Xi$ also lies on the edge $e$ of $\Gamma$ and also on the positive horizontal leg of $\Lambda'$.  This forces  $v_0'$ to lie in the closed halfspace determined by the horizontal line of $\RR^2$ through $P$ and containing the remaining endpoint of $e$.

 Since $P$ is the top vertex of $e$, we can find $0<\varepsilon,\varepsilon'\ll 1$ so that $v_0'$ lies in either the open segment $D$  of $\RR^2$ with endpoints $v_0 \pm \varepsilon (1,0)$ or in the open parallelogram $\cP$ determined by  $D$ and $D-\varepsilon' (1,1)$. These regions are labeled (I) and (II), respectively, in the corresponding picture in~\autoref{fig:PositionVerticesMixedDimension}. Note that if $v_0'\in D$ we must have $\lenl'\leq 0$, and furthermore, we can find $0<\varepsilon''\ll 1$ so that $D\times (-\varepsilon'', 0]$ lies in $\Xi$. This determines a parallelogram in $\RR^3$ with all but one edge removed, namely $D\times \{0\}$.

   Finally, assume $v_0'$ belongs to the parallelogram $\cP$ and let $L$ be the horizontal line through $P$. Then, $v_1'$ can be any point in the union of two segments of slopes $\pm 1$ with common endpoint $v_0'$ and the second endpoint on $L$. Note that both segments have the same length. The union of all these segments spans the trapezoid marked with (II) in~\autoref{fig:PositionVerticesMixedDimension}. These conditions determine a triangular prism with all but one facet removed, namely, the one containing all points  $(v_0', \lenl')$ where $v_0'\in \cP$, $\lenl'> 0$ and $v_1'\in D$. 
   We conclude from this that the local moves for type (2') are as described: both regions are glued along a segment, namely $D\times \{0\}$, whose center of mass is $(v_0,0)$. 
\end{proof}

Our next result describes local moves for local tangency types (4a)/(4b) and (4a')/(4b'):
\begin{lemma}\label{lem:Type4And4pCells} Assume $(v_0,\lenl)\in \RR^3$ determines a tropical $(1,1)$-curve $\Lambda$ tangent to $\Gamma$ at a point $P$ of type  (4a), (4b), (4a') or (4b'). Then, the region $\Xi$ describing the local movement of $(v_0,\lenl)$ is either
  \begin{enumerate}[(i)]
  \item a parallelepiped with all by one facet removed (for types (4a) and (4b)),
  \item a parallelepiped with all by two adjacent facets removed (for type (4b'), or
  \item the join of a parallelogram and  a parallelepiped with all by two adjacent facets removed, along this common edge  (for type (4a')).
  \end{enumerate}
  In all cases, $(v_0,\lenl)$ is the center of mass of the lowest dimensional boundary face of  $\Xi$.
\end{lemma}

\begin{proof} The statement for types (4a) and (4b) is analogous to that of types (2a) and (2b), as we now explain. For our choice of representative for $\Lambda$ seen in~\autoref{fig:classificationLocalTangencies}, we know that $P=v_1$. We let $e$ be the edge of $\Gamma$ containing the tangency  $P$. We are allowed to vary $v_0$ within a small  open parallelogram $\cP$ with edge directions  $e$ and $e^{\perp}$, 
  and center of mass $v_0$. The edge lengths can be picked small enough to guarantee  that $\overline{\cP}$ is disjoint from $e$ but the slope one line through any point $v_0'\in \cP$ intersects the relative interior of the edge $e$.

  Setting $v_1'$ as this intersection point determines the maximal value $\mu$ of $\lenl'$ with $(v_0', \lenl')\in {\Xi}$. By compactness of $\overline{\cP}$ we can find $0<\varepsilon\ll 1$ so that $\{v_0'\}\times (\mu-\varepsilon, \mu]\subseteq {\Xi}$ for all $v_0'\in \overline{\cP}$. Since $e$ is one of the edges of $\cP$ and the quantity $\mu$ depends linearly on $v_0'$, it follows that $\Xi$ is a  parallelepiped with all but one facet removed. By construction, the  point $(v_0,\lenl)$ is the center of mass of this facet.

    Next, we assume $P$ has type (4a') or (4b') and let $e$ be the edge of $\Gamma$ containing $P$. The multiplicity of $P$ determines the direction of $e$: it is $(-1,1)$ for type (4a'), and either $(-1,3)$, $(3,-1)$, $(1,3)$ or $(-3,1)$ for type (4b'). Up to $\Dn{4}$-symmetry, we may restrict to the first two options. The region $\Xi$ will be different for each of them, since for type (4a') the length parameter $\lenl'$ in each point $(v_0',\lenl')$ of $\Xi$ can have arbitrary sign, whereas $\lenl'\geq 0$ for type (4b').

 \autoref{fig:PositionVerticesMixedDimension} shows the possible locations of the vertices $v_0'$ and $v_1'$. Each of them lies on opposite halfspaces of an open parallelogram with center of mass $v_0$, edge directions parallel to $e$ and $(1,1)$ sliced along the line spanned by  $e$. The points in $(v_0',\lenl')$ with $\lenl'\geq 0$ determine a parallelepiped $\cQ$ in $\RR^3$ with all but two facets removed. Those remaining correspond to curves $\Lambda'$ where at least one of its vertices lies on the edge $e$. This characterizes $\Xi$ for a type (4b') tangency, since the direction of $e$ prevents any member of $\Xi$ to have an edge with negative slope. 

 The situation is different for type (4a'). Assuming the edge lengths of $\cQ$ are sufficiently small, the local moves for this type are obtained by joining $\cQ$ and a parallelogram in $\RR^3$ along the edge with center of mass $(v_0',0)$. This new component corresponds to curves $\Lambda'$ where both $v_0'$ and $v_1'$ lie on the edge $e$, and $\lenl'\leq 0$. In particular, we may find $0<\varepsilon''\ll 1$ so the distance between $v_1'$ and $v_0'$ is strictly less than $\varepsilon''$. This concludes our proof.
\end{proof}

\autoref{pr:localMovesTritangents} has an important topological implication: each point of a tritangent class admits a local dimension, corresponding to the dimension of the local move. In addition, whenever a local move is bounded, its boundary is determined by  a finite collection of hyperplanes in $\RR^3$ depending only on $\Gamma$. Since tritangent classes are closed subsets of $\RR^3$ by construction, we conclude:
\begin{corollary}\label{cor:TritangentsArePolyhedralComplexes}
Each tritangent class of $\Gamma$ is a connected polyhedral complex in $\RR^3$. The same is true for its inverse image under the tropical Segre embedding from~\eqref{eq:TropSegre}.
\end{corollary}

Our next result characterizes full-dimensionality of cells on each  tritangent class. 

\begin{proposition}\label{pr:3dCells}
  Assume a tritangent class of $\Gamma$ contains a cell $\sC$ of dimension 3. Then, any member $\Lambda$ in the relative interior of $\sC$ can only contain type (1a), (1b) or (1') tangencies.
\end{proposition}

\begin{proof} By construction, each tangency of $\Lambda$  only admits 3-dimensional local moves. This corresponds to the possibilities seen in the  first seven pictures of~\autoref{fig:localMoves}. Inspecting each one of them with the aid of~\autoref{fig:PositionVerticesMixedDimension} we see that the 3-dimensional component of the local moves for these 13 tangency types (up to symmetry) correspond to types (1a), (1b) or (1'). 
  \end{proof}

\begin{remark}\label{rk:dimensionCellRestrictsTangencyTypes}  Notice that the dimension of a cell $\sC$ in a tritangent class of $\Gamma$ restricts the tangency types that can occur in a  member $\Lambda$ of its relative interior. Indeed, $\dim\sC$ is a lower bound for the dimension of the unique open cell of local moves containing $\Lambda$ as its center of mass. From~\autoref{fig:localMoves} we see that if $\dim \sC=2$, then the only local tangencies that  $\Lambda$ can have  are (1a), (1b), (2a), (2b), (4a), (4b), (3a), (3c), (3bb), (3aa) through (3bb2), (7), (1') and (2'). Similarly, (3h), (3d), (8), (6a') and (6b') can only appear if $\dim \sC=0$. The remaining  types are all allowed if $\dim \sC=1$.
\end{remark}

In order to combine local moves associated to various tangency points on a member of a tritangent class  of $\Gamma$ it becomes useful to compare the position of vertices of $\Gamma$ relative to the edge directions of $\Lambda$. The following definition arises naturally from~\cite[Definition 3.4]{CM20}) and it will be central to the constructions in 
{Sections}~\ref{sec:proof-lift-part-3d}, \ref{sec:non-special-bounded}, \ref{sec:technicalLemmasDim1-2} and~\ref{sec:proof-lift-part-2d-or-less}.

\begin{definition}\label{def:orderingRelToEdges}
  Let $v,v'$ be two points of $\RR^2$. We say that $v$ is \emph{smaller} than $v'$ \emph{relative to a weight} $\ww \in \RR^2$, and write $v\prec_\ww v'$, if $v\cdot \ww <v'\cdot \ww$. If $v\cdot \ww = v'\cdot \ww$, we say that the points are $\ww$-aligned and write $v=_{\ww} v'$. When $w=(1,-1), (1,1)$, $(1,0)$ and   $(0,1)$ we write the corresponding partial order as $\preceq_{d}$, $\preceq_{ad}$, $\preceq_{h}$ and $\preceq_v$, respectively.
\end{definition}

\section{The bounded subcomplex of a tritangent class}
\label{sec:polyhedra-bounded-complex}

The characterization of local moves preserving local tangencies carried out in~\autoref{sec:local-moves} confirms that  tritangent classes to $\Gamma$ have the structure of polyhedral complexes in $\RR^3$. As explained in~\autoref{sec:introduction}, our strategy to prove~\autoref{thm:partitionThm} consists of moving within each class to find liftable members and determine their multiplicities. Item (iv) of the genericity assumptions  listed in \autoref{def:fgenericRelToGamma} will allow us to peel off suitable cells in any input class $\tclass$ to simplify our search. The resulting subcomplex, which we denote by $\bclass$, will be connected, bounded and contain all liftable members of $\tclass$. Its construction is the main result of this section.

\begin{remark}\label{rm:finevCoarseStructure}
 As we will see in~\autoref{rem:coarsestStructure}, tritangent classes do not always admit a  coarsest polyhedral complex structure as was the case for bitangent classes to tropical smooth plane quartic curves~\cite{CM20}. For this reason, and to simplify several arguments in the search for liftable members in each class $\tclass$, we endow each of them  with a structure that is compatible with the polyhedral structure on $\RR^2$ induced by the curve $\Gamma$. More precisely, for  each cell $\sC$ of  $\tclass$, the vertices of any given  member in its relative interior lies in a fixed component of the structure on $\RR^2$ induced by $\Gamma$.
\end{remark}

In what follows, we present combinatorial and tangency restrictions on members of a given cell of $\tclass$ based on its boundedness. These statements will determine which cells to remove from $\tclass$ to build the bounded subcomplex $\bclass$.

\begin{proposition}\label{pr:unbounded}
  Let $\sC$ be an unbounded cell in a tritangent class. Then, its relative interior has a trivalent member $\Lambda$ with a vertex in the interior of a connected component of $\RR^2\smallsetminus \Gamma$ dual to the vertex $(0,0)$, $(3,0)$, $(0,3)$ or $(3,3)$ in the Newton subdivision of $\sextic$. Furthermore, in the first two cases, the corresponding vertex is $v_0$, whereas it is $v_1$ for the last two. In addition, the slope of the edge of $\Lambda$ is  unique for each case: it is positive for the first and last cases, and negative for the other two.
\end{proposition}

\begin{proof} Since $\sC$ is an unbounded polyhedron,  there exists  a  non-zero vector $w=(a,b,c)\in \RR^3$  such that $(v_0,\lenl) + \RR_{\geq 0}w \subseteq \sC$ for any $(v_0,\lenl)\in \sC^{\circ}$.  Given $\lambda \in \RR_{\geq 0}$, we write $(v_0', \lenl')= (v_0,\lenl) + \lambda w$ and let $\Lambda$ be the  $(1,1)$-curve  tritangent to $\Gamma$ parameterized by this point. If $\lambda\gg 0$, we know that $\Lambda$ is trivalent since no $4$-valent tritangent can have its vertex in an unbounded chamber of $\RR^2\smallsetminus \Gamma$ or in the interior of a leg of $\Gamma$.   We write $v'_1$ for the remaining vertex of $\Lambda$.

  We analyze three possibilities for $w$ up to $\Dn{4}$-symmetry depending on the signs of $a$ and $b$.
  First, assume $a=b=0$. Then, given $(v_0,\lenl)\in \sC$ we can find $\lambda\gg 0$ so that $v_1' \in (3,3)^{\vee}$ and $\lenl'>0$ if $c> 0$, whereas $v_1'\in (0,3)^{\vee}$ and $\lenl'<0$ if $c< 0$.
 
  Second, suppose $b=0$ and $a>0$. By construction, we can find $\lambda \gg 0$ so that $v_0'$ lies in $\overline{(3,j)^{\vee}}$  for some $j\in \{0,1,2,3\}$. We claim that $j=0$. Indeed, if this were not the case, there would be a multiplicity one intersection point between the negative vertical leg of $\Lambda$ and a positive horizontal leg of $\Gamma$, in violation of the tritangency property.

  To show  $v_0'\in (3,0)^{\vee}$ it suffices to check that  $v_0'$ cannot lie on the lowest positive horizontal leg of $\Gamma$. Since $\lambda\gg 0$, we can assume $v_0'$ is in the relative interior of such leg and $|\lenl'|\gg 0$. Then, $v_0'$ is part of a type (3a) horizontal tangency, forcing $\lenl'> 0$. However, if $\lenl'\gg 0$, the curve $\Lambda$ and one of the remaining two positive horizontal legs of $\Gamma$ would meet transversely at a multiplicity one intersection point, contradicting the tritangency property.

 Finally, we must analyze the case when $a,b\neq 0$. There are two possibilities, up to $\Dn{4}$-symmetry  based on whether or not the signs of $a$ and $b$ agree. If $a,b>0$, then $v_0'\in (3,3)^{\vee}$ and the negative vertical leg of $\Lambda$ intersects a positive horizontal leg of $\Gamma$ at a multiplicity one point, which cannot occur. 
 On the contrary, if $a>0$ and $b<0$,  we have $v_0'\in (3,0)^{\vee}$. Furthermore,  $\lenl'<0$ since in all other cases, the negative horizontal leg of $\Lambda$ would intersect the rightmost negative vertical leg of $\Gamma$ at a point of multiplicity one.
\end{proof}

\begin{remark}\label{rm:firstCellsToToss}
  The genericity conditions on $\sextic$ relative to $\Gamma$ ensure that no tritangent satisfying the conditions of \autoref{pr:unbounded} can lift, even if they belong to bounded cells of $\tclass$. In particular the following collections of cells of $\tclass$ will not contribute to its lifting partition:
\begin{equation}\label{eq:unboundedCellsToToss}\begin{aligned}
    \mathscr{U}_{0}^{+}&:=\{\sC \in \tclass: \exists (v_0,\lenl)\in \sC^{\circ} \text{ with } \lenl > 0 \text{ and } v_0 \in (0,0)^{\vee}\},\\
\mathscr{U}_{1}^{+}&:=\{\sC \in \tclass: \exists (v_0,\lenl)\in \sC^{\circ} \text{ with } \lenl > 0   \text{ and } v_0+\lenl(1,1) \in (3,3)^{\vee} \}, \\
\mathscr{U}_{0}^{-}&:=\{\sC \in \tclass: \exists (v_0,\lenl)\in \sC^{\circ} \text{ with } \lenl > 0 \text{ and } v_0 \in (3,0)^{\vee}\},\\
\mathscr{U}_{1}^{-}&:=\{\sC \in \tclass: \exists (v_0,\lenl)\in \sC^{\circ} \text{ with } \lenl > 0 \text{ and }  v_0-\lenl(-1,1) \in (0,3)^{\vee} \}.
\end{aligned}
\end{equation}
Our assumptions on the polyhedral structure of $\tclass$ stated in~\autoref{rm:finevCoarseStructure} guarantee that in the definitions of these four sets, the existential quantifiers can be replaced with universal ones.
\end{remark}

Our next statement characterizes tritangents complementing those from~\autoref{pr:unbounded} that admit an unbounded displacement within the corresponding tritangent class:

\begin{lemma}\label{lm:unboundedtranslations} Let $\Lambda$ be the curve  tritangent to $\Gamma$ corresponding to a point $(v_0,\lenl)$ in a positive-dimensional tritangent class $\tclass$ of $\Gamma$. Assume that $\lenl>0$, $v_1\in \Gamma$ and let $P$ be a tangency point in the same connected component of $\Lambda\cap \Gamma$ as $v_1$. If the set $(v_0,\lenl) + \RR_{\geq 0} (0,0,1)$ lies in $\tclass$, then: 
              \begin{enumerate}[(i)]
              \item The edge in the Newton subdivision of $\sextic$ dual to that of $\Gamma$ responsible for the tangency at $P$ has either endpoints $(2,3)$ and $(3,2)$ or $(2,2)$ and $(3,3)$.
              \item The tangency type of $P$ is either  (5a), (4a), (6a) or non-transverse along the edge of $\Lambda$.
                \item If  $P$ is a non-transverse diagonal tangency, we can continuously move $v_1$ in the  (-1,-1) direction while fixing the remaining tangency points. The movement will stop at a local tangency of type (4a'), (5a), (6a) horizontal or vertical.
              \end{enumerate}
            \end{lemma}

            \begin{proof} The result follows by analyzing the tangency types of members of  a tritangent class that can occur when moving freely in  the $(0,0,1)$ direction. The proof of~\autoref{pr:unbounded} confirms that by fixing $v_0$ and moving $v_1$ away from $v_0$ we can find a trivalent member $\Lambda'$ of $\tclass$ with a slope one edge and  whose  top vertex lies in $(3,3)^{\vee}$. Starting from $\Lambda'$, we displace this vertex in the $(-1,-1)$ direction until we reach $\Gamma$. By construction, all these intermediate curves will belong to $\tclass$.

              Once we reach $\Gamma$, the corresponding top vertex will lie in a connected component of $\Lambda\cap\Gamma$ of even stable intersection multiplicity. The bidegree of $\Gamma$ restricts the possible local tangency types to (5a), (3ab), (3cb), (4a) or (6a) diagonal. Claims (i) and (ii) follows by construction, whereas (iii) is a consequence of (i) since the remaining tangency points must lie outside the positive horizontal and vertical legs of $\Lambda$.
            \end{proof}

            \begin{remark}\label{rm:secondCellsToToss} The genericity conditions on $\sextic$ ensure that relative interiors of cells of  $\tclass$ containing a member satisfying the conditions of~\autoref{lm:unboundedtranslations}(iii) will not contribute to the lifting partition. The collection recording such cells is obtained as the union of the following four sets:
              \begin{equation}\label{eq:secondSetToRemove}
                \begin{aligned}
                  \tR_{0}^{+}:=&\{\sC\in \tclass: \exists(v_0,\lenl) \in \sC^{\circ} \text{with }  \lenl >0, v_0^{\vee} \supseteq \{(1,0), (0,1)\}, 
                  (i, 2-i)\notin v_0^{\vee} \text{ for }i=0,1,2\},\\
                  \tR_{0}^{-}:=&\{\sC\in \tclass: \exists(v_0,\lenl) \in \sC^{\circ} \text{with }  \lenl <0, v_0^{\vee} \supseteq \{(2,0), (3,1)\},
(1+i, i)\notin v_0^{\vee} \text{ for }i=0,1,2\},\\
                  \tR_{1}^{+}:=&\{\sC\in \tclass: \exists(v_0,\lenl) \in \sC^{\circ} \text{with }  \lenl >0, v_1 = v_0+\lenl (1,1) \text{ satisfies } v_1^{\vee} \supseteq \{(2,3), (3,2)\} \text{ and}\\
                  & \qquad (i+i, 3-i)\notin v_1^{\vee} \text{ for }i=0,1,2\}\},\\
                  \tR_{1}^{-}:=&\{\sC\in \tclass: \exists(v_0,\lenl) \in \sC^{\circ} \text{with }  \lenl <0 , v_1 = v_0 - \lenl (-1,1) \text{ satisfies } v_1^{\vee}\supseteq \{(0,2), (1,3)\} \text{ and} \\ & \qquad (1+i, i)\notin v_1^{\vee} \text{ for }i=0,1,2\}.
                \end{aligned}
              \end{equation}
              As with the sets defined in~\eqref{eq:unboundedCellsToToss},~\autoref{rm:finevCoarseStructure} ensures that all existential quantifiers featuring in the definitions of the above sets can be replaced by universal ones.
            \end{remark}

          {Remarks}~\ref{rm:firstCellsToToss} and~\ref{rm:secondCellsToToss} allow us to build a important subcomplex containing all liftable members of the input tritangent class. Here is the precise definition:

\begin{definition}\label{def:boundedComplex}
  Let $\tclass$ be a tritangent class on $\Gamma$ and endow it with a
  polyhedral complex structure compatible with the one induced on
  $\RR^2$ by $\Gamma$. We let $\bclass$ be the following collection of
  cells obtained from~\eqref{eq:unboundedCellsToToss}
  and~\eqref{eq:secondSetToRemove}:
  \begin{equation}\label{eq:bclass}
\bclass :=    \tclass \smallsetminus  \bigcup_{i=0}^1(\mathscr{U}_{i}^{+} \cup \mathscr{U}_{i}^{-}\cup \tR_{i}^{+}\cup \tR_{i}^{-}).
  \end{equation}
  We endow $\bclass$ with the induced polyhedral structure and call it the \emph{bounded subcomplex} of $\tclass$.  
\end{definition}

\noindent 
Note that, as it occurs with $\tclass$,  the complex $\bclass$ need not be pure-dimensional. Thus, we set
\begin{equation}\label{eq:dimbclass}
\dim \bclass:=
\max\{\dim \sC: \sC  \text{ a cell of } \bclass\}.
\end{equation}
\noindent
            
The following result is a direct consequence of the construction of $\bclass$:

\begin{corollary}\label{cor:bcomplexDeformation} The support of the subcomplex $\bclass$ is a strong deformation retract of $|\tclass|$. In particular, it is connected when viewed as a subset of $\RR^3$ with the induced subspace topology.
            \end{corollary}

\begin{proof}
  We establish the statement in a two-step process. First, we use the collections in~\eqref{eq:unboundedCellsToToss} to build an intermediate complex  \[\tclass':=\tclass\smallsetminus \bigcup_{i=0}^1(\mathscr{U}_{i}^{+} \cup \mathscr{U}_{i}^{-})
  \]
 and a strong deformation retract $\rho\colon |\tclass| \to |\tclass'|$ onto its support. In turn, the four sets from~\eqref{eq:secondSetToRemove} determine a  second deformation retract $\rho'\colon |\tclass'| \to |\bclass|$
  onto  $|\bclass|$. The composition $\rho'\circ \rho\colon \tclass \to \bclass$ will yield the desired map.

  To simplify notation, we identify each point $(v_0,\lenl)$ of a cell in $\tclass$  with the tuple $(v_0,v_1)$ recording the coordinates of the  vertices of the corresponding tritangent.
  To determine the map $\rho$ outside $|\tclass'|$, it suffices to specify the image of any point in the relative interior of a cell $\sC$ in $\mathscr{U}_{0}^{+}\cup \mathscr{U}_{1}^{+} \cup \mathscr{U}_{0}^{-}\cup \mathscr{U}_{1}^{-}$. Exploiting $\Dn{4}$-symmetry we may assume $\sC \in \mathscr{U}_{0}^{+}\cup \mathscr{U}_{1}^{+}$. Given $(v_0,v_1)\in \sC^{\circ}$, we use the slope one line $L$ through $v_0$ and $v_1$ to assign $\rho(v_0,v_1)= (v_0',v_1')$, where
  \[
  \{v_0'\}:=\begin{cases}
  \{v_0\} & \text{ if } v_0\notin (0,0)^{\vee},\\
  L\cap \overline{(0,0)^{\vee}} & \text{ otherwise},
  \end{cases}\quad \text{ and } \quad
   \{v_1'\}=\begin{cases}
  \{v_1\} & \text{ if } v_0\notin (3,3)^{\vee},\\
  L\cap \overline{(3,3)^{\vee}} & \text{ otherwise}.
  \end{cases}
  \]
  In practice, the point $(v_0',v_1')$ is obtained from $(v_0,v_1)$ by fixing any vertex outside the chambers $(0,0)^{\vee}$ and $(3,3)^{\vee}$ and otherwise moving them towards each other until we meet the boundary of the corresponding chamber. Thus, we know that $(v_0',v_1') \in \tclass$ and that the map $\rho$ is continuous. \autoref{rm:finevCoarseStructure} and the definition of the sets $\mathscr{U}_{0}^{+}$ and $\mathscr{U}_{1}^{+}$ ensure that $\rho(v_0',v_1')\in |\tclass'|$.

  To conclude, we build the map $\rho'\colon |\tclass'|\to |\bclass|$. As before, it suffices to determine the image of a point in the relative interior of any of the cells in~\eqref{eq:secondSetToRemove}.  Exploiting the action of $\Dn{4}$ it is enough to define $\rho'$ on the relative interior of any cell $\sC$ of $\tR_{0}^{+}\cup \tR_{1}^{+}$. 
  As with the definition of $\rho$, for each $(v_0,v_1)\in \sC^{\circ}$ we  set $\rho'(v_0,v_1) = (v_0'',v_1'')$ for two suitable points $v_0'',v_1''$ on the line $L$. We need only analyze two cases, namely $\sC\in \tR_{0}^{+}\cap \tR_{1}^{+}$ or $\sC\in \tR_{0}^{+}\smallsetminus \tR_{1}^{+}$.

  First, assume $\sC\in \tR_{0}^{+}\cap \tR_{1}^{+}$.   Since $v_1^{\vee}\supseteq \{(1,0),(0,1)\}$ and $v_1^{\vee}\supseteq \{(2,3),(3,2)\}$ by~\autoref{rm:secondCellsToToss}, we set $v_0''$ and $v_1''$ to be the top and bottom vertices of the slope one edges of $\Gamma$ containing $v_0$ and $v_1$, respectively. The definition of both sets $\tR_i^+$ for $i=0,1$ ensures that $v_0''\neq v_0$ and $v_1''\neq v_1$ and that the point corresponding to $(v_0'', v_1'')$ lies in the relative interior of a cell in $\bclass$.

  In practice, to define $\rho$, we move the vertices of $\Lambda$ towards each other until the tangencies of the components of $\Gamma \cap \Lambda$ containing $v_0$ and $v_1$ are of type (5a) or (6a), either horizontal or vertical and assigning the resulting tritangent as the image of $\Lambda$ under $\rho$. This confirms the continuity of $\rho$.

  Finally, we treat the case when  $\sC\in \tR_{0}^{+}\smallsetminus \tR_{1}^{+}$, so $v_0\in \Gamma$. The definition of $v_0''$ and $v_1''$ depends on whether or not $v_1$ is in the connected component of $\Gamma \cap  \Lambda$ containing $v_0$. If it is not, we set $v_1''=v_1$ and let $v_0''$ be the top vertex of the slope one edge of $\Gamma$ containing $v_0$.  This assignment is continuous: it corresponds to fixing the vertex $v_1$ of $\Lambda$ and moving $v_0$ towards it until we get a type (5a) or (6a) tangency. The resulting tritangent lies in $|\bclass|$.

  On the contrary, if $v_0$ and $v_1$ are in the same connected component of $\Lambda\cap \Gamma$, the bidegree of $\Gamma$ ensures that $v_1$ is not a vertex of $\Gamma$. We set $v_0''=v_1''=v_1$. This choice corresponds to fixing the vertex $v_1$ of $\Lambda$ and moving $v_0$ towards it until we get a $4$-valent tritangent $\Lambda''$. Notice that $\Lambda''$ has a type (4a') tangency in the interior of a slope one edge of $\Gamma$. Therefore, this point cannot be in the relative interior of any cell in the sets $\mathscr{U}_i^{\pm}$ or $\tR_{i}^{\pm}$ for $i=0,1$. Thus, $(v_0'',v_1'')$  belongs to $|\bclass|$.
\end{proof}

            Our last two statements in this section characterize the possible valencies of general members in relative interiors of cells of $\bclass$ in terms of its dimension. 

            \begin{lemma}\label{lm:trivalentsAreEnough}
We let $\sC$ be a maximal cell of $\bclass$. If $\sC$ is positive-dimensional, then a general member in its relative interior is always trivalent.
            \end{lemma}

            \begin{proof} We let $\Lambda$ be a general member of $\sC$. We argue by contradiction, and assume that $\Lambda$ is $4$-valent.
The bidegree of $\Gamma$ restricts the tangency types that can occur on $\Lambda$ precisely because its unique vertex must lie in $\Gamma$. In particular,  $\Lambda$ must contains a  local tangency of one of the following five types: (3a'), (4a'), (6a'), (4b') or (6b').

              We let $P$ be the  tangency point in the same connected component of $\Lambda\cap \Gamma$ as the vertex of $\Lambda$. ~\autoref{fig:trivalentsAreEnough} shows five of the six possible locations of the tangency points of  $\Lambda$ and how $\Lambda$ can be deformed to a member $\Lambda'$ while fixing all but the tangency at $P$, which turns into the tangency $P'$. In all cases, \autoref{pr:unbounded} and~\autoref{lm:unboundedtranslations} confirm that $\Lambda'$ lies in a new cell of $\bclass$  containing $\sC$ in its boundary. This contradicts the maximality of $\sC$.

              The only case that is absent from the figure corresponds to a type (4a') tangency at the vertex $\Lambda$ along a slope -1 edge of $\Gamma$ (called $e$), with the remaining two tangency points lying on the negative horizontal and vertical legs of $\Lambda$.  Since $\bclass$ is not $0$-dimensional, one of these tangencies cannot be of type (3c'). Therefore, we can deform $\Lambda$ into a trivalent tritangent $\Lambda'$ with an edge of slope -1, a type (3aa) diagonal tangency along the edge $e$  and a type (1a) tangency on one of its negative legs. Since $v_0,v_1\in \Gamma$ and each vertex is adjacent to a leg containing a tangency in its relative interior, the cell containing $\Lambda'$ lies in $\bclass$. Again, we reach a contradiction since $\sC$ is a face of this cell.
              \end{proof}

            \begin{lemma}\label{lm:4valentSingleton}
              Suppose that $\dim \bclass =0$ and that its unique member $\Lambda$ is $4$-valent. Then, $\Lambda$ carries a type (4a') tangency. In addition, the remaining two tangency points are of type (3c'), both lying on the same halfspace determined by the  span of the edge of $\Gamma$ containing the unique vertex of $\Lambda$.
            \end{lemma}

            \begin{proof} The bidegree of $\Gamma$ confirms that, up to symmetry, there are at most six possible configurations of tangency points on $\Lambda$. Five of them are depicted in~\autoref{fig:trivalentsAreEnough}, together with non-local moves that make $\Lambda$ into a trivalent member of $\tclass$.  In all these cases, $\bclass$ is positive-dimensional by~\autoref{pr:unbounded}.

              The sixth option  corresponds to a type (4a') tangency at the vertex of $\Lambda$ along a slope -1 edge of $\Gamma$ (labeled $e$) and  two extra tangencies on the negative horizontal and vertical legs of $\Lambda$, respectively. The bidegree of $\Gamma$ forces the dual edge $e^{\vee}$ to have vertices $(2,2)$ and $(3,3)$. Unless these two tangencies are of type (3c') we will be able to deform $\Lambda$ into a trivalent member of $\bclass$ that has two tangency points on its negative legs (either by having $v_1\in e$ and  $v_0\in (2,2)^{\vee}$,  or having both vertices remain in $e$ as part of a type (3aa) tangency). This concludes our proof. 
              \end{proof}

                        \begin{figure}
              \includegraphics[scale=0.35]{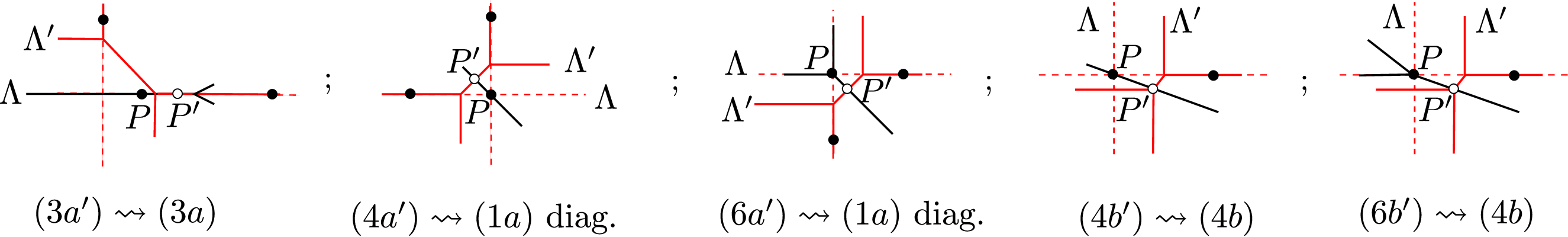}
              \caption{Deformation of sample 4-valent members $\Lambda$ with a give tangency point $P$ in a positive-dimensional cell of $\bclass$ into a trivalent member $\Lambda'$ with corresponding new tangency $P'$.\label{fig:trivalentsAreEnough}}
            \end{figure}

\section{Proof of the lifting partition theorem when $\dim \bclass = 3$}\label{sec:proof-lift-part-3d}

As we discussed in~\autoref{sec:introduction}, our proof strategy for the lifting partition theorem is dimension dependent. In this section, we provide a complete description of the complex $\bclass$ in dimension three and obtain the aforementioned theorem as a direct corollary. The partial orderings $\prec_v$ and $\prec_d$ from~\autoref{def:orderingRelToEdges} play a central role in our construction.

Throughout, we let  $\sC$ be a $3$-dimensional  cell of $\bclass$.
As we saw in~\autoref{pr:3dCells}, any member $\Lambda$ in its interior  can only contain tangencies of type (1a) or (1a').
This fact and~\autoref{lm:trivalentsAreEnough} allows us to restrict our analysis to the case when $\Lambda$ is trivalent, with a slope one edge,  and has three type (1a) tangencies. The bidegree of $\Gamma$ prevents both vertices of $\Lambda$ from being in the same connected component of $\RR^2\smallsetminus \Gamma$, so one of these tangencies will lie on the unique edge of $\Lambda$. We call this point $P$. Combining this observation with~\autoref{pr:unbounded}, we conclude that there are only two cases to analyze, up to $\Dn{4}$-symmetry. 

\autoref{fig:boundedDim3Cases} describes the partial Newton subdivision of $\sextic$ recording the cells responsible for the three tangencies of $\Lambda$. We record the six vertices of $\Gamma$ which are endpoints of the edges carrying the tangencies as $A$, $A'$, $B$, $B'$ and $C$, $C'$. Note that in both cases, the slope of the edge of $\Gamma$ containing $P$ is uniquely determined in each case.

\begin{figure}
  \includegraphics[scale=0.35]{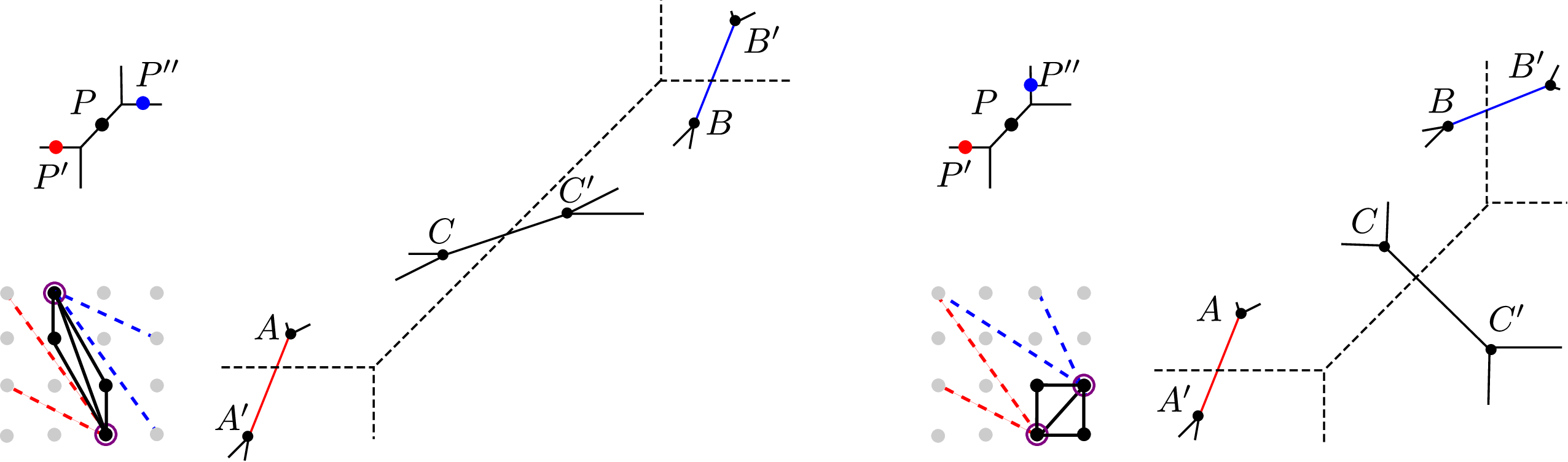}
  \caption{Possible distribution of tangencies between $\Lambda$ and $\Gamma$, the associated partial Newton subdivisions of $\sextic$, and relevant edges of $\Lambda$ and $\Gamma$, where $\Lambda$ is a generic member of a 3-dimensional cell of $\bclass$. We mark the dual vertices corresponding to the chambers of $\RR^2\smallsetminus \Gamma$ containing the vertices of $\Lambda$. The tangency points, the edges of $\Gamma$ containing them and their possible dual edges are color-coded accordingly.\label{fig:boundedDim3Cases}}
\end{figure}

Our next result characterizes lifting partitions for  tritangents classes with $\dim \bclass =3$. The proof technique is similar to that of~\cite[Lemma 4.6]{CM20}.

\begin{theorem}\label{thm:partitionDim3}
  If $\dim \bclass = 3$, then the lifting partition of $\tclass$ equals $(8,0,0,0)$.
\end{theorem}

\begin{proof} We distinguish two cases, depending on the location of the tangency points on $\Lambda$, or equivalent, on the slope of the edge of $\Gamma$ carrying the tangency $P$.  
  {Lemmas}~\ref{lm:case1Dim3} and~\ref{lm:case2Dim3} below confirm the statement by discussing the two representatives depicted in~\autoref{fig:boundedDim3Cases}.
\end{proof}

\begin{lemma}\label{lm:case1Dim3}
  Assume that $\bclass$ is 3-dimensional, and let $\Lambda$ be a general member in the interior of a maximal cell of $\bclass$ that carries a tangency along an edge of  $\Gamma$ of slope $1/3$. Then, the complex $\bclass$ is supported on a polytope  in $\RR^3$. Moreover, $\tclass$ has lifting partition $(8,0,0,0)$ and all liftable members lie on the boundary of this polytope.
\end{lemma}

\begin{proof} 
  Up to $\Dn{4}$-symmetry we may assume that the two horizontal legs of $\Lambda$ contain tangency points, as seen on the left of~\autoref{fig:boundedDim3Cases}. Note that the vertices $v_0$ and $v_1$ of $\Lambda$ belong to the connected component of $\RR^2\smallsetminus \Gamma$ dual to  $(2,0)$ and $(1,3)$, respectively.

  As indicated in~\autoref{fig:parallelogramOnlyHDim3}, we draw six lines through the endpoints of the three edges of $\Gamma$ carrying the tangencies: four horizontal lines through $A, A', B$ and $B'$, and  two slope one lines through $C$ and $C'$. We let $\cP_0$ and $\cP_1$ be the two parallelograms determines by them, with the convention that each  vertex of $\Lambda$ satisfies $v_i\in \cP_i$ for $i=0,1$. 
  
  The partial Newton subdivision of $\sextic$ seen on the left of ~\autoref{fig:boundedDim3Cases} yields particular orderings among four of these vertices with respect to $\prec_v$ and $\prec_d$, namely
  \begin{equation}\label{eq:orderingDim3}
    A\prec_d C\prec_d C'\prec_d B \quad \text{ and } \quad A\prec_v C\prec_v C'\prec_v B.
  \end{equation}
  As a consequence, we know that the edge $e:=\overline{CC'}$ does not intersect either of the parallelograms $\cP_0$, $\cP_1$.  In particular, placing $v_0$ along the top and right edges of $\cP_0$, and $v_1$ on the bottom and right edges of $\cP_1$ (in a compatible way) yield elements of  $\bclass$.
  Unlike  $e$, the edges $e':=\overline{AA'}$ and $e'':=\overline{BB'}$ can intersect the parallelograms $\cP_0$ and $\cP_1$.

  In what follows we show that the support of $\bclass$ agrees with a polytope in $\RR^3$ built out of $\cP_0$, $\cP_1$, and the chambers $(0,2)^{\vee}$ and $(1,3)^{\vee}$.  
  By construction, points of $\cP_0\cap \overline{(2,0)^{\vee}}$ and $\cP_1\cap \overline{(1,3)^\vee}$ that are diagonally aligned will produce elements in the support of  $\tclass$. The precise shape of both polygons  depends on the position of the edges $e'$ and $e''$ relative to $\cP_0$ and $\cP_1$, respectively.

\begin{figure}
  \includegraphics[scale=0.25]{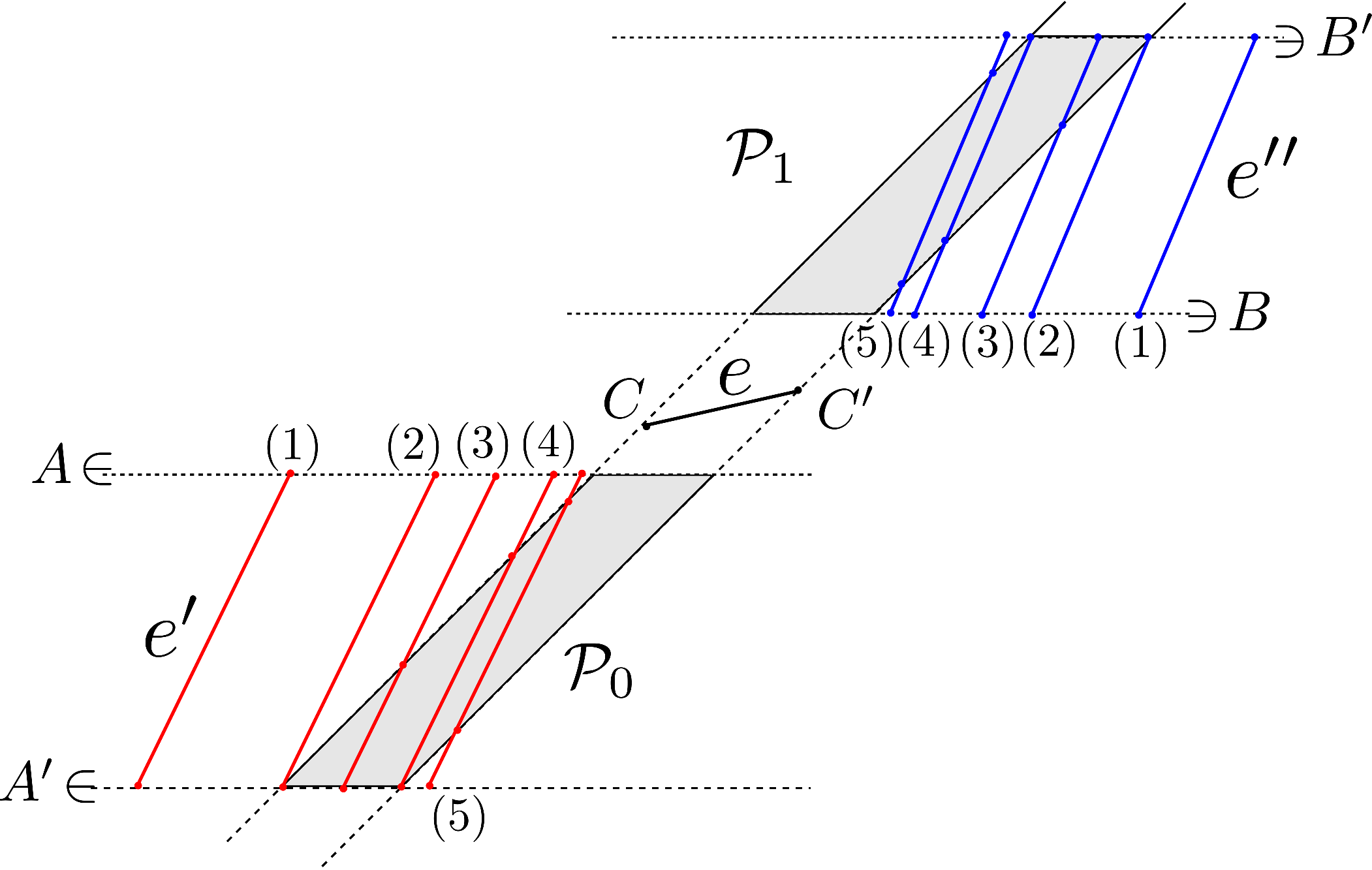}
  \caption{Positions of the (red) edge $e':=\overline{AA'}$, the (blue) edge $e'':=\overline{BB'}$ and the (black) edge $e:=\overline{CC'}$, relative to the parallelograms  $\cP_0$ and  $\cP_1$. 
    \label{fig:parallelogramOnlyHDim3}}
\end{figure}

    By symmetry, it is enough to discuss the situation for the pair $(e',\cP_0)$. We will show that the vertex  $v_0$ of any member of $\bclass$ must lie in $\cP_0 \cap \overline{(2,0)^{\vee}}$.   
There are five cases to consider, and their viability depends on the slope of each edge. All options will be valid if $e'$ has slope $2$, whereas (1) can only occur if its slope is $2/3$. 

  \begin{itemize}
  \item[(1)] \emph{The edge $e'$ avoids $\cP_0$}. In particular, this means that the bottom and left edges of $\cP_0$ are valid locations for $v_0$ for any member of $\bclass$. Indeed, the vertex $v_1$ can be placed on the bottom edge of $\cP_1$. The local moves of $v_0$ within $\tclass$ are constrained  to $\cP_0$ by $e'$ and $e$. 
  \item[(2)] \emph{The point $A'$ is a vertex of $\cP_0$ and $A'=_d C$.} If so, $A'$ is a valid position for the vertex $v_0$ of a member of $\tclass$, since  $v_1$ can be placed on the bottom edge of $\cP_1$. The tangency type at $v_0$ will be (6a) horizontal, and we will be able to move $v_0$ in the direction $(-1,-1)$, leading to an unbounded 2-dimensional movement within $\tclass$. The local moves of  any other point of $\cP_0\cap\overline{(0,2)}^{\vee}$ are restricted to $\cP_0$.
  \item[(3)] \emph{The point $A'$ lies in the relative interior of an edge of $\cP_0$}. As in Case (2), $A'$ is a valid location for $v_0$. Furthermore, the same is true for any  point  in  $\cP_0\cap e'$ since we will be able to place $v_1$ on the bottom edge of $\cP_1$ and produce a member of $\tclass$.  By construction, any valid local move of $v_0$ away from $\cP_0\cap \overline{(0,2)^{\vee}}$ produces elements in cells of $\tclass$ outside $\bclass$. Hence, local moves withing $\bclass$ are constrained to the polygon  $\cP_0\cap \overline{(0,2)^{\vee}}$.
  \item[(4)] \emph{The point $A'$ is a vertex of $\cP_0$ and $C\prec_d A'$}. In this case, $e'\cap  \cP_0$ is a segment containing $A'$. The description of $\cP_0\cap e'$ and the local moves for a member of $\bclass$ with $v_0 \in \cP_0\cap \overline{(0,2)^{\vee}}$ is the  same as in (3).
  \item[(5)] \emph{We have $A'\notin \cP$ and $e'\cap  \cP_0$ is a segment not containing $A'$}. The description is the same as with (3) but the role of $A'$ is now played by the lower endpoint of the segment $e'\cap \cP_0$. 
  \end{itemize}

  From the above analysis and $\Dn{4}$-symmetry we conclude that for any member of $\bclass$ we have $v_0\in \cP_0\cap \overline{(0,2)}^{\vee}$ and $v_1\in \cP_1\cap \overline{(1,3)}^{\vee}$. Furthermore, for each $(v_0,\lenl)\in \bclass$, the value of $\lenl$ lies in a segment, whose endpoints correspond to placing $v_1$ on the bottom edge of $\cP_1$ and either the top edge of $\cP_1$, or the edge $e''\cap \cP_1$ depending on whether or not  $v_0\preceq_d B'$. These constraints are linear in $v_0$, confirming that $\bclass$ is a polytope in $\RR^3$.

It remains to determine which points of $\bclass$ have positive lifting multiplicity. Note that by construction, all members of $\bclass$  contain a diagonal tangency point (labeled $P$)  along $\overline{CC'}$ and its type is either (1a) or (2a). Therefore, the only ones that can lift are those for which $P=C$ or $C'$. These correspond  to tritangents where $v_0$ and $v_1$ lie on the right and left edges of $\cP_0 \cap \overline{(2,0)^\vee}$ and $\cP_1\cap \overline{(1,3)^{\vee}}$, respectively. Whereas the points along the horizontal lines through $A$ and $B$ lead to type (2a) horizontal tangencies, the others can be of type (2a), (4a) or (6a), depending on the five possible positions of $v_0$ and $v_1$ relative to $\cP_0$ and $\cP_1$, respectively, seen in~\autoref{fig:parallelogramOnlyHDim3}. It follows from here that the lifting partition for $\tclass$ is $(8,0,0,0)$. Furthermore, for all liftable members, the vertices $v_0$ and $v_1$ lie on the boundary of $\bclass$.
\end{proof}

  \begin{figure}
  \includegraphics[scale=0.3]{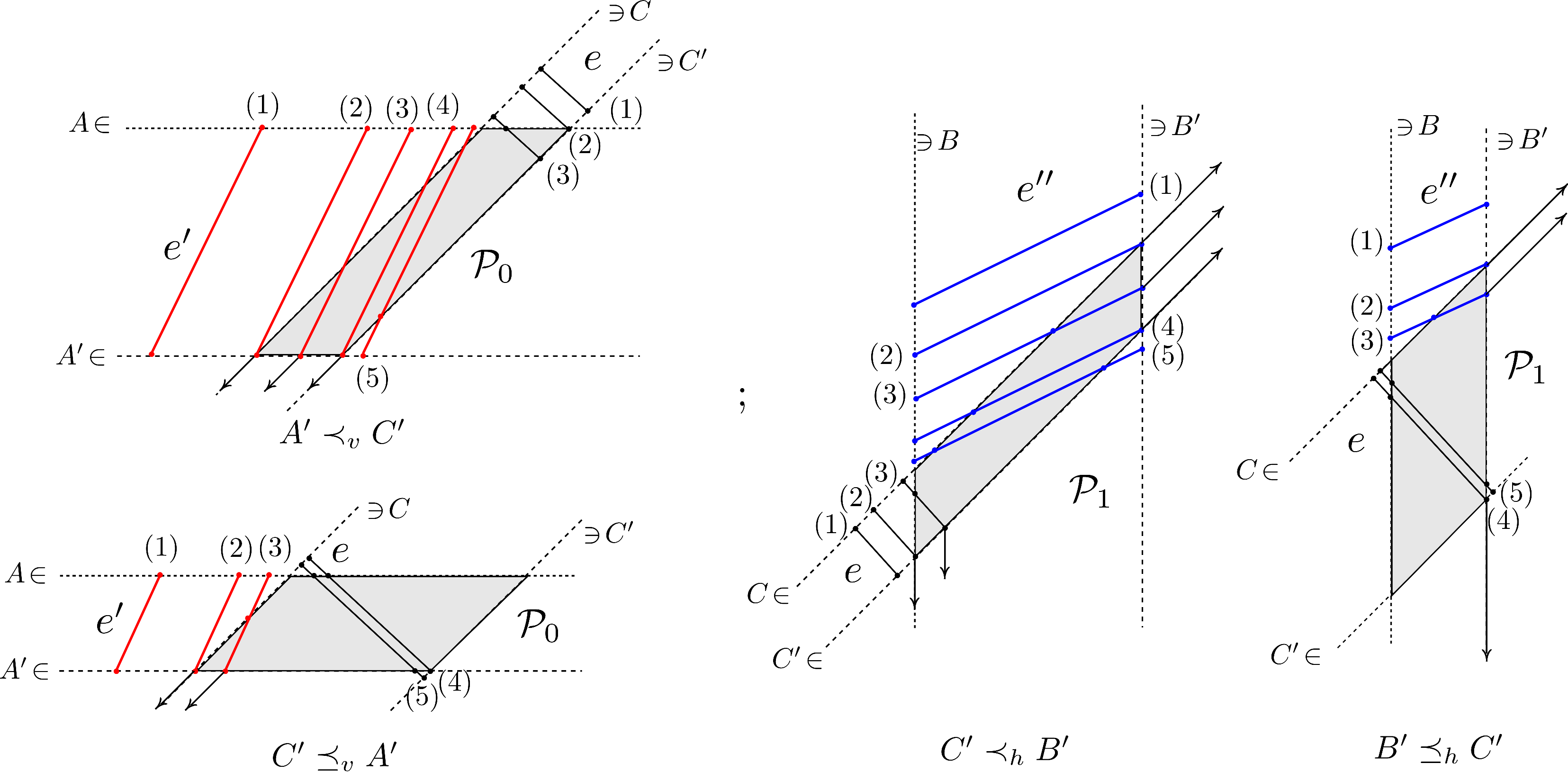}
    \caption{From left to right: location and local moves  of the vertices $v_0$ and $v_1$ for a tritangent $\Lambda$ as in the right of~\autoref{fig:boundedDim3Cases}, and positions of the (red) edge $e':=\overline{AA'}$, the (blue) edge $e'':=\overline{BB'}$ and the (black) edge $e:=\overline{CC'}$, relative to the parallelograms  $\cP_0$ and  $\cP_1$. The latter depends on the relative $\prec_h$-order between $C'$ and $A'$, and the $\prec_v$-order between $C'$ and $B'$.\label{fig:parallelogramHAndVDim3}}
\end{figure}

\begin{lemma}\label{lm:case2Dim3}
 Suppose that $\bclass$ is 3-dimensional and a general member $\Lambda$ on the interior of a top-dimensional cell has a tangency along a slope one edge of $\Gamma$. Then, $\tclass$ has lifting partition $(8,0,0,0)$.  Furthermore, all liftable members correspond to points   on the boundary of a unique $3$-dimensional polytope contained in the support of $\bclass$.
\end{lemma}

\begin{proof}
  By the action of $\Dn{4}$ we may assume that $\Lambda$  has a slope one edge and  tangencies along its negative horizontal and positive vertical legs, as seen in the right of~\autoref{fig:boundedDim3Cases}. We employ the same techniques as in the proof of~\autoref{lm:case1Dim3}, focusing on the major combinatorial differences. The relevant parallelograms and the relative positions of the edges $e, e'$ and $e''$ with respect to them can be seen in~\autoref{fig:parallelogramHAndVDim3}.

  The partial Newton subdivision of $\sextic$ imposed by $\Lambda$ appears on the  the right of~\autoref{fig:boundedDim3Cases}. It determines only some inequalities between $A, B$ and $C$, namely,
  \begin{equation}\label{eq:inequalitiesCase2Dim3}
    A\prec_hC\prec_h B\quad,\quad
    A\prec_vC\prec_v B\quad and \quad
    A, B\prec_d C.
  \end{equation}
  Unlike what happened in~\autoref{lm:case1Dim3}, the edge $e$ can intersect either of the parallelograms $\cP_0$ and $\cP_1$ due to its slope value and  the  lack of information between the relative $\prec_v$ (respectively, $\prec_h$)  order between $C'$ and the points $A,A'$ (respectively, $C'$ and $B,B'$).
    In addition, there will be up to five relative positions for the edges $e$, $e'$ and $e''$ with respect to these parallelograms. They depend on the orders $\prec_h, \prec_v$ indicated in~\autoref{fig:parallelogramHAndVDim3} as well as on the slope of $e'$ and $e''$. If $e'$ has slope $2/3$, then $e'$ only admits position (1) relative to  $\cP_0$. Similarly, if $e''$ has slope $3/2$, only position (1) for this edge with respect to $\cP_1$ can be reached. The figure also includes the unbounded directions in which $v_0$ and $v_1$ can move while remaining in $\tclass$.

The inequalities in~\eqref{eq:inequalitiesCase2Dim3} confirm that any  member of $\bclass$ with a slope one edge and a tangency along $e$ must have its vertices $v_0$ and $v_1$ in the polytopes  $\cP_0 \cap \overline{(2,0)^{\vee}}$ and $\cP_1 \cap \overline{(3,1)^{\vee}}$, respectively.
    The positions of the edge $e$ relative to the parallelograms $\cP_0$ and $\cP_1$  determine which points of these polytopes can be combined to produce a polytope $\cQ\in \RR^3$ whose interior points $(v_0,\lenl)$ yield trivalent members of  $|\bclass|$ with slope one edges. This polytope will satisfy the conditions of the statement. We distinguish two situations.

  First  suppose that the position of $e$ relative to each parallelogram is one of (1) through  (4). Then,   as in the proof of~\autoref{lm:case1Dim3},  each point in $\cP_0 \cap \overline{(2,0)^{\vee}}$ is diagonally aligned with points in a slope one segment (possibly of length zero) on  $\cP_1 \cap \overline{(3,1)^{\vee}}$. By symmetry, the same holds if we reverse the roles of $v_1$ and $v_0$. This determines a $3$-dimensional polytope in $\RR^3$ contained in the support of $\bclass$. We let $\cQ$ be this polytope.

On the contrary, assume $e$ has position (5) with respect to either $\cP_0$ or $\cP_1$. We let $C''$ be the unique minimal point with respect to $\preceq_d$ on the finite set
\[e\cap ((A'+\RR (1,0)) \cup (B'+\RR(0,1))).\]
Points $v_0 \in \cP_0$ and $v_1\in \cP_1$ with either $C''\prec_d v_0$ or $C''\prec_d v_1$ will not admit a diagonally-aligned companion in the other set to produce a member of $\tclass$. Thus, replacing $C'$ with $C''$ and building the corresponding parallelograms restricts the  positions of $e$ relative to either of them to cases (1) through (4). This yields  a $3$-dimensional polytope  in the support of $\bclass$. As before, we set $\cQ$ to be this convex body.

In both situations discussed above, arguments analogous to those used in the proof of~\autoref{lm:case1Dim3} confirm that  the polytope $\cQ$ lies in the support of $\bclass$. Note that if the relative position of $e$ with respect to one of the parallelograms $\cP_0$ or $\cP_1$ is (1), all members corresponding to points in $\cQ$ are trivalent. In all other cases,  the boundary of $\cQ$ can include $4$-valent tritangents. If so, their unique vertex must lie in $e$.

Next, we determine which   members of $\cQ$ produce tritangents $\Lambda'$ that  lift over $\overline{\K}$.
 The information provided by the tangency point $P'$ on the edge $e'$ of $\Gamma$ restricts the location of $v_0$ to segments in the boundary of $\cP_0\cap \overline{(2,0)^{\vee}}$ that are either horizontal  or included in $e'\cup e$. Similarly, the point $P''$ has positive local lifting multiplicity if and only if the point $v_1$ lies on a vertical segment or a subset of $e\cup e''$ in the boundary of $\cP_1\cap \overline{(3,1)^{\vee}}$.

 \begin{figure}[t]
  \includegraphics[scale=0.35]{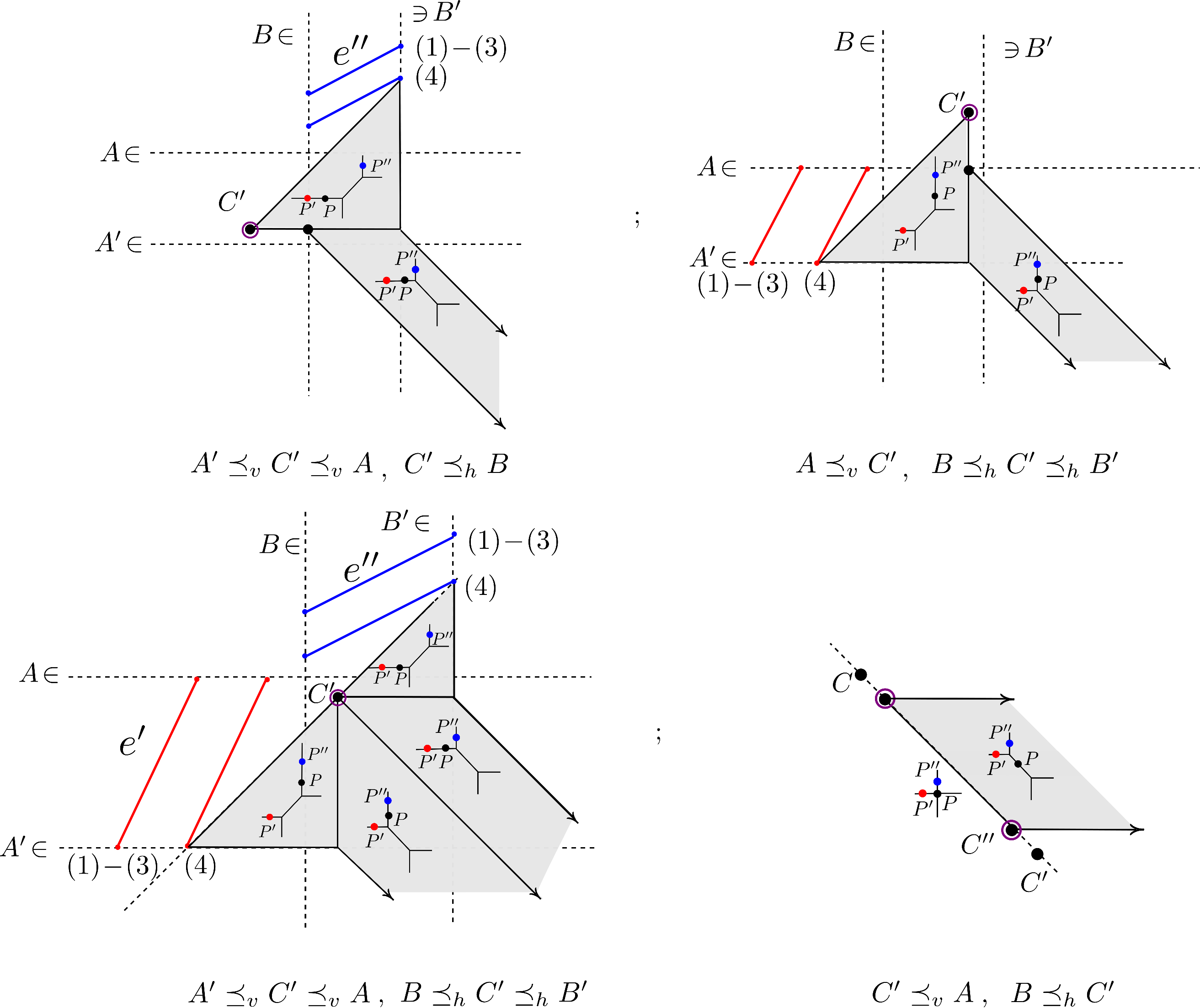}
  \caption{Combinatorial type and location of the tangency points for potential extra members of  $\tclass$ outside the polytope $\cQ$ not featured in~\autoref{fig:parallelogramHAndVDim3}. Their appearance depends on the $\prec_v$ order between $A, C', A$ and the $\prec_h$-order between $B,B', C$ indicated below each set. The point $C''$ is defined in the proof of \autoref{lm:case2Dim3}. 
    \label{fig:extraCellDim3HV}}
\end{figure}

In turn, the tangency point $P$ on the edge $e$ yields a pair  $(\Lambda, P)$  with positive local lifting multiplicity  if and only if its type is (2a), (4a) or (6a) for trivalent members, and either (4a') or (6a') for 4-valent ones. In addition to $C$ and $C'$ there are at most two other possible locations for $P$, namely the intersection points of $e$ and lines through $A$ and $B$ used to define $\cP_0$ and $\cP_1$. 
These options determine four possible locations for both $v_0$ and $v_1$. However, only eight of these 16 combinations produced diagonally aligned vertices. Moreover, four of these will a type (2a) diagonal tangency at $C$.  Thus, $\cQ$ has eight liftable members, all of multiplicity one, in its boundary.

To finish our proof, we must describe the remaining members of $\tclass$ and confirm that no point outside $\cQ$ produces liftable tritangents. These extra tritangents appear in  \autoref{fig:extraCellDim3HV}. Their nature depends on the $\preceq_v-$ and $\preceq_h$-order between $C'$ and the pairs $A, A'$ and $B,B'$, respectively. Their existence is precisely due to the fact that $e$ may intersect the parallelograms  $\cP_0$ and  $\cP_1$. In all cases, these extra members lie in lower dimensional cells of $\tclass$. The genericity assumption on $\sextic$ ensures that only a finite number of members in these extra cells (marked in purple) can potentially lift over $\overline{\K}$. However, all of them are already in $\cQ$. This concludes our proof.
  \end{proof}

            \begin{remark}\label{rem:coarsestStructure} The existence of cells of $\tclass$ outside the polytope $\cQ$ seen in the top row of~\autoref{fig:extraCellDim3HV} confirms that $\tclass$ does not admit a coarsest structure as a polyhedral complex.           \end{remark}

\section{The bounded complex $\bclass$ and its generic members  beyond dimension 3}\label{sec:comb-bound-compl}

In~\autoref{sec:polyhedra-bounded-complex}, we introduced the bounded complex $\bclass$ associated to a tritangent class $\tclass$ of $\Gamma$. It consists of a subset of the collection of bounded cells in $\tclass$ and, furthermore, it is a strong deformation retract of $\tclass$. In this section we study some of its combinatorial properties when its dimension is at most two. Our main objective is to classify the possible distributions of tangency types of a generic member in the relative interior of a top-dimensional cell of $\bclass$. \autoref{thm:dim210} below accomplishes this task, paving the way to prove the lifting partition theorem beyond dimension 3.

The following is the main result of this section. Combined with~\autoref{pr:3dCells}, and 
{Lemmas}~\ref{lm:trivalentsAreEnough} and~\ref{lm:4valentSingleton}, it gives a full characterization of  all general members of maximal dimensional cells in $\bclass$.
            
\begin{theorem}\label{thm:dim210}
  Let $\tclass$ be a tritangent class of $\Gamma$. Assume that $\dim \bclass \leq 2$ and let $\sC$ be a cell of $\bclass$ with  $\dim \sC =   \dim\bclass$. Then, up to symmetry, any generic  member in the relative interior of $\sC$ 
  agrees with one of the curves listed in~\autoref{tab:combClassificationGenericPerDim}. 
\end{theorem}

            \begin{proof}
  Pick a general trivalent member in the relative interior of $\sC$. Up to symmetry, we assume its unique bounded edge has slope one. Since $\sC\in \bclass$, we know that $\Lambda$ has no tangency of type (3ab) or (3cb). Similarly,~\autoref{pr:unbounded} confirms that if a vertex of $\Lambda$ lies outside $\Gamma$, then, one of the two legs of $\Lambda$ adjacent to it must contain a tangency point.

  First, we analyze the possible movements of both vertices of $\Lambda$. There are precisely five options:
  \begin{enumerate}[(i)]
  \item \label{movei} one vertex  moves around an open ball in $\RR^2$ centered around it, but the location of the second vertex is uniquely determined by the position of the first one;
  \item \label{moveii} both vertices move independently along  an open segment, which is forced to have slope one;
  \item \label{moveiii} both vertices move on open segments, but the position of the first one determines the location of the second one;
  \item \label{moveiv} one vertex is fixed, and the other one moves along an open segment of slope one;
       \item \label{movev} both vertices are fixed.
  \end{enumerate}
 Note that the dimension of $\sC$ is fixed for each case. Indeed, the first two can only occur if $\dim \sC=2$, whereas the next two happen if $\dim \sC=1$. Finally, the last option forces $\sC$ to be a point.

             \begin{table}[tb]
              \begin{tabular}{|c|c|c|}
                \hline
                  $\dim\bclass$  & \multicolumn{2}{c|}{Generic members of a cell $\sC$ of $\bclass$ with $\dim \sC=\dim\bclass$} \\
                \hline
                \multirow{4}{*}{2}  &  (i) & (ii) \\
                \cline{2-3} 
&                        \multirow{3}{*}{\includegraphics[scale=0.3]{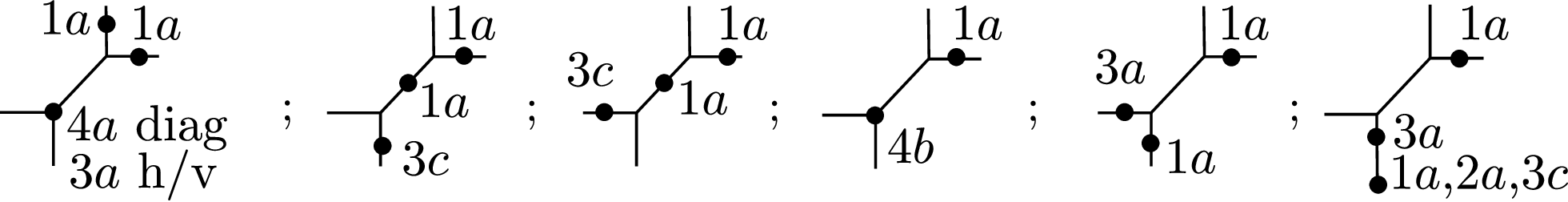}} &    \multirow{3}{*}{\includegraphics[scale=0.3]{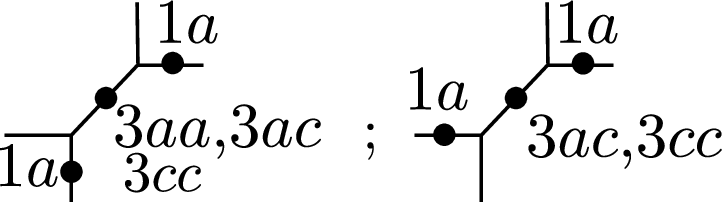}}  \\
                & & \\
                & & \\
                         \hline
                \multirow{7}{*}{1}  &  (iii) & (iv) \\
                \cline{2-3} 
&                        \multirow{5}{*}{\includegraphics[scale=0.3]{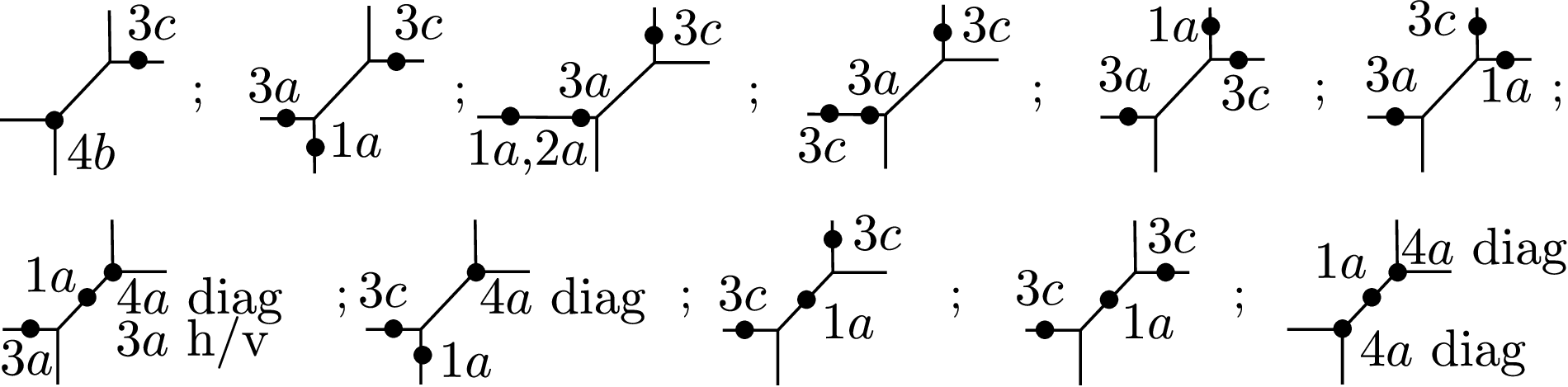}} &    \multirow{5}{*}{\includegraphics[scale=0.3]{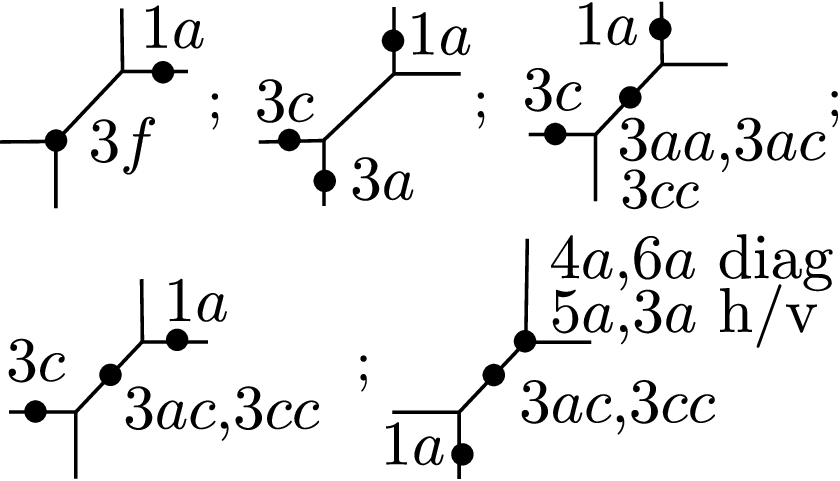}}
                \\
                & & \\
                & & \\
                & & \\
                & & \\
                & & \\
                         \hline
                \multirow{7}{*}{0}  &  \multicolumn{2}{c|}{(v)}  \\
                \cline{2-3} & \multicolumn{2}{c|}{\includegraphics[scale=0.3]{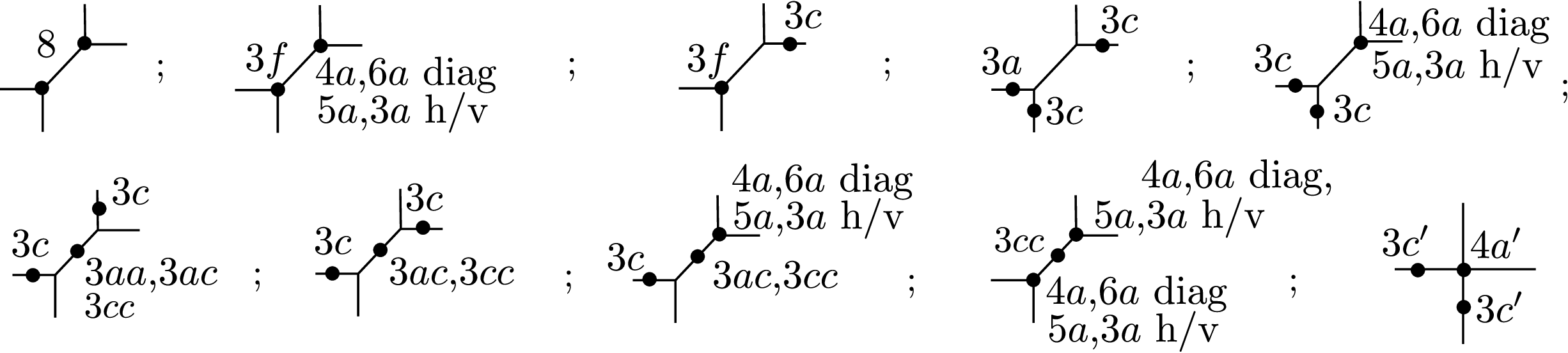}}  \\               
                         \hline
              \end{tabular}
              \caption{Possible generic 
                members that can appear in the interior of a top-dimensional  cell of the bounded complex of a tritangent class depending on the local moves of the vertices of $\Lambda$ discussed in the proof of~\autoref{thm:dim210}.\label{tab:combClassificationGenericPerDim}}
            \end{table}

We will see that each of the possibilities listed above produces one of the sample members in~\autoref{tab:combClassificationGenericPerDim}.
We treat each case separately. Movement labeled \eqref{movei} and ~\eqref{moveiii} are treated in 
{Lemmas}~\ref{lm:movementi} and~\ref{lm:movementiii}.

  Next, assume the local moves of $v_0$ and $v_1$ are as in \eqref{moveii}. It follows for this that any tangency belonging to a leg of $\Lambda$ must be of type (1a). Since $\dim \sC=2$,   \autoref{pr:3dCells} and \autoref{rk:dimensionCellRestrictsTangencyTypes}  combined ensure that $\Lambda$ has a non-transverse diagonal tangency at a point $P'$, of type   (3aa), (3ac), (3cc), (3bb1), (3bb2), or  (7).  Without loss of generality, we choose the representatives for these types seen in~\autoref{fig:classificationLocalTangencies}. 

  We claim that the first three cases suffice, since the last three are not generic. Indeed, 
  if $P'$ has type (7), then moving $v_0$ and $v_1$ towards each other produces a member in $\sC^{\circ}$ with a type (3aa) tangency. Similarly if $P'$ has type (3bb1) or (3bb2), moving $v_0$ away from $v_1$  gives rise to  a member in $\sC^{\circ}$ with a tangency of type (3cc) or (3ac), respectively.

  If $P'$ has type (3aa), (3ac) or (3cc), the condition $\sC\in \bclass$ forces type (1a) tangencies on precisely one leg adjacent to each vertex of $\Lambda$. Up to symmetry, there is one choice for type (3aa) (imposed by the bidegree of $\Gamma$) and two for the remaining types. This produces the last two options seen in the first row of the table, concluding our proof when $\dim \sC=2$.

Next, if the movement corresponds to option \eqref{moveiv}, then one of the vertices cannot move if we want to remain in $\sC$. \autoref{lm:fixAVertex} determines the tangencies responsible for fixing it. Any extra tangency must be transverse since $\dim \sC=1$. The bidegree of $\Gamma$ and the condition $\sC\subset \bclass$ force the latter to be of type  (1a) and adjacent to the movable vertex of $\Lambda$. The possible configurations are found in the (1, (iv)) entry of the table.

Finally, we discuss the case when $\dim \sC=0$, corresponding to the movement labeled \eqref{movev}. After checking that a type (8) tangency fixes $\Lambda$, and that the only companions to a type (3f) that fix both vertices of $\Lambda$ are those seen in the table,~\autoref{lm:fixAVertex} confirms the remaining possible configurations of tangency points indicated in the table. This concludes our proof.
\end{proof}

            Our next lemmas are used in the proof of~\autoref{thm:dim210}. The first two correspond to the movements labeled (i) and (iii). The third one describes tangency types that  fix one of the vertices of the tritangent curve $\Lambda$.

                        \begin{lemma}\label{lm:movementi}               Let $\bclass$ and $\sC$ be as in~\autoref{thm:dim210}. Assume that $\sC$ has dimension two, and its relative interior contains a trivalent member with a vertex moving on an open ball in $\RR^2$. Then, $\sC$ contains one of the trivalent samples appearing in the $(2,(i))$-entry in~\autoref{tab:combClassificationGenericPerDim}.
            \end{lemma}

                        \begin{proof}
                     Let $\Lambda$ be a generic trivalent member of $\sC$.     Without loss of generality, we suppose that the vertex $v_1$ moves around an open ball centered at this point and that the position of $v_0$ is uniquely determined by that of $v_1$. Since $\Lambda\in \sC^{\circ}$ is general, the description of the local moves for $v_1$ force this vertex to to lie outside $\Gamma$. Since $\sC\in \bclass$, it follows that one of the legs adjacent to $v_1$ must contain a tangency point. By symmetry, we assume it is the positive horizontal one. In turn, the local movement of $v_1$ ensures that any such tangency point is of type (1a) or (1b). If the positive vertical leg of $\Lambda$ carries a tangency, then $v_1\in (1,1)^{\vee}$, whereas if this leg contains no tangency, we have $v_1\in (1,3)^{\vee}$.

  If $v_1\in (1,1)^{\vee}$, the bidegree of $\Gamma$ combined with~\autoref{pr:unbounded} forces $v_0\in \Gamma$. Since $\Lambda$ is a generic member of $\sC$ and its unique edge has slope one, it follows from~\autoref{rk:dimensionCellRestrictsTangencyTypes} that the remaining tangency of $\Lambda$ must be one of two types: either (4a) diagonal, or (3a) horizontal or vertical. In addition, it must be contained in the same connected component of $ \Gamma \cap \Lambda $ as $v_0$. This configuration appears as the first element in the first row of~\autoref{tab:combClassificationGenericPerDim}.

  On the contrary, if $v_1\in (1,3)^{\vee}$, the bidegree of $\Gamma$ ensures that $v_0\notin (1,3)^{\vee}$. Thus, we have two options for $v_0$: either $v_0\in \Gamma\cap \overline{(1,3)^{\vee}}$ or $v_0$ lies in the closure of a chamber of $\RR^2\smallsetminus \Gamma$ adjacent to $(1,3)^{\vee}$, forcing $\Lambda$ to have a diagonal tangency of type (1a) or (1b). The last option implies that $\sC\notin \bclass$ by~\autoref{pr:unbounded}, contradicting our assumptions on $\sC$. In turn, if the type (1a) occurs, then the remaining tangency point of $\Lambda$ is of type (3c) and  lies on a leg adjacent to $v_0$. This corresponds to the second and third members seen in the first row of the table.

  Finally, if $v_0\in \Gamma\cap \overline{(1,3)^{\vee}}$, then $\Lambda$ has a tangency point in the same component of $ \Gamma \cap \Lambda$ as $v_0$. Call this point $P$. The dimension of $\sC$ and~\autoref{rm:finevCoarseStructure} discards the possibility of a type (6a) or (5a) tangency at $P$.  In turn, the bidegree of $\Gamma$ forces  $P$ to be either a diagonal or vertical tangency.

  We claim that the first situation cannot occur for dimension reasons. To prove the latter, we argue by contradiction and suppose $v_0$ is a type (4a) diagonal tangency. In such situation, the remaining tangency of $\Lambda$ would lie on its negative vertical  leg and  its type would be (1a), (2a) or (3c). For the first two options, a local move would produce a member in a 3-dimensional cell of $\bclass$ adjacent to $\sC$, contradicting the maximality of $\sC$. Similarly, if its type were (3c), then moving $v_0$ downwards and $v_1$ accordingly would produce a new member in a 2-dimensional cell of $\bclass$ with $\Lambda$ in its boundary, which cannot occur since $\Lambda\in \sC^{\circ}$.

  We conclude from this analysis and~\autoref{pr:unbounded} that $P$ has type (4b) or (3a) (either horizontal or vertical). For this last option, the table shows the possible locations and types of the remaining tangency point, which are determined for bidegree and dimension reasons. 
              \end{proof}

            \begin{lemma}\label{lm:movementiii}
              Let $\bclass$ and $\sC$ be as in~\autoref{thm:dim210}. Assume that $\sC$ has dimension one, and its relative interior contains a trivalent member whose vertices move (not independently), along open segments. Then,  $\sC$ contains one of the trivalent samples listed in the $(1,(iii))$-entry in~\autoref{tab:combClassificationGenericPerDim}.
            \end{lemma}

            \begin{proof}
              Let $\Lambda$ be a generic trivalent member of $\sC$. Suppose that its unique edge has slope one and  that the local moves of both vertices are as in \eqref{moveiii}. For dimension reasons, the movement cannot be along a slope one edge. In particular, $\Lambda$ cannot contain a non-transverse diagonal tangency. Since $\Lambda\in \sC^{\circ}$, $\sC \in \bclass$, tangencies must occur within  one of the two legs adjacent to each  vertex of $\Lambda$. We call them $P_0$ and $P_1$. Furthermore, the dimension restrictions on the local moves of $v_0$ and $v_1$ restrict the tangency types of both $P_0$ and $P_1$ to one of the following: (4b), (3a), (3c), (4a) diagonal or (3bb) horizontal/vertical. 

              In what follows, we confirm that we can ignore the (3bb) tangency option for dimension reasons. \autoref{fig:classificationLocalTangencies} shows two possibilities for the star of $\Gamma$ at $v_0$ in the presence of a type (3bb)   tangency along the negative horizontal leg of $\Lambda$. For each of them, the condition $\sC \in \bclass$ fixes the  triangle $v_0^{\vee}$ in the Newton subdivision of $\sextic$ and the chamber of $\RR^2\smallsetminus \Gamma$ containing $v_1$ in its closure. Indeed,  $v_1$ would  lie in either $\overline{(2,2)^{\vee}}$ or $\overline{(1,3)^{\vee}}$, respectively.

              In the first case, $v_1$ would also belong to a component of $\Gamma \cap \Lambda$ containing a tangency of type (3a) horizontal or vertical, or (4a) diagonal. In the second situation, the horizontal leg of $\Lambda$ adjacent to $v_1$ would have a tangency point, necessarily of type (3c). In both cases, moving $v_0$  in the direction $(1,0)$ (and $v_1$ accordingly)  would turn the (3bb) tangency into a (3c) tangency on the negative horizontal leg of $\Lambda$ and a type (1a) tangency, either on its negative vertical leg or on its unique edge, respectively. Thus, the type (3bb) tangency need not be considered.

In what follows, we discuss the rest of the tangency types listed above. After each option is treated, we discard it for the remaining cases (for both vertices of $\Lambda$). We start with type (4b). The bidegree of $\Gamma$ and the fact that $\sC\in \bclass$ has dimension one, fix the location and type of the complementing tangency on $\Lambda$: it must be of type (3c) and lie on the positive horizontal leg of $\Lambda$. This corresponds to the first configuration seen in the second row of the table.

Next, suppose that $v_0$ lies in a connected component of $\Gamma \cap \Lambda$ with a  (3a) horizontal tangency. We claim that there are six configurations containing such tangency, all listed in the table. First, notice that there are two possibilities regarding the negative vertical leg of $\Lambda$: either it contains no tangencies or one of type (1a). In the latter case, the bidegree forces $v_1$ to belong to the chamber of $\RR^2\smallsetminus \Gamma$ dual to $(1,3)$, leading to a type (3c) tangency along the positive horizontal leg.

In the former case, we have two possibilities, either the negative horizontal leg of $\Lambda$ contains a second tangency, or not. If it does, then this tangency is necessarily of type (1a), (2a) or (3c), and is placed to the left of the type (3a) one containing $v_0$. In particular, the edge $v_0^{\vee}$  in the Newton subdivision of $\sextic$ has endpoints $(3,0)$ and $(3,1)$. Going through the list of three options for tangencies in the star of $\Lambda$ at $v_1$ provided above, we conclude that 
$v_1 \in (3,1)^{\vee}$ and that its type is (3c) vertical.

On the contrary, if our type (3a) tangency containing $v_0$ is the only one on a leg of $\Lambda$ adjacent to $v_0$, it follows that $v_0^{\vee}$ is the edge with endpoints $(1,0)$ and $(1,1)$. Our choice of movement limits any tangency in the interior of the edge of $\Lambda$ to be of type (1a). This fact and  the bidegree of $\Gamma$ restrict the location of $v_1$ to two possibilities: it can lie in $(1,1)^{\vee}$ or in the closure of $(2,2)^{\vee}$. In the first case,  $\Lambda$  has two tangencies on the legs adjacent to $v_1$: one of type (1a) and one of type (3c). For the second case, the edge of $\Lambda$ carries a type (1a) tangency. The remaining tangency will be at $v_1$ and be of type (4a) diagonal  or (3a) along a leg of $\Lambda$ adjacent to $v_1$. 

Third, suppose that $\Lambda$ has a  type (3c) tangency on its negative horizontal leg.  From our previous discussion, combined with symmetry, we may assume that $\Lambda$  only contains tangencies of types (3c), (4a) or (1a). Since $\Lambda\in \sC^{\circ}$, our  description of the movement of $v_0$ along a segment forces $v_0\notin \Gamma$. The bidegree of $\Gamma$ implies that $v_0 \in (2,2)^{\vee}$ or $v_0\in (2,0)^{\vee}$ and ensures the presence of at most one tangency on the negative vertical leg of $\Lambda$, which is necessarily of type (1a). These situations yield three of the last four configurations seen on the  (1, (iii))-entry  in the table, as we now explain.

In the first case, the remaining tangency is necessarily be of type (4a) diagonal. In the second case, we  have a diagonal tangency, either of type (1a) or (4a). If the former occurs along an edge $e$ of $\Gamma$ (with slope $-1$ or $1/3$), the remaining tangency is of type (3c), but its location depends on the slope of $e$. To conclude, we observe that the tangency (4a) case can be discarded, since this tritangent is in the boundary of a 1-dimensional cell of $\bclass$ by~\autoref{rm:finevCoarseStructure}. Indeed, in the presence of such type, the remaining one is of type (3c) along the positive vertical leg, for bidegree reasons. Thus, moving $v_1$ upwards and $v_0$ towards the left produces a type (1a) tangency along a slope -1 edge of $\Gamma$, which is a configuration that we have already encountered.

Finally, if both vertices of $\Lambda$ carry type (4a) diagonal tangencies, the bidegree of $\Gamma$ forces the remaining tangency to be diagonal, and of type (1a), which corresponds to the last tritangent listed on the $(1,(iii))$-entry  in the table. 
     \end{proof}

\begin{lemma}\label{lm:fixAVertex}
  Let $\sC$ be a cell of the complex $\bclass$   and let $\Lambda$ be  a general trivalent member in its relative interior. Assume that $\Lambda$ contains a vertex $v$ that cannot move locally while remaining in $\bclass$.
   Then, one of the following two conditions occur:
  \begin{enumerate}[(i)]
  \item there are two non-transverse tangencies in the interior of two of the rays in $\Star_{\Lambda}(v)$;
  \item 
 $v$ carries a tangency of  type (5a) or (4a)/(6a) diagonal on the boundary of an unbounded chamber of $\RR^2\smallsetminus \Gamma$, complemented by a non-transverse diagonal tangency not containing $v$.
  \end{enumerate}
\end{lemma}

\begin{proof} Without loss of generality, we assume that the unique edge of $\Lambda$ has slope one. 
  Since the vertex $v$ is not a priori determined, we can freely  pick representatives for local tangencies as in~\autoref{fig:classificationLocalTangencies} whenever convenient.

  By construction, configurations from case (i) will always fix the input vertex.   In what follows, we suppose that $\Lambda$ does not satisfy the conditions of case (i). We can discard the situations when $\Gamma \cap \Lambda$ has one component since type (7) does not fix any vertex of $\Lambda$, whereas (8)  satisfies (i). Similarly, \autoref{fig:PositionVerticesMixedDimension} and the bidegree of $\Gamma$ allows  local moves for types (3h) and (3d) within $|\bclass|$  that do not fix either vertex of $\Lambda$. In turn, type (3f) falls into case (i), so we can ignore it. 

  We distinguish two different scenarios, depending on whether or not $\Lambda$ carries a non-transverse diagonal tangency. 
  First, suppose it does, and let $P$ be the corresponding tangency in the interior of the edge of $\Lambda$.  Since case (i) cannot occur and $\sC\in \bclass$, it follows that $v\in \Gamma$. In addition,  the description of local movements for each possible tangency type at $P$ done in~\autoref{pr:localMovesTritangents}  ensures that  $v$ and $P$  lie in different connected components of $\Gamma \cap \Lambda$.
Moreover, a tangency must occur at $v$, and its type is either (5a) or diagonal of type (4a) or (6a). This fact
 combined with  \autoref{lm:unboundedtranslations} confirms that we can only move $v$  away from the other vertex of $\Lambda$, at which point it reaches an unbounded component of $\tclass$. Thus, $v$ lies in either $\overline{(3,3)^{\vee}}$ (if $v=v_1$) or $\overline{(0,0)^{\vee}}$ (if $v=v_0$). This confirms that the  conditions of case (ii) are satisfied.

  Next, assume that $\Lambda$ contains no non-transverse diagonal tangencies. Again, we have two possible situations imposed by the bidegree of $\Gamma$ and fact that  $\sC\in \bclass$: either both vertices of $\Lambda$ lie in the closure of the same chamber, or of two adjacent ones.

  If the latter holds,  we must then have a type (1a) or (2a) diagonal tangency at a point $P'$ in the edge of $\Lambda$, since type (1b) and (2b) tangencies correspond to cells of $\tclass$ outside $\bclass$ by~\autoref{pr:unbounded}. Let $e$ be the edge of $\Gamma$ responsible for this tangency. Its slope can be either -1 or 1/3. We claim that for both choices, the vertices of $\Lambda$ can always move while remaining in $\bclass$, contradicting our assumptions on $\Lambda$. We treat both cases separately.

First, assume that  $e$ has slope $1/3$. Then, the bidegree of $\Gamma$ ensures that $\Lambda$  has tangencies on both its horizontal legs. As a consequence, we can move both vertices in the horizontal direction (with appropriate sign) to turn $P'$ into a type (1a) tangency while preserving the other two tangencies. In particular, no vertex of $\Lambda$ is fixed.

On the contrary, suppose that $e$ has slope -1. Then, since we know from~\autoref{pr:unbounded} that $v_0\notin (0,0)^{\vee}$ and $v_1\notin (3,3)^{\vee}$, the bidegree of $\Gamma$ further restricts the locations of $v_0$ and $v_1$ relative to the Newton subdivision of $\sextic$ (up to symmetry). More precisely, 
there are two possibilities: either  $v_0\in\overline{(0,2)^{\vee}}$ and $v_1\in \overline{(1,3)^{\vee}}$,   or $v_0\in\overline{(1,1)^{\vee}}$ and $v_1\in \overline{(2,2)^{\vee}}$. 
  For the first combination, the remaining two tangency points lie on the negative vertical and positive horizontal leg of $\Lambda$ and their types are (1a), (2a) or (3c). In this situation,  we can move both $v_0$ and $v_1$ in the vertical and horizontal directions (either positive or negative, depending on the tangency type of $P'$) while remaining in $\bclass$. Similarly, for the second location of $v_0$ and $v_1$, the remaining tangency points will lie in the boundary of $(1,1)^{\vee}$ and $(2,2)^{\vee}$. Thus, we can  move both vertices along these two boundaries in a compatible way,   while staying in $\bclass$. Again, we reach a contradiction.

  It remains to treat the case when   $v_0$ and $v_1$ lie in the closure of the same chamber, which we fix to be $(i,j)^{\vee}$. In this situation,  $\Gamma \cap \Lambda$ has either two or three  connected components. We claim that the first situation will not fix either vertex of $\Lambda$. Indeed,  since $\Lambda$ has no non-transverse diagonal tangency, there are only four options for the type of a multiplicity four tangency, namely, (4b), (5b), (6b) or (3bb), all of which involves a vertex of $\Lambda$. By symmetry, we may assume it is $v_0$.

  For the first three types, the bidegree of $\Gamma$ forces the remaining tangency point to lie on the positive horizontal leg of $\Lambda$. Furthermore,  its type would be (1a), (2a), (3c), (4a) or (6a). In all cases, we can move $v_0$ along the slope $-1/3$ edge of $\Gamma$ containing it, and $v_1$ towards or away from $v_0$, while remaining in $\bclass$. In turn, if $v_0$ is part of a type (3bb) tangency (which we may assume to be horizontal), then the choice of $v_0^{\vee}$ determines the location of  the remaining tangency: it will be either in $\overline{(2,2)^{\vee}}$ or on the positive horizontal leg of $\Lambda$. In both situations, we can move $v_0$ horizontally, and shift $v_1$ accordingly, while remaining in $\bclass$.

To conclude, we must treat the case when $\Lambda\cap \Gamma$ has three connected components. 
  Note that the condition  $v_0, v_1\in \overline{(i,j)^{\vee}}$ forces one of the vertices to have two distinct tangencies  in the union of both legs adjacent to it. We remark that these tangencies can lie on the same leg. Up to symmetry, we may assume this occurs for $v_0$. We argue separately for the cases $v_0\notin \Gamma$ and $v_0\in \Gamma$ since the values of $i,j$ are different in each situation.

  If $v_0\notin \Gamma$, then $(i,j)=(2,2)$. The bidegree of $\Gamma$, \autoref{pr:unbounded} and~\autoref{lm:unboundedtranslations} combined ensure both that  $v_1\in \Gamma$,  a  that the legs of $\Lambda$ adjacent to $v_0$ each have a tangency of type (1a), (2a) or (3c). Unless these tangencies are both non-transverse (i.e., of type (3c)), both vertices of $\Lambda$ can move while remaining in $\bclass$. This situation corresponds to case (i), which we did not allow for $\Lambda$.

  Finally, if $v_0\in \Gamma$, then this vertex is part of a tangency, necessarily of types (3a), (4a) horizontal/vertical, (5a) or (4a)/(6a) diagonal, since type (6a) horizontal/vertical cannot occur at $v_0$ for bidegree reasons. For the last three options, symmetry and~\autoref{lm:unboundedtranslations} confirm that $(i,j)=(3,1)$, and the remaining tangencies lie on the negative horizontal and positive vertical legs of $\Lambda$. In addition, their types are (1a), (2a) or (3c). Thus, we can move $v_0$ horizontally and $v_1$ vertically (with appropriate signs) while remaining in $\bclass$. This contradicts the original movement assumptions.

  For the remaining two cases, we let $Q$ be the tangency in the component of $\Gamma \cap \Lambda$ containin $v_0$. We assume it is horizontal.  For type (4a), we have that $(i,j)=(2,2)$, the negative vertical leg contains a tangency point (of type (1a), (2a) or (3c)), and $v_1$ lies in the boundary of the chamber $(2,2)^{\vee}$. Moving $v_0$ downwards and $v_1$ along the boundary of $(2,2)^{\vee}$ produces new tritangents in $\bclass$, which cannot occur.

  To conclude, if $Q$ is of type (3a), we see that  $(i,j)=(1,3)$ or $(3,1)$ depending on whether or not, the negative vertical leg of $\Lambda$ carries a tangency point (called $Q'$) in its relative interior. The bidegree of $\Gamma$ restricts the location of the remaining tangency of $\Lambda$. In the second situation, $\Lambda$ has two tangencies on the negative horizontal leg and one on the positive vertical leg, allowing us to displace $v_0$ to the right and $v_1$ downwards while fixing two of the three tangencies. These movements produce new members in $\bclass$, which we did not allow.

  Similarly, if $(i,j)=(1,3)$, the remaining tangency of $\Lambda$  lies on the positive horizontal leg and we can move both vertices of $\Lambda$ in the horizontal direction unless  $Q'$ has type (3c). The latter cannot happen since it corresponds to case (i) from the statement.  Both options, contradict our earlier assumptions on $\Lambda$. This concludes our proof.
\end{proof}

\section{The non-special bounded complex $\bclassns$ and the refined lifting partition theorem} 
\label{sec:non-special-bounded}

The genericity assumptions on the polynomial $\sextic$ relative to $\Gamma$ discussed in~\autoref{def:fgenericRelToGamma} confirm that certain special distributions of tangency points on $\Lambda$ will not yield liftable tritangents. This section is devoted to studying these members and the local structure of cells containing them. We start with a definition:

 \begin{definition}\label{def:specialTritangents}
We say a curve $\Lambda$ tritangent  to $\Gamma$ is \emph{special} if two of its tangencies lie on the interior of the same leg of $\Lambda$. Similarly, a cell of $\bclass$ is \emph{special} if it contains a special tritangent in its relative interior.
\end{definition}

 As the next result shows, the bidegree of $\Gamma$ restricts the  location of all tangencies in a special tritangent in the support of $\bclass$:

 \begin{lemma}\label{lm:specialConfigPrecisions} Let  $\Lambda$ be a special tritangent in the support of $\bclass$. Then:
   \begin{enumerate}[(i)]
   \item all vertices of $\Lambda$ belong to the closure of the same unbounded chamber  in $\RR^2 \smallsetminus \Gamma$, dual to one of the eight boundary points $(i,j)$ in the square of side length three with $3\nmid (i+j)$;
   \item the location of the tangencies depend solely on the pair $(i,j)$;
   \item among the two tangencies in the interior of the same leg of $\Lambda$, the one closest to the attaching vertex in $\Lambda$ has type (3a) and lies on a leg of $\Gamma$.    
   \end{enumerate}
 \end{lemma}
 \begin{proof}
   Up to $\Dn{4}$-symmetry, we can restrict to the case when   $\Lambda$  has two tangencies in the interior of its negative horizontal edge. Call them $P$ and $P'$ (viewed from left to write), and let $P''$ be the remaining tangency point between $\Lambda$ and $\Gamma$.
   The bidegree of $\Gamma$ and the presence of two tangencies on the same leg of $\Lambda$ force the tangency type of $P'$ to be (3a) and furthermore, it must be either on the top or bottom positive horizontal leg of $\Gamma$. Thus, claim (iii) holds.

   In turn, this last condition ensures that the vertices of $\Lambda$  lie on the closure of the same chamber of $\RR^2\smallsetminus \Gamma$, namely, $(3,i)^{\vee}$ with $i\in \{1, 2\}$. Note that $\Lambda$ can be $4$-valent. 
   If $i=1$,  $P''$ lies on the positive vertical leg of $\Lambda$, whereas if $i=2$,  $P''$ belongs to the negative vertical leg of $\Lambda$. The action of $\Dn{4}$ confirms claims (i) and (ii)  in the statement.
 \end{proof}

Our choice of structure for $\bclass$ is compatible with that induced on $\RR^2$ by $\Gamma$ (see~\autoref{rm:finevCoarseStructure}). This confirms that we can test whether a cell is special or not by picking any member in its relative interior. More precisely, we have:
\begin{lemma}\label{lm:allOrNoMemberOfACellIsSpecial}
  Given a cell of $\bclass$, either none or all points in its relative interior correspond to tritangents with special configurations of tangencies.
\end{lemma}

\begin{proof} We let $\sC$ be a cell containing a special tritangent $\Lambda$ in its relative interior. By construction, 
  the property of having a special tangency configuration is preserved by local moves. The conditions on the polyhedral complex structure on $\tclass$ and $\bclass$ imposed in~\autoref{rm:finevCoarseStructure} ensure that all members of the relative interior of $\sC$ have this property if a single member has this feature.
\end{proof}

 We let  $\mathscr{S}$ be the collection of all special cells of $\bclass$. This set gives rise to a natural subcomplex of $\bclass$, which we now define:

 \begin{definition} The non-special bounded complex of $\tclass$ is the complex 
  \begin{equation}\label{eq:nonSpecialBoundedComplex}
  \bclassns:=\bclass\smallsetminus \mathscr{S}
  \end{equation}
  with polyhedral structure induced from $\bclass$.
 In analogy with~\eqref{eq:dimbclass}, the dimension of $\bclassns$ becomes
\begin{equation}\label{eq:dimbclassns}
  \dim \bclassns:=\max\{\dim \sC: \sC  \text{ a cell of } \bclassns\}. 
\end{equation}
 \end{definition}

 Our interest in  $\bclassns$ stems from its properties. First,  when $\sextic$ is generic relative to $\Gamma$, its support contains all liftable members of $|\tclass|$. Second,~\autoref{thm:pathConnectednessOfNSComplexes} below confirms that the complex $\bclassns$ is connected. In particular, the search for liftable members can be performed by local moves starting from any of the non-special generic members of $|\bclass|$ listed in~\autoref{tab:combClassificationGenericPerDim}.  
Note that the inclusion $\bclassns\subseteq \bclass$ yields the inequality
 $\dim \bclassns \leq \dim \bclass$.
As we will see in~\autoref{sec:comp-dimens-tclass}, equality need not hold. \autoref{thm:comparingDimensions} below ensures that we have equality whenever $\Gamma$ is generic.

 Most notably, it is the dimension of  $\bclassns$ rather than $\bclass$, which becomes the topological invariant  restricting the lifting partition of $\tclass$. This is the content of our next statement, which we refer to as the \emph{refined lifting partition theorem}.  \autoref{thm:partitionThm} follows as a direct corollary of this result.

            \begin{table}[tb]
              \begin{tabular}{|c|c||c|c|}
                \hline
                $\dim \bclassns$ & Lifting partition of $\tclass$ &
                $\dim \bclassns$ & Lifting partition of $\tclass$\\
               \hline\hline
                $3$ & $(8,0,0,0)$&
                $2$ & $(4,2,0,0)^*$ or $(0,4,0,0)$ \\\hline
                $1$ & $(0,2,1,0)^*$ or $(0,0,2,0)$ &
                $0$ & $(0,0,0,1)$\\
                \hline
              \end{tabular}
              \caption{Classification of possible lifting partitions of a tritangent class to a tropical smooth $(3,3)$-curve $\Gamma$. The values with $^*$ occur if, and only if, the support of $\bclassns$ with prescribed dimension contains a member with a type (4b) tangency.\label{tab:classificationPartitions}}
            \end{table}

 \begin{theorem}\label{thm:refinedLiftingPartition}
   There are six possible values for the lifting partition of a tritangent class to $\Gamma$. Furthermore, they are completely determined by the dimension of its associated non-special bounded complex and whether or not its support contains a member with a (4b) tangency (see~\autoref{tab:classificationPartitions}).
\end{theorem}

 \begin{proof} \autoref{thm:partitionDim3},  combined with \autoref{pr:dimSpecialCells} below,  establishes the statement if $\dim \bclassns = 3$. The lower dimensional cases follow from 
   {Theorems}~\ref{thm:partitionDim2},~\ref{thm:partitionDim0} and~\ref{thm:partitionDim1}.
 \end{proof}

 Our next result determines the dimension of a maximal cell in $\mathscr{S}$. 

\begin{proposition}\label{pr:dimSpecialCells}
  A maximal cell in $\mathscr{S}$ has dimension two or one, depending on whether or not the complementing tangency to those in the interior of the same leg is transverse. Furthermore, such a cell only contains trivalent members in its relative interior.
\end{proposition}

\begin{proof} Let $\Lambda$ be a member in the relative interior of a maximal cell of $\mathscr{S}$, say $\sC$. \autoref{lm:specialConfigPrecisions} (i) combined with the action of $\Dn{4}$ allow us to assume that the vertices of $\Lambda$ lie in $\overline{(3,1)^{\vee}}$, and that the tangencies are located on the positive vertical and negative horizontal legs of $\Lambda$. The latter leg contains two tangency points. 

  If $\Lambda$ is $4$-valent, we can move $v_0$ to the left and $v_1$ upwards and obtain a trivalent special member $\Lambda'$ of $\tclass$ that has $\sC$ in its boundary, contradicting the maximality of $\sC$. Thus, $\Lambda$ is necessarily trivalent.

  The statement regarding the dimension follows by~\autoref{pr:localMovesTritangents}. Indeed, by the maximality of $\sC$, the vertical tangency is of type (1a) or (3c). A type (3c) one forces both vertices of $\Lambda$ to move along segments of different directions, yielding a $1$-dimensional cell. On the contrary, a type (1a) tangency paves the way for a $3$-dimensional local movement of $v_1$, so the dimension of $\sC$ is determined by the local moves of $v_0$ restricted by the horizontal (3a) tangency. Hence, $\sC$ has  dimension $2$, as desired.
  \end{proof}

In the remainder of this section, we focus our attention on two properties of  $\bclassns$, namely, its path-connectedness and the estimation of its dimension when $\Gamma$ has generic edge lengths. Both will be inherited from the complex $\bclass$.

  \begin{theorem} \label{thm:pathConnectednessOfNSComplexes} Let $\tclass$ be a tritangent class to $\Gamma$. Then, the non-special bounded  polyhedral complex $\bclassns$ is path-connected.
  \end{theorem}

  \begin{proof} Throughout, we denote by  $\Sigma^{(1)}$ the $1$-skeleton of a  polyhedral complex $\Sigma$. 
    A well-known fact in polyhedral geometry (inherited from convexity) states that a polyhedral complex is connected if, and only if, its $1$-skeleton has the same property.  Since polyhedral complexes are locally path-connected, the same equivalence holds for path-connectedness.     In particular, since the support of $\bclass$ is connected by~\autoref{cor:bcomplexDeformation}, we know that $(\bclass)^{(1)}$ is  path-connected.

    We prove the statement by showing that the $1$-skeletons of $\bclassns$ is path-connected. By construction  we have the natural inclusions of complexes
    \[(\bclassns)^{(1)}\subseteq (\bclass)^{(1)}.\]

We let $\mathscr{S}^{(1)}$ be the set of  all $1$-dimensional special cells of $\bclass$.  The complex   $(\bclassns)^{(1)}$ is  constructed by removing from $(\bclass)^{(1)}$ the relative interior of all elements of $\mathscr{S}^{(1)}$.  In order to confirm the path-connectedness of $(\bclassns)^{(1)}$, we must first describe $\mathscr{S}^{(1)}$ and the intersection of each of its members with  $(\bclassns)^{(1)}$.

By~\autoref{lm:specialConfigPrecisions}, there are eight possible distribution of tangencies on a special tritangent, corresponding to the  boundary points in  the Newton polytope of $\sextic$ excluding its four vertices. For example, those associated  to the point $(3,1)$  have two tangencies on the negative horizontal leg and one on the positive vertical leg. In addition, its unique edge has slope one. All remaining special configurations are obtained from this one by the action of $\Dn{4}$.

Two facts arise from this observation. First, we can write the support of $\mathscr{S}$ as a disjoint union of eight sets $\{|\mathscr{S}_{i}|\}_{i=1}^{8}$ corresponding to the cells associated to the same boundary point in the Newton polytope of $\sextic$. Secondly, that $(\bclass)^{(1)}$ is built from $(\bclassns)^{(1)}$ by attaching the cells of each collection $\mathscr{S}^{(1)}_{i}$
along eight  subsets of vertices in $(\bclassns)^{(1)}$.
Thus, we can build a decreasing sequence of graphs 
\[
G_9:=(\bclass)^{(1)}\supsetneq G_8\supsetneq \ldots \supsetneq G_2\supsetneq G_1:=(\bclassns)^{(1)}
\]
\noindent where each $G_{i+1}$ is obtained from $G_{i}$ by attaching the cells of $\mathscr{S}_i^{(1)}$ onto a subset of vertices in $G_{1}$.  

We argue by contradiction, and assume $G_1$ is not path-connected. Using induction we show that $G_9$ would not be path-connected, which violates our earlier findings.

\begin{figure}[t]
  \includegraphics[scale=0.35]{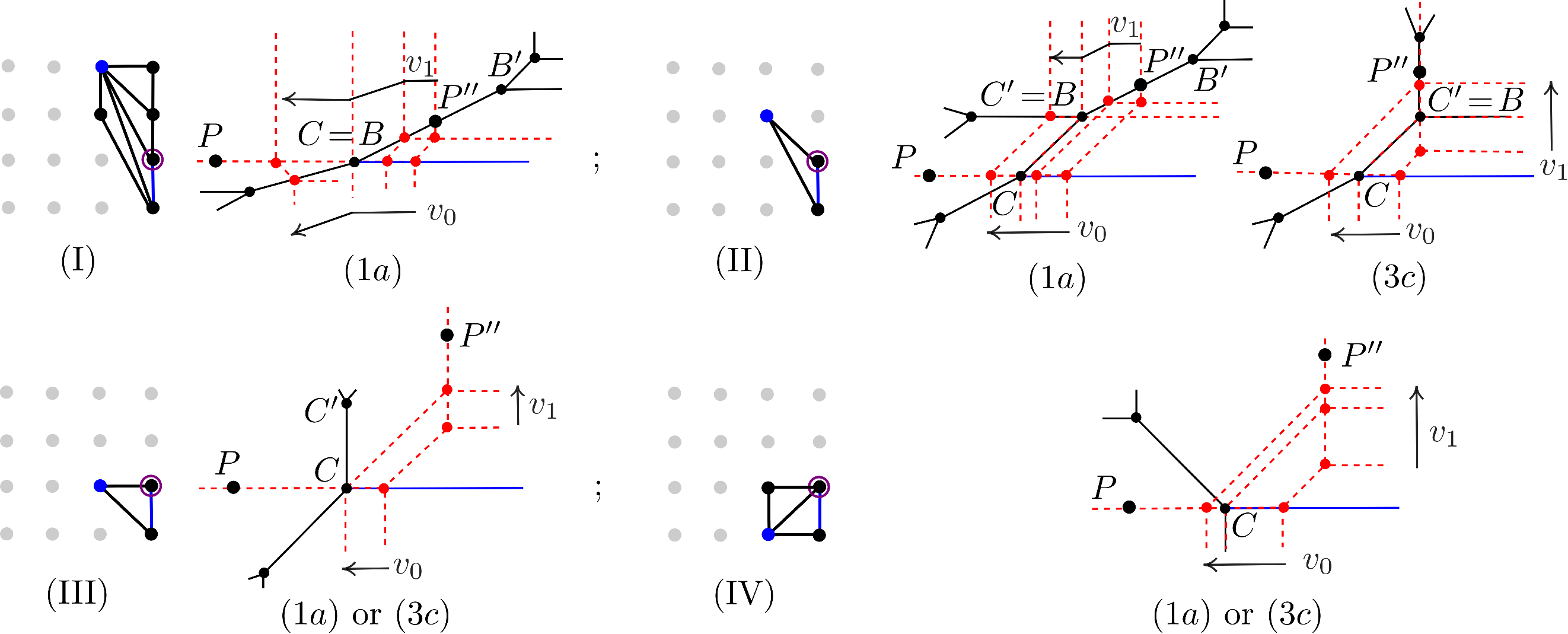} 
  \caption{Local moves for the vertices of a trivalent tritangent $\Lambda$ with a slope one edge for  configurations with two horizontal tangencies on its negative horizontal leg  (one of type (3a)) depending on the  type of the remaining tangency point $P''$, which lies on its positive vertical leg.\label{fig:specialConfigurations}}
\end{figure}

For the inductive step, we assume that $G_i$ is not path-connected. Exploiting symmetry, we may assume that $\mathscr{S}_i^{(1)}$ is associated to the point $(3,1)$. Fix a member $\Lambda$ in the relative interior of a top-dimensional cell $\sC$  of $\mathscr{S}$.
\autoref{fig:specialConfigurations} shows the four partial Newton subdivision of $\sextic$ imposed by the presence of such a special tritangent. We use the notation from the figure and let $P''$ be the tangency point complementing  those lying on the same leg of $\Lambda$.
The proof of~\autoref{pr:dimSpecialCells} confirms that  the dimension of $\sC$ is either one or two, depending on the tangency type of $P''$. In both cases, $\mathscr{S}_i$ is supported on $\sC$.

When $P''$ has type (3c), then $\sC$ is a segment. In addition, this cell is the only element in  $\mathscr{S}_i^{(1)}$. Since $\sC$, has only one vertex in $G_1 \subseteq G_i$, the non path-connectedness of $G_i$ implies  that of $G_{i+1}$.

On the contrary, if $P''$ has type (1a), we know that $\sC$ is $2$-dimensional. Note that we can identify this cell with a polygon $\cP$ in $\overline{(3,1)^{\vee}}$ recording the location of the vertex $v_1$ of each member. The set $\mathscr{S}_i^{(1)}$ then becomes a subset of the edges of $\cP$. Its precise description depends on the partial Newton subdivision of $\sextic$ seen in~\autoref{fig:specialConfigurations}.

For cases (I) through (III), $\mathscr{S}_i^{(1)}$ agrees with the boundary of $\cP$, which has up to five edges. Only one of its vertices corresponds to a non-special tritangent. 
In turn, for case (IV), $\mathscr{S}_i^{(1)}$ is a chain of edges obtained by removing   a single edge from  the boundary of $\cP$, namely, the one corresponding to tritangents with the vertex $v_0$ placed at the point $C$ seen in the figure. This segment yields an edge of $G_1$.

In all four cases, we see that $G_{i+1}$ is obtained from $G_i$ by attaching  a chain of edges along a subset of its leaves, corresponding to either one or two adjacent vertices of $G_1\subsetneq G_{i}$. Thus, the graph $G_{i+1}$ is not path-connected, as we wanted to show.
  \end{proof}

  In order to determine the dimension of $\bclassns$ from that of $\bclass$ it becomes essential to know which cells of $\bclassns$ are adjacent to special ones.  The next lemma accomplishes this task.

\begin{lemma}\label{lm:specialConfigurations} 
  Let $(\Lambda,P,P',P'')$ be a tritangent tuple to $\Gamma$ in the relative interior of a cell $\sC$ of $\mathscr{S}$. Assume that $\Lambda$ is trivalent with a  slope one edge, $P'$ is a (3a) tangency on the negative horizontal leg and $P$ is another horizontal tangency on the same leg further away from the vertex $v_0$ of $\Lambda$. 
  If $\sC$ is maximal, we can find a neighboring cell of $\bclass$ contained in $\bclassns$ and a member in its relative interior with a tangency at $P$ and one of the following  configurations of tangencies:
  \begin{center}
   \includegraphics[scale=0.35]{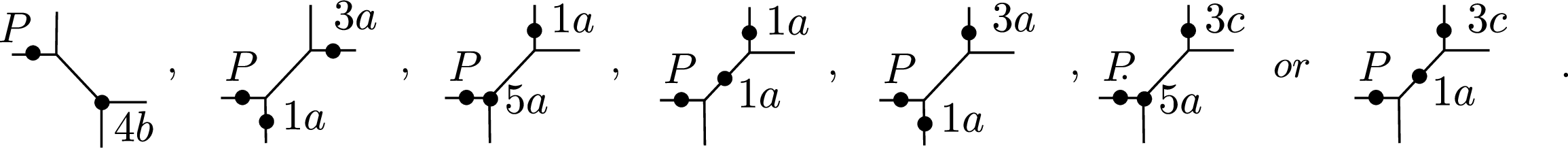}
  \end{center}
  The last three occur when $P''$ has type (3c), whereas the remaining four have type (1a).
\end{lemma}

\begin{proof} Since $\Lambda$ has an edge of slope one, the bidegree of $\Gamma$ forces  $P$ to lie on the lowest positive horizontal leg of $\Gamma$ and have tangency type (1a), (2a), or (3c), marked in blue in~\autoref{fig:specialConfigurations}. We let $C$ be the unique vertex of the lef of $\Lambda$ containing $P$.
  
  The possible non-special configuration of the  tritangent (labeled $\Lambda'$) that is equivalent  to $\Lambda$ and lies on a neighboring cell to $\sC$ is solely restricted by the Newton subdivision of $\sextic$. More precisely, it depends on the tangency type of $P''$ and the dual cell $C^{\vee}$. The figure shows how to obtain $\Lambda'$ from $\Lambda$ by moving $v_0$ towards $P$. The construction confirms that $\Lambda'$ lies in the relative interior of a cell neighboring $\sC$ and has the tangency types listed in the statement.
\end{proof}

Our last statement in this section provides explicit sufficient conditions ensuring equality between the dimensions of $\bclassns$ and  $\bclass$. When combined with~\autoref{thm:refinedLiftingPartition}, it confirms the data provided in~\autoref{tab:classificationPartitionsGeneric}.

\begin{theorem}\label{thm:comparingDimensions}
  Let $\tclass$ be a tritangent class to $\Gamma$. If $\Gamma$ has generic edge lengths, then  \[\dim \bclassns = \dim \bclass.\] 
 In addition, if $\dim \bclassns=2$ or $3$, the same equality holds even without any genericity assumption.
  \end{theorem}

\begin{proof}   The second part of the statement is a direct consequence of~\autoref{pr:dimSpecialCells}. Thus, to prove the first part, it suffices to analyze the cases when $\bclassns \leq 1$. Since a special tritangent must contain a type (3a) tangency, combining~\autoref{pr:3dCells} with~\autoref{lm:allOrNoMemberOfACellIsSpecial} confirms that $\dim\bclass \leq 2$.

  We claim that $\dim \bclass <2$.  We argue by contradiction, and suppose  $\dim \bclass =2$.
Our dimension assumptions and~\autoref{pr:dimSpecialCells} ensure that $\bclass$ contains a $2$-dimensional special cell, which is necessarily maximal.
We let $(\Lambda,P,P',P'')$ be a tritangent tuple in the interior of such a cell. Up to symmetry, we suppose that the unique edge of $\Lambda$ has slope one. As in~\autoref{lm:specialConfigurations}, we may assume that both $P$ and  $P'$ lie on the negative horizontal leg of $\Lambda$ whereas $P''$ is on the positive vertical leg. Furthermore, $P''$ must be of type (1a) for dimension reasons.

The aforementioned lemma provides four options for a  tritangent member $\Lambda'$ in the interior of a neighboring non-special cell to $\sC$, which we label $\sC'$. Since $\Gamma$ has generic edge lengths, we know that $P$ is of type (1a). Thus, from left to right, the dimension of $\mathscr{C'}$ appearing in the lemma equal $2$, $2$, $1$ and $3$, respectively.  Our assumptions on $\dim \bclassns$ and $\dim \bclass$ discard all options, but the third one. However, in this situation, we can further move $v_0$ in two directions (either upwards or away from $v_1$), to lie in a two-dimensional non-special cell of $\bclass$. Again, we reach a contradiction.

Equipped with the conditions $\dim \bclassns \leq \dim \bclass \leq 1$, it suffices to show that $\dim \bclassns = 0$ forces $\dim \bclass=0$ when $\Gamma$ has generic edge lengths. The path-connectedness of $\bclassns$ shown in~\autoref{thm:pathConnectednessOfNSComplexes}, confirms that $\bclassns$ is a point whenever its dimension is zero. If $\dim \bclass =1$, this point will be in the boundary of a 1-dimensional special cell $\sC$ in $\bclass$. For dimension reasons, the corresponding tritangent $\Lambda'$ agrees with the last three listed in~\autoref{lm:specialConfigurations}. However, all of them admits positive dimensional movements within $\bclassns$.  This cannot happen by our earlier dimension assumption.\end{proof}

\section{Finding liftable tritangents through local moves}\label{sec:technicalLemmasDim1-2}

In this section, we present several technical lemmas that are used in~\autoref{sec:proof-lift-part-2d-or-less} to simplify the proof of the refined lifting partition theorem beyond dimension three. Our objective is to determine how many liftable members we can encounter in a tritangent class  through prescribed local moves starting from a fixed  trivalent tritangent $\Lambda$ in its support.

Our first three results discuss what happens in the presence of a non-transverse diagonal tangency on $\Lambda$. For our applications, it is convenient to assume that the unique edge of $\Lambda$ has a negative slope. We are interested in describing what happens when  moving the vertex $v_0$ both away and towards $v_1$. The outcome will vary with the type of the aforementioned diagonal tangency.

\begin{lemma}\label{lm:MoveDownUpRight3-cAnd1V} 
  Let  $(\Lambda,P,P',P'')$ be a tritangent tuple where $\Lambda$ is trivalent and has a slope -1 edge, that
  lies in the relative interior of a positive-dimensional cell of $\bclass$.
  Assume that $P'$ is a non-transverse diagonal tangency in the interior of the edge of $\Lambda$,  and $P''$ is a type (1a) tangency located on its negative vertical leg. Let $e'$ and $e''$ be the edge of $\Gamma$ containing $P'$ and $P''$, respectively. 
  If  $v_0 \in (3,2)^{\vee}$, and $P$ lies in the positive vertical or negative horizontal leg of $\Lambda$, then:
  \begin{enumerate}[(i)]
  \item Moving $v_0$ away from $v_1$ leads to at most one liftable tritangent $\Lambda'$ with a tangency at $P$, whose lifting multiplicity is $2\mult(\Lambda,P)$.
  \item If $e'$ and $e''$ are adjacent, moving $v_0$ towards $v_1$ produces a tritangent $\Lambda'$ with lower vertex in $\Gamma$, containing a non-transverse diagonal tangency and a type (1a) tangency on the positive horizontal leg of $\Lambda'$, obtained by translating $P'$ and $P''$, respectively.
  \item If $e'$ and $e''$ are not adjacent, moving $v_0$ towards $v_1$ leads to at most one liftable tritangent $\Lambda'$ with a tangency at $P$. Moreover,
its lifting multiplicity equals $2\mult(\Lambda, P)$.
  \end{enumerate}
  \end{lemma}

\begin{figure}
  \includegraphics[scale=0.35]{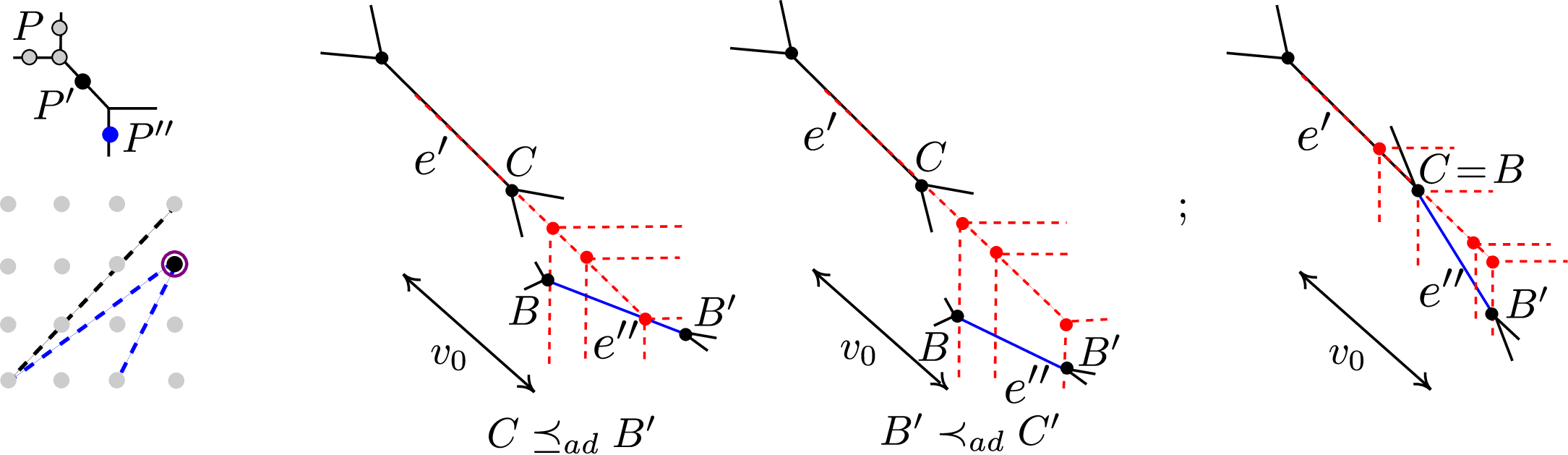}
  \caption{Movement of $v_0$ in the presence of a non-transverse diagonal tangency (labeled $P'$) and a type (1a) tangency along the negative vertical leg of $\Lambda$ (labeled $P''$), whenever $v_0\in (3,2)^{\vee}$, $\Lambda$ is trivalent and its unique edge has negative slope.\label{fig:3-cAnd1Vertical}}
  \end{figure}

\begin{proof} The configuration of tangency points on $\Lambda$ described in the statement agrees with the one seen on the left of~\autoref{fig:3-cAnd1Vertical}. Our assumptions that $v_0\notin \Gamma$ and $\Lambda' \in |\bclass|$ ensure that $P'$ has either a type (3bb1), (3ac) or (3cc) diagonal tangency. Thus, by \autoref{tab:LiftingMultiplicities}, we know that $\mult(\Lambda,P)=\mult(\Lambda,P,P')=0$ for type (3bb1), whereas $\mult(\Lambda,P, P')=2\mult(\Lambda,P)$ for the last two cases.

  To prove the three assertions, we analyze the movement constraints imposed by the tangency types of $P, P'$ and $P''$.  To simplify the exposition, we preserve the labeling of the tangency points as we perform our moves to obtain $\Lambda'$ from $\Lambda$. The required notation for the relevant endpoints $C, B, B'$ of the edges $e'$ and $e''$ is established in the figure.

  We start by proving the first claim.  As $B\prec_h B'$ and $B\preceq_v C$, we can move $v_0$ away from $v_1$ and obtain a new element of  $|\bclass|$. While the translated point $P''$ remains a type (1a) tangency, the resulting tritangent will not lift. Our movement is stopped by reaching the edge $e''$ with the lower vertex of $\Lambda'$ (if $C\prec_{ad} B'$) or $P''$ becomes the vertex $B'$ (if $C\preceq_{ad} B'$), leading to a type (4a), (2a) or (6a) vertical tangency at $P''$. Since $\mult(\Lambda', P'')=1$ by~\autoref{tab:LiftingMultiplicities} and the remaining two tangency points and their types are preserved, the lifting multiplicity of $\Lambda'$ is $2\mult(\Lambda, P)$.

  The proof of the other two items is similar.   By construction, we know that $B\prec_v C$ whenever $C\neq B$, or equivalent, when $e'$ and $e''$ are not adjacent. If $B \prec_v C$, moving $v_0$ towards $v_1$ will have a tangency (1a) vertical tangency at the translated point $P''$ until we reach a point $v_0'$ where $v_0'=_hB$. At that moment, $P''$ matches $B$ and becomes a type (2a) tangency. As in the first item, the lifting multiplicity of $\Lambda'$ is $2\mult(\Lambda, P)$. On the contrary, if $C=B$, we can move $v_0$ towards $v_1$ beyond $C$, obtaining a tritangent curve $\Lambda'$ with a non-transverse diagonal tangency along the edge of $\Lambda$ containing $P'$ and a type (1a) tangency  on its positive horizontal leg.
\end{proof}

\begin{lemma}\label{lm:MoveDown3-aAnd1H}
  Let  $(\Lambda,P,P',P'')$ be a tritangent tuple where $\Lambda$ is trivalent, with a slope -1 edge, that lies in the relative interior of a positive-dimensional cell of $\bclass$.
 Suppose that  $P'$ is a non-transverse diagonal tangency in the interior of the edge of $\Lambda$.   
  Assume that $v_0 \in (1,1)^{\vee}$ and  $P''$ is a type (1a) tangency on the positive horizontal leg of $\Lambda$. Then, 
  moving $v_0$ away from $v_1$ leads to at most one liftable tritangent containing $P$ as a tangency. Furthermore, its lifting multiplicity is $2\mult(\Lambda,P)$.
\end{lemma}

\begin{proof} 
Let $e'$ and $e''$ be the edge of $\Gamma$ containing $P'$ and $P''$, respectively. 
The proof follows the same reasoning as that of~\autoref{lm:MoveDownUpRight3-cAnd1V}, using the notation established in~\autoref{fig:3-aAnd1Horizontal}.
Since $v_0\in (1,1)^{\vee}$, we know that  $C\preceq_v B$ and the inequality is strict whenever $C\neq B$.

\begin{figure}[t]
  \includegraphics[scale=0.35]{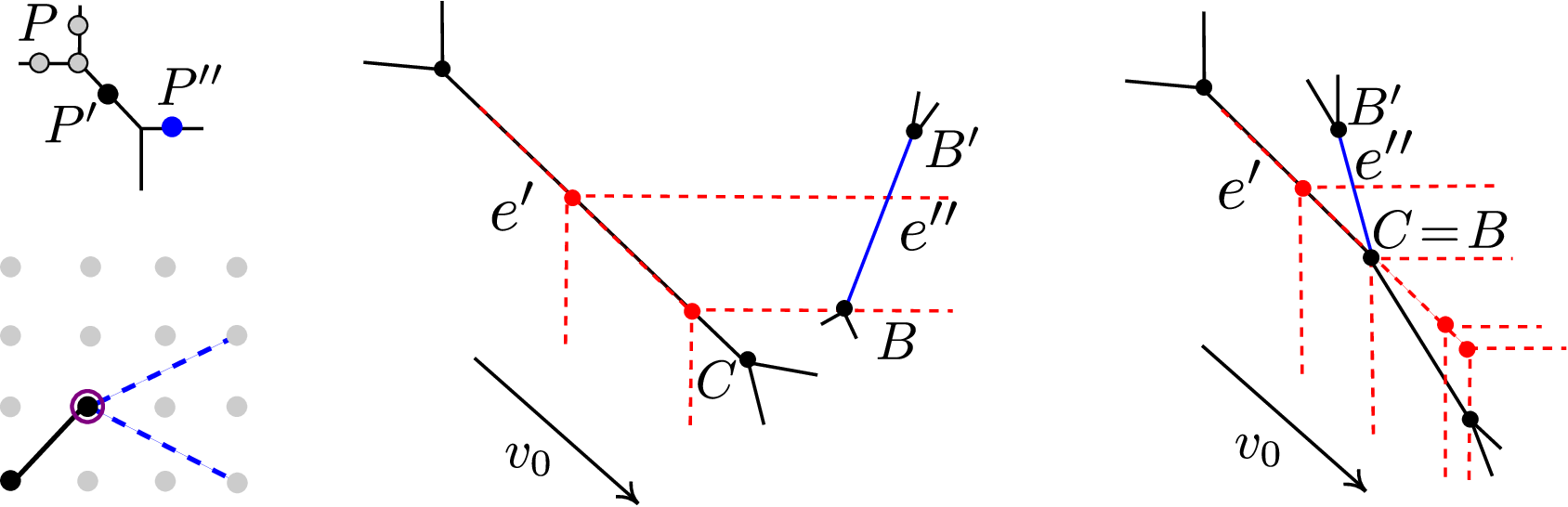}
  \caption{Movement of $v_0$ away from $v_1$ when $v_0$ lies in the boundary of $(1,1)^{\vee}$ and $\Lambda$ is trivalent with an edge of negative slope, having a non-transverse diagonal tangency $P'$ and a type (1a) tangency $P''$ along its positive horizontal leg.\label{fig:3-aAnd1Horizontal}}
  \end{figure}

If $C\neq B$, we can move $v_0$ away from $v_1$ while remaining in  $|\bclass|$. The movement stops when the translated point $P''$ reaches $B$ and the lifting multiplicity is $2\mult(\Lambda, P)$. In turn, when $C=B$ the movement continues beyond $C$ and transitions into the setting of~\autoref{lm:MoveDownUpRight3-cAnd1V} (i). Thus, the claim follows.
\end{proof}

\begin{figure}
  \includegraphics[scale=0.35]{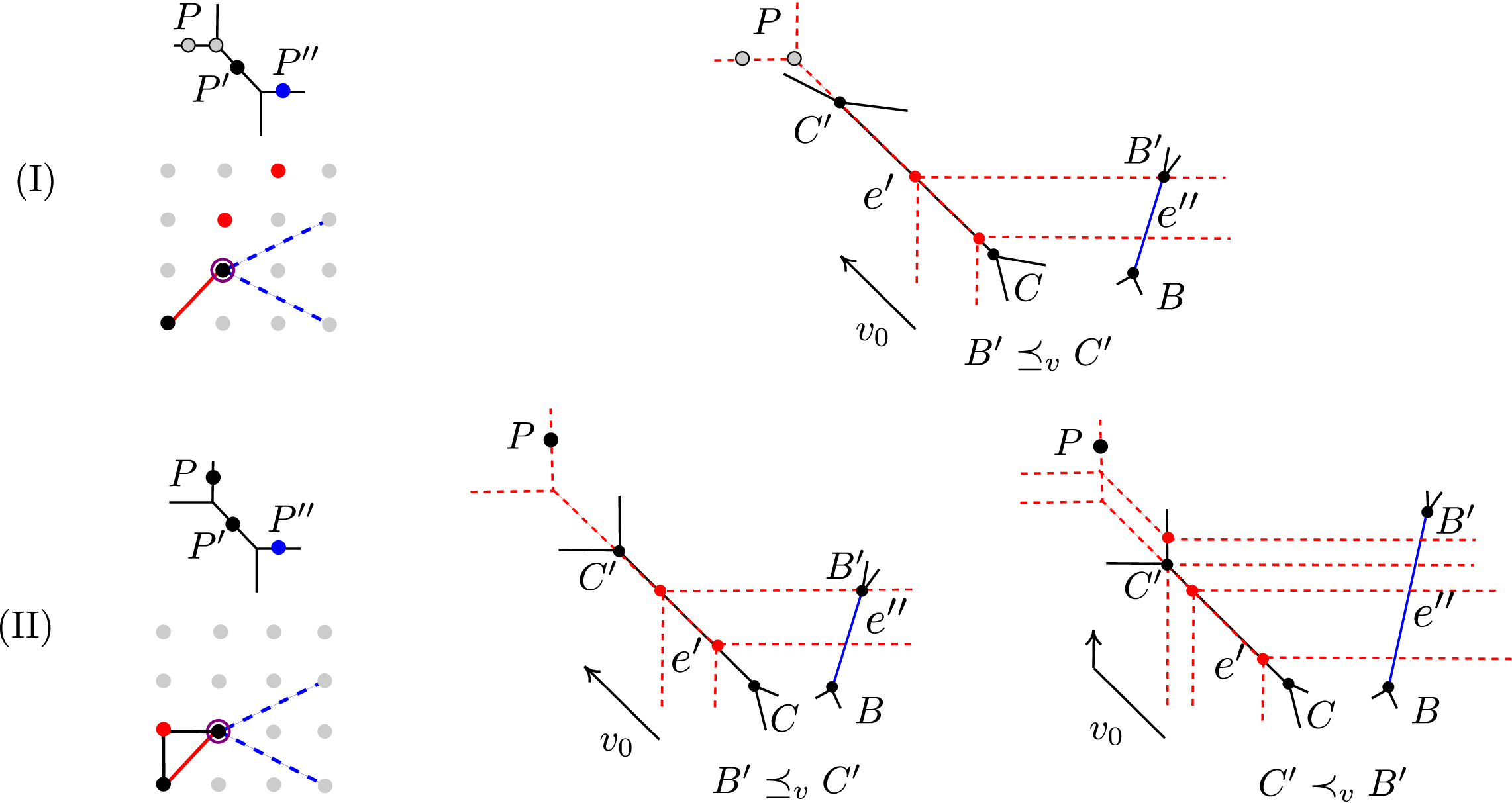}
  \caption{Movement of $v_0$ towards  $v_1$ when $v_0$ lies in the boundary of $(1,1)^{\vee}$ and $\Lambda$ is trivalent with an edge of negative slope, 
    in the presence of a (3ac) diagonal tangency $P'$ and a type (1a) tangency $P''$ along its positive horizontal leg.\label{fig:3caAnd1Horizontal}}
  \end{figure}

\begin{lemma}\label{lm:MoveUp3caAnd1H}
  Let  $(\Lambda,P,P',P'')$ be a tritangent tuple where $\Lambda$ is trivalent, with a slope -1 edge, that lies in the relative interior of a positive-dimensional cell of $\bclass$.
Assume that  $P'$ is a type (3ac) diagonal tangency  with $v_0\in \Gamma$, and  $P''$ is a type (1a) tangency  on the positive horizontal leg of $\Lambda$.
  Let $e'$ and $e''$ be the edges of $\Gamma$ containing $P'$ and $P''$, respectively, and denote by  $C'$ and $B'$ the top-most vertices of $e'$ and $e''$, respectively.
  Then, 
  moving $v_0$ towards $v_1$ leads to one of two possible scenarios:
  \begin{enumerate}[(i)]
  \item If $B'\preceq_v C'$, then 
    the movement produces at most one liftable tritangent with a tangency at $P$,  with lifting multiplicity $2\mult(\Lambda,P)$. 
    \item If $C'\prec_v B'$, then $P$ lies in the vertical leg of $\Lambda$, and  we can move $v_0$ upwards past $C'$ (and $v_1$, accordingly) to produce a tritangent with tangencies on its  negative vertical and positive horizontal legs,  of types (3a) and (1a), respectively.
  \end{enumerate}
\end{lemma}

\begin{proof} The proof is outlined in~\autoref{fig:3caAnd1Horizontal}. The approach depends on the combinatorics of the dual cell to $C'$. By construction, the dual edge to $e'$ has vertices $(0,0)$ and $(1,1)$. Thus, the third vertex of $(C')^{\vee}$ equals $(i,i+1)$ for some $i\in\{0,1,2\}$.

  Whenever $B'\prec_v C'$, we can continuously move $v_0$ towards $v_1$ while remaining in $|\bclass|$, where $P'' \in e''$ remains of type (1a). The movement stops when  $P''$ reaches $B'$, and we obtain a type (2a) tangency. Thus, only the last one can potentially lift. Furthermore, its  lifting multiplicity equals $2\mult(\Lambda, P)$, since the tangency type of $P'$ remains (3ac) because $C'\prec_v v_1$.

  On the contrary, when $C'\preceq_v B'$, it follows that $i=0$, so the tangency point $P$ lies in the vertical leg of $\Lambda$, as in the second row of the figure. If $C'=_vB'$, the movement stops when $v_0$ reaches $C'$. The resulting tritangent has a  type (5a) tangency at the point $P'$ and a type (2a) tangency at $P''$. We conclude from this that  we obtain at most one liftable member when moving $v_0$ towards $v_1$, with lifting multiplicity $2\mult(\Lambda, P)$ as per~\autoref{tab:LiftingMultiplicities}.

  Notice that the only  movement  of $v_0$ past $C'$  preserving $P$ as a tangency is horizontal, along the horizontalleg of $\Gamma$ containing $C'$. By construction, the resulting tritangents are special, with two tangencies along its positive horizontal leg, one of which is $B'$. Thus, none of these members lift.
  
  Finally, if $C'\prec_v B'$ we can move $v_0$ vertically past $C'$ while $P''\in e''$ remains a type (1a)  tangency and $P'$ becomes a type (3a) vertical tangency. The vertex $v_1$ will move upwards unless $P$ has tangency type (4a) or (6a), in which case it will move travel along the edge of $\Gamma$ that contains $P$ in the boundary of the chamber $(0,1)^{\vee}$ and is responsible for this vertical tangency.
  \end{proof}

Our next two lemmas focus on the case when the tritangent $\Lambda$ has a slope one edge and contains two tangency points, of types (3a) horizontal and (1a) vertical, on the boundary of the chamber $(1,2)^{\vee}$ of $\RR^2\smallsetminus \Gamma$. By construction, the remaining tangency point between $\Lambda$ and $\Gamma$ (which we call $P$) lies in the positive horizontal leg of $\Lambda$. Our objective is to determine which liftable members we encounter when performing any movement of $\Lambda$ that preserves $P$. Our first statement discusses the left-oriented movement, while~\autoref{lm:MoveRightFrom3a1P} treats  displacements towards the right.

  \begin{figure}[t]
  \includegraphics[scale=0.35]{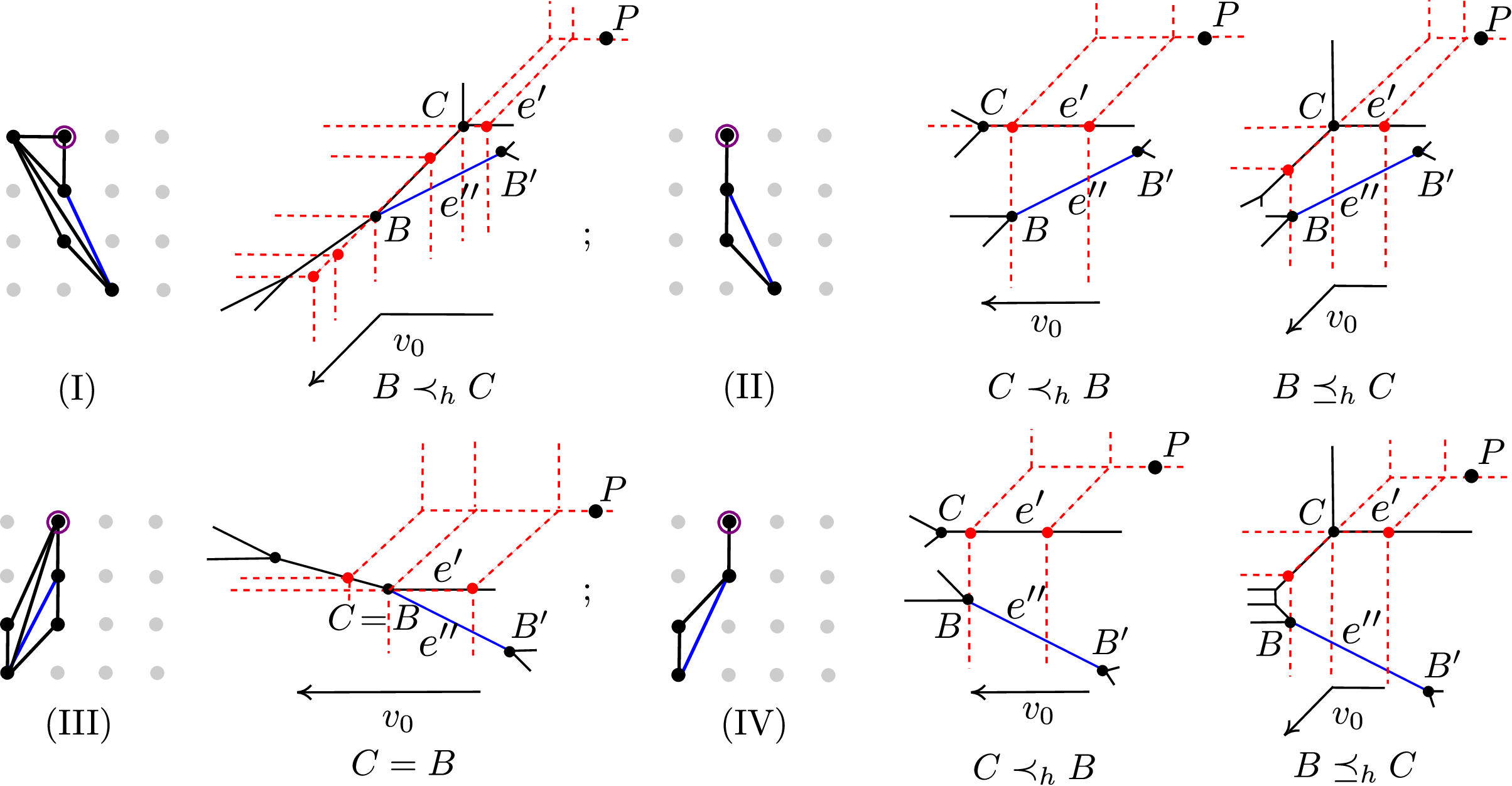}
  \caption{Left-oriented movement of $v_0$ in the presence of a (3a) tangency on the negative horizontal leg of $\Lambda$,  a type (1a) tangency along its negative vertical leg  and an extra tangency labeled $P$ on the positive horizontal leg that is fixed throughout the movement. Here,  $\Lambda$ is trivalent, its unique edge has slope one and $v_1\in \overline{(1,3)^{\vee}}$ .\label{fig:3aAnd1VerticalLeftMove}}
  \end{figure}

\begin{lemma}\label{lm:MoveLeftFrom3a1P}
  Let $(\Lambda,P,P',P'')$ be a tritangent tuple where $\Lambda$ is trivalent and has a slope one edge, which lies in the relative interior of a positive-dimensional cell of $\bclass$. Assume that $P'$ is a (3a) horizontal tangency and $P''$ is a (1a) vertical tangency, both in  $\overline{(1,2)^{\vee}}$, as in~\autoref{fig:3aAnd1VerticalLeftMove}.  Then, 
  moving $v_0$  in the left direction while preserving $P$ leads to one of two possible scenarios:
  \begin{enumerate}[(i)]
  \item If the edges of $\Gamma$ containing $P'$ and $P''$ have the same left-most vertex, we encounter a member of $|\bclass|$ with a type (4b) tangency at $v_0$.
    \item In all other cases,  we obtain at most one liftable member through this movement  and its lifting multiplicity equals $2\mult(\Lambda, P)$. Furthermore, the location of its lower vertex depends only on the position of $P'$ and $P''$ within this curve.
  \end{enumerate}
\end{lemma}

\begin{proof} We let $e'$ and $e''$ be  edges of $\Gamma$ containing $P'$ and $P''$, respectively. We let $C$ and $B$ be the left-most endpoints of $e'$ and $e''$, respectively. \autoref{fig:3aAnd1VerticalLeftMove} describes four different scenarios, depending on the possible combinations of dual cells $C^{\vee}$ and $B^{\vee}$, the relative $\preceq_h$ order between $B$ and $C$, and whether or not $C=B$. The last situation is the only one in which we encounter a type (4b) tangency during the left movement of $\Lambda$, and appears in the case labeled (III) in the figure.

  When $C\neq B$, the figure confirms that we reach at most one liftable tritangent as a result of this movement. Furthermore, the tangencies on this curve complementing $P$ have fixed types. Indeed, the curve has a type (2a) tangency at the translated point $P''$ and a non-transverse one at $P'$, which can be of type (3a) (if $C\prec_h B$),  type (5a)  (if  $C=_hB$), or a diagonal non-transverse one of types (3ac) or (3cc) (if $B\prec_h C$). The latter one is symmetric to the setting of the rightmost picture in~\autoref{fig:3-aAnd1Horizontal}.

  In all  three cases, the lifting multiplicity of the new tritangent equals $2\mult(\Lambda, P)$ and the position of its lower vertex  is independent of $P$. Any other member of $|\bclass|$ obtained along the way has a type (1a) or (3bb1) tangency and hence does not lift.
\end{proof}

\begin{figure}[t]
  \includegraphics[scale=0.35]{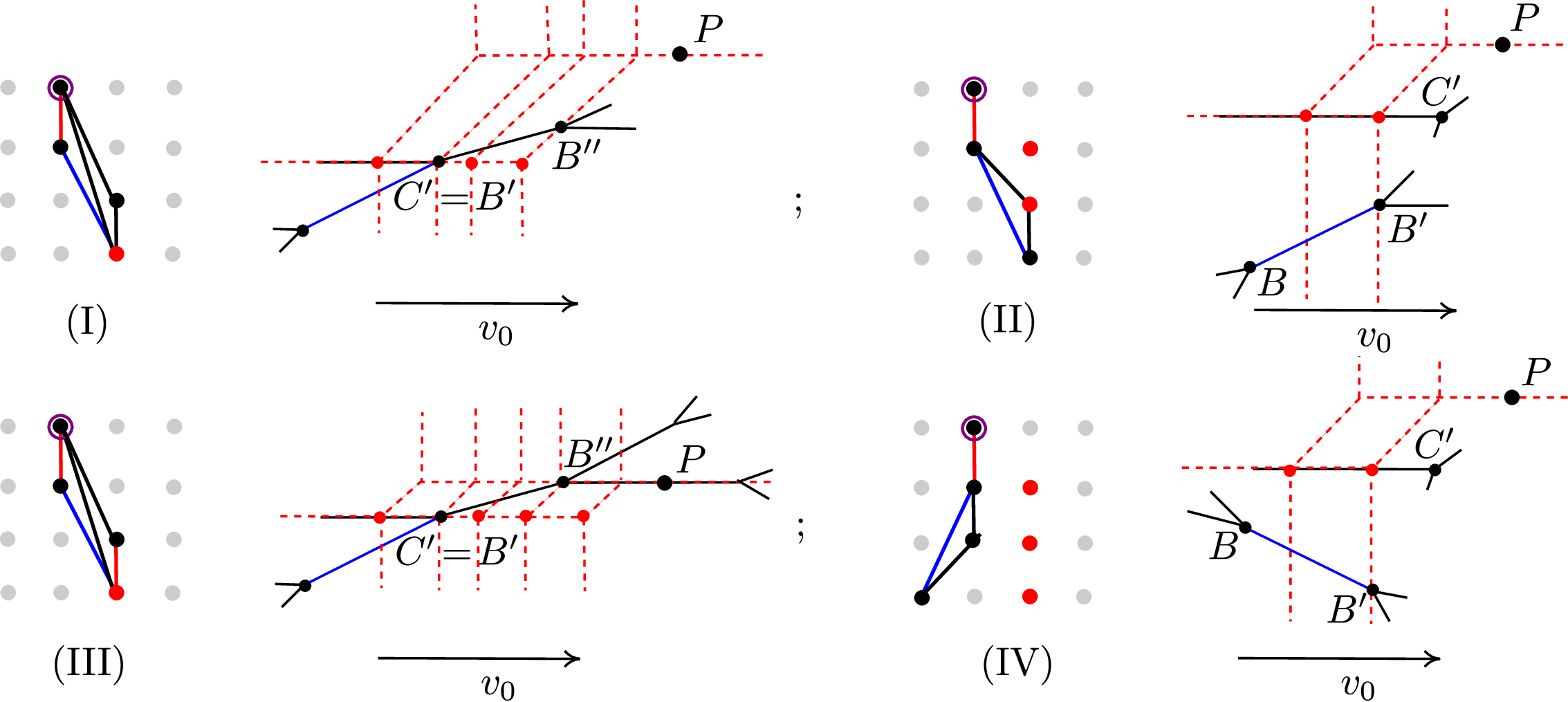}
  \caption{Right-oriented movement of $v_0$ in the presence of a (3a) tangency on the negative horizontal leg of $\Lambda$,  a type (1a) tangency along the negative vertical leg  and an extra tangency labeled $P$ on the positive horizontal leg that is fixed. Here,  $\Lambda$ is trivalent, its unique edge has slope one and  $v_1\in \overline{(1,3)^{\vee}}$ .\label{fig:3aAnd1VerticalRightMove}}
  \end{figure}

\begin{lemma}\label{lm:MoveRightFrom3a1P}
  Let $(\Lambda,P,P',P'')$ be a tritangent tuple where $\Lambda$ is trivalent and has a slope one edge, which lies in the relative interior of a positive-dimensional cell of $\bclass$. Suppose that  $P'$ is a (3a) horizontal tangency, $P''$ is a (1a) vertical tangency, both in  $\overline{(1,2)^{\vee}}$, and $P$ is a tangency on the positive horizontal leg of $\Lambda$, as in~\autoref{fig:3aAnd1VerticalRightMove}.  Then, 
  moving $v_0$ towards the right while preserving $P$ leads to one of two possible scenarios:
  \begin{enumerate}[(i)]
  \item If $P$ and $P''$ are not in the boundary of the same chamber of $\RR^2\smallsetminus \Gamma$, we obtain at most one liftable member with a tangency at $P$, and its lifting multiplicity equals $2\mult(\Lambda, P)$.
    \item In all other cases, we reach a type (3bb) horizontal tangency, containing $P$. Moving further to the right while preserving the type (3c) tangency at the translated point $P'$,  produces a member with a type (3a) horizontal tangency and a type (1a) vertical one, both in  $\overline{(2,1)^{\vee}}$.
  \end{enumerate}
\end{lemma}

\begin{proof} The result follows by analyzing the possible combinatorial types of the rightmost vertices  of the edges of $\Gamma$ containing $P'$ and $P''$, which we label as $C'$ and $B'$ in the figure. Notice that the existence of the tangency point $P$ on the positive horizontal leg of $\Lambda$ hinders the inequality $C'\prec_h B'$. Thus, we have $B'\preceq_h C'$.  Furthermore, $B'=_hC'$ can only occur if $B'=C'$.

  The setting when  $B'\prec_h C'$ corresponds to  the Newton subdivisions labeled (II) and (IV). In these situations the right movement will be stopped by $B'$, since along the way $v_1$ stays in the chamber $(1,3)^{\vee}$ for convexity reasons. This establishes claim (i) for such $B'$ and $C'$.

  In turn, when $B'=C'$, the subdivisions that can occur are  (I) and (III). In both cases, we can move $v_0$ past $C'$ to obtain a diagonal type (1a) tangency along an edge of $\Gamma$ of slope $1/3$, while keeping $v_1\in (1,3)^{\vee}$. We let $B''$ be the rightmost endpoint of this edge as indicated in the figure.

  If $P\notin \overline{(2,0)^{\vee}}$, we are in case (I) and the movement  stops when we reach $B''$.  The point $B''$ carries a  type (2a) diagonal tangency, while $P'$ becomes a type (3c) horizontal tangency. The resulting tritangent $\Lambda'$ belongs to $|\bclass|$ and  has lifting multiplicity $2\mult(\Lambda, P)$, as stated in  claim (i).

  On the contrary, if $P\in \overline{(2,0)^{\vee}}$, we are in case (III). When $v_0$ reaches $B''$, we acquire a type (3bb) horizontal tangency at the pair $(P,B'')$, as predicted by claim (ii). Moving $v_0$ further to the right past $B''$ turns this (non-liftable) multiplicity four tangency into a combination of two tangencies: a horizontal (3a) one and a vertical (1a) one, both in $\overline{(2,1)^{\vee}}$.
  \end{proof}

The following result is a natural consequence of the previous two lemmas:

\begin{corollary}\label{cor:leftAndRightmostLiftableMembers3aAnd1Vertical}
  Let  $(\Lambda,P,P',P'')$ be a tritangent tuple contained in the interior of a positive-dimensional cell of $\bclass$, where $\Lambda$ has a slope one edge, $P'$ is a (3a) horizontal tangency, $P''$ is a (1a) vertical tangency, both in  $\overline{(1,2)^{\vee}}$, and $P$ lies in the positive horizontal leg of $\Gamma$. If no member of $|\bclass|$ has a type (4b) tangency, then:
    \begin{enumerate}[(i)]
    \item any member $\Lambda'$ of $|\bclass|$ with lower vertex $v_0'$ satisfying $v_0'\preceq_h v_0$ has a tangency  in the  edge of $\Gamma$ containing $P$;
    \item   the number of such members that lift  depends on the tangency type of $P$: there will be exactly one (with lifting multiplicity $4$) if $P$ has type (3c), and two  (with lifting multiplicity $2$) in all other cases;
\item if $P$ and $P''$ are not in the boundary of the same chamber of $\RR^2\smallsetminus \Gamma$, the conclusion of (i) applies to  members $\Lambda'$ with $v_0\prec_h v_0'$; furthermore, there are precisely either one or two such members that lift depending on the tangency type of $P$, and their lifting multiplicities match those of item (ii);
\item if $P$ and $P''$ are in boundary of the same chamber of
 $\RR^2\smallsetminus \Gamma$, there is precisely one liftable member $\Lambda'$ with $v_0\prec_h v_0'$, and its lifting multiplicity is four.
  \end{enumerate}
\end{corollary}

\begin{proof} The result follows by tracing the possible local moves at $\Lambda$. The necessary information to prove the first two items, including all relevant cells in the Newton subdivision of $\sextic$, is available in~\autoref{fig:3aAnd1VerticalLeftMove}. We let $e$ and $e'$ be the edges of $\Gamma$ containing $P$ and $P'$, respectively.

  If $v_0'=v_0$, the presence of the tangency $P$ together with \autoref{pr:unbounded} and~\autoref{lm:unboundedtranslations} imply that the only moves that produce members in $|\bclass|$ with the same lower vertex as $\Lambda$ correspond to a (possibly degenerate) segment recording the location of the top-vertex of any such tritangent. For each of them, the tangency point $P$  remains in the edge $e$. Moving $v_0$ towards the left along the horizontal edge $e'$ will not change this fact.

The figure shows the various options to move $v_0$ in the left direction beyond the leftmost vertex of $e$, which is labeled by $C$. If we encounter a member $\Lambda'$ with $v_0'=C$, our hypotheses on the tuple $(\Lambda, P,P',P'')$ and $\bclass$ confirm that case (III) in the figure does not occur, forcing the star  of $\Gamma$ at $C$ to be a min-tropical line with vertex $C$. Thus, we can move $v_0$ upwards and reach a special cell of $\bclass$ (in the sense of~\autoref{def:specialTritangents}) or diagonally away from $C$. In both cases, the extra tangency will remain at $e$, even if the vertex $v_1$ is also shifted towards or away from $v_0$. This proves the first claim of the statement.

The second part is established by combining our first claim and~\autoref{lm:MoveLeftFrom3a1P} (ii). Indeed, if $P$ is of type (3c), there is a unique liftable member $\Lambda'$ with $v_0'\prec_h v_0$ and its lifting multiplicity is 4. In turn, if the tangency at $P$ is transverse, any liftable member $\Lambda'$ with $v_0'\preceq_h v_0$ must have a tangency at $e$ with positive local lifting multiplicity, i.e., $P$ must be of type (2a), (4a) or (6a). By the lemma, the location of $v_0'$ is independent of $P$, so we have at most two such liftable members.

In order to construct them, we start by first shifting $v_1$ towards $v_0$ to turn our initial point $P$ into a type (2a) tangency on the lower vertex of $e$. Note that we can do so because the slope of $e$ is positive. After this step, the lemma produces a  single liftable member with a tangency at the aforementioned vertex of $e$, whose multiplicity equals 2. A second liftable member $\Lambda''$ will be obtained by moving $v_1$ away from $v_0$ until we reach a (2a), (4a) or (6a) horizontal tangency along the positive horizontal leg of $\Lambda''$. Its lifting multiplicity will also equal $2$. 

The proof of (iii), is similar to that of (ii), and it is obtained by following the movements described by~\autoref{lm:MoveRightFrom3a1P} (i). Finally, we establish  claim (iv). The hypothesis on $P$ and $P''$ implies that these two points are on the boundary of the chamber ${(2,0)^{\vee}}$, which corresponds to the subdivision seen in~\autoref{fig:3aAnd1VerticalRightMove} (III). Thus, the tangency point $P$ in  $\Lambda$ has type (3c). This ensures that the local movement at $v_1$ is also horizontal. A displacement  beyond the vertex $B''$ produces a tuple $(\Lambda',Q,Q',Q'')$ where $\Lambda$ is trivalent with a slope one edge, and the types of $Q, Q'$ and $Q''$ are (1a), (3a) and (3c), respectively. Reversing the roles of the vertices of $\Lambda$ and using item (ii) on this new tuple produces exactly one liftable member $\Lambda''$ with $v_0\prec_h v_0''$ whose lifting multiplicity is four.
\end{proof}

\section{Proof of the refined lifting partition theorem when $\dim \bclassns = 2$}\label{sec:proof-lift-part-2d-or-less}

In~\autoref{sec:comb-bound-compl} we established the lifting partition theorem for tritangent classes with full-dimensional non-special bounded complexes. The present section proves the statement in the next relevant case, namely, when  $\bclassns$ is 2-dimensional.  By~\autoref{thm:comparingDimensions} we know that $\dim \bclassns = \dim \bclass = 2$.
 Our main result, stated below, characterizes the possible lifting partitions of the corresponding tritangent class  under this restriction.  Its proof spans the remainder of this section.

 \begin{theorem}\label{thm:partitionDim2}
  Let $\tclass$ be a tritangent class of $\Gamma$ satisfying $\dim \bclassns = 2$. Then, the lifting partition of $\tclass$ equals $(4,2,0,0)$ or $(0,4,0,0)$. Furthermore, the first partition occurs if, and only if, $|\bclassns|$ contains a  member with a (4b) tangency point.
\end{theorem}

\begin{proof}
  By our earlier discussion, we know that $\dim \bclass=2$. The top-row of~\autoref{tab:combClassificationGenericPerDim} shows the distribution of tangency points on a generic member of a top-dimensional cell of $\bclass$. After discarding the special ones (i.e., the last member on the  $(2,(i))$-entry), we are left with 12 possibilities to inspect.   
  {Propositions}~\ref{pr:type4b1} through~\ref{pr:13acOr3cc1sameOrNot} below determine the lifting partitions for these 12 sample generic tritangents, conveniently grouped by common features. The statement is a direct consequence of those results.
\end{proof}

In what follows, we provide separate statements  corresponding to  the 12 non-special generic members of $\bclassns$ used in the proof of~\autoref{thm:partitionDim2}. Each of them involves tritangent curves that impose the  same restrictions on the Newton subdivision of $\sextic$. To simplify the arguments it is  convenient to pick alternative $\Dn{4}$-orbit representatives for some of the tritangents listed in  the $(2,(ii))$-entry of~\autoref{tab:combClassificationGenericPerDim}. The results presented in~\autoref{sec:technicalLemmasDim1-2}  greatly simplify our task.

Our first lemma discusses lifting conditions for  members of a cell of $\bclassns$ containing a member with a  type (4b) tangency in its relative interior.

\begin{lemma}\label{lm:type4b_1}
  Let $\Lambda$ be a trivalent tritangent containing two tangency points, one of type (4b) and one of type (1a).  Then, there exists a polyhedron $\cQ$ containing $\Lambda$, with $\cQ\subseteq |\bclass|$ that has exactly four liftable elements, all with multiplicity one. Furthermore,  each of these members is a vertex of $\cQ$, has a multiplicity four tangency point at a vertex, whose type is either (4b'), (6b'), (5b) or (6b). In addition, $\cQ$ contains an open segment corresponding to members with two tangency points, of types (5b) and (1a), respectively.
\end{lemma}

\begin{proof} We start by setting up the required notation. By symmetry, we may assume that the unique edge of $\Lambda$ as slope one, and the type (4b) tangency, called $P$, lies at the vertex $v_0$ of $\Lambda$. Since $\Lambda\in |\bclass|$, it follows that $v_0\in \overline{(1,3)^{\vee}}$. We let $P'$ be the complementing tangency, which must necessarily lie on the positive horizontal leg of $\Lambda$. We denote by $e$ and $e'$ the edges of $\Gamma$ containing $P$ and $P'$, respectively.    Note that the edge $e$ has slope $-1/3$, but the slope of  $e'$ can vary.

  We establish the statement by following the same strategy as the one from~\autoref{lm:case1Dim3}. The relevant data can be found in~\autoref{fig:parallelogramsDim2_4b}. We set $\cQ := \cP\cap \overline{(1,3)^{\vee}}$, where $\cP$ is a parallelogram constructed from the endpoints of $e$ and $e'$, namely $A,A',C$ and $C'$. This polytope  records the location of the top-vertex of relevant members of $|\bclass|$ obtained from local moves of $\Lambda$. 

\begin{figure}[t]
  \includegraphics[scale=0.3]{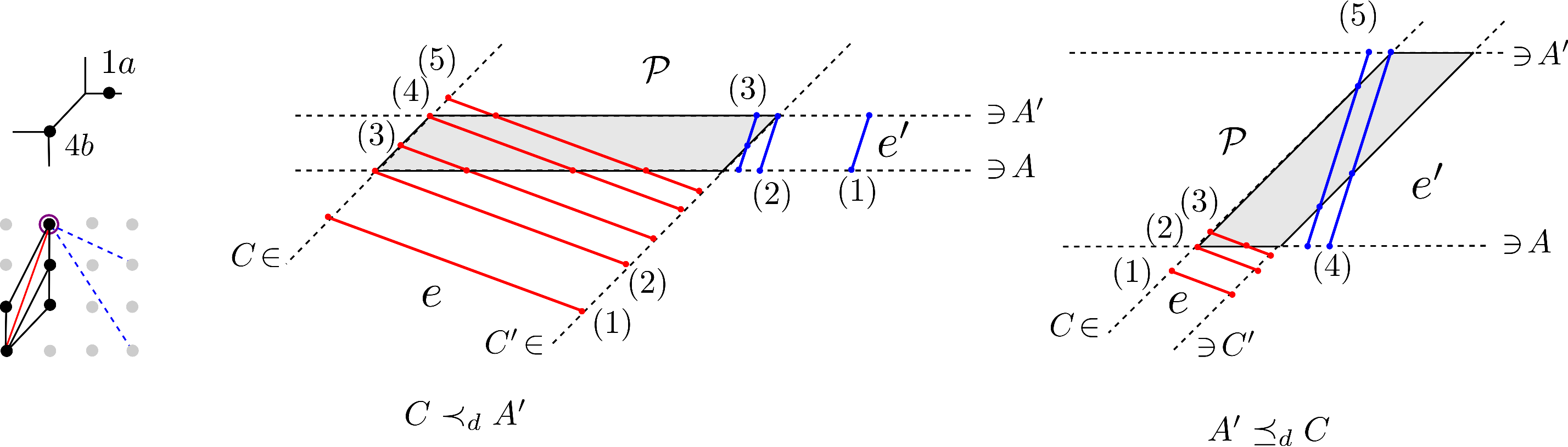}
  \caption{From left to right: partial Newton subdivision of $\sextic$ imposed by the existence of a tritangent triple $(\Lambda,P,P')$ to $\Gamma$, where $\Lambda$ has a slope one edge, $P=v_0$ has type (4b),  $P'$ has type (1a) and lies the positive horizontal leg of $\Lambda$; 
location and local moves  of the vertex $v_1$ of $\Lambda$; positions of the (red) edge $e:=\overline{CC'}$, and (blue) edge $e':=\overline{AA'}$ of $\Gamma$ containing $P$ and $P'$, respectively,  relative to the parallelogram  $\cP$ built from $A, A', C$ and $C'$. 
      \label{fig:parallelogramsDim2_4b}}
\end{figure}

  The leftmost part of the figure shows  partial information on the Newton subdivision of $\sextic$ that is imposed by the distribution of tangencies in $\Lambda$. The depicted  cells yield only  partial information regarding the relative orders between $A,A,C$ and $C'$. More precisely, we have:
  \begin{equation}\label{eq:inequalities4bDim2}
    C\prec_hC'\prec_h A \prec_h A'\quad,\quad
    C'\prec_vA\prec_v A'\quad and \quad
    C\prec_d C'\prec_d A.
  \end{equation}
In turn, this gives to up to five possible positions for the edges $e$ and $e'$ relative to the parallelogram $\cP$. Notice that positions (2) through (5) for $e'$ can only occur if its slope equals 2. For position (1), the slope of $e'$ can be either $2$ or $2/3$.

Starting from $\Lambda$ we can reach any member of  $\cQ$ by moving the vertex $v_0$ along the edge $e$. For each choice of $v_0$, the location of $v_1$ is either unique or a slope one segment within $\cQ$. It follows that any trivalent member of $\cQ$ has a tangency at $v_0$ of type (4b), (5b) or (6b). In the 4-valent case, its type will be either (4b') or (6b'). By construction, the remaining tangency point (called $P'$) must be horizontal and lie in $e'$. This guarantees that $\cQ\subseteq |\bclass|$. 
\autoref{tab:LiftingMultiplicities} confirms that a type (4b) tangency will not produce a liftable member, whereas the other ones lift  unless $P'$ has type (1a) or (1a'). Furthermore, the same table ensures their lifting multiplicities equal one.

To conclude our proof, we must to determine the location of the liftable members of $\cQ$. 
By construction, the right-most slope one edge of $\cQ$ corresponds to those members with a type (5b) tangency at $v_0=C'$. Its interior points have a type (1a) tangency, whereas its two endpoints yield liftable members, where the extra tangency is of type (2a), (4a) or (6a) depending on the position of $e'$ relative to $\cP$. 

The remaining liftable members of $\cQ$ are the left-most endpoints of the horizontal  edges in the boundary of $\cQ$ unless $e'$ is in position (5). This last case can be viewed in the same way, where the top horizontal edge of the cell degenerates to a point. The figure confirms that  the tangency types of $v_0$ for these two members is (6b) only if $e$ is in position (1). In all other cases, one of them has either a type (4b') or (6b') tangency at its lower vertex. 
\end{proof}

The next proposition follows naturally from the previous lemma:

\begin{proposition}\label{pr:type4b1} If $\dim \bclassns=2$ and contains a member with a (4b) and a (1a) tangency, then, the lifting partition of $\tclass$ equals $(4,2,0,0)$.
\end{proposition}

\begin{proof}   The result is a consequence of 
  {Lemmas}~\ref{lm:MoveRightFrom3a1P} and~\ref{lm:type4b_1}, as we now explain.  Fix a trivalent member $\Lambda$ with the prescribed tangencies. Without loss of generality, we assume the edge of $\Lambda$ has slope one and that the type (4b) tangency lies at $v_0$.   The partial Newton subdivision of $\sextic$ corresponds to the combination of those depicted in 
  {Figures}~\ref{fig:3aAnd1VerticalLeftMove} (III) and~\ref{fig:3aAnd1VerticalRightMove} (IV).

  By~\autoref{lm:type4b_1}, we can find a polygon $\cQ$ inside $|\bclass|$ containing $\Lambda$ in its relative interior. This polygon has four vertices corresponding to liftable tritangents, each of multiplicity one. All other points of $\cQ$ yield non-liftable tritangents. This fact explains the first entry of the  partition claimed in the statement. In what follows, we produce two extra tritangents in the set $|\bclass|\smallsetminus Q$, each with lifting multiplicity two.

  By the lemma, $\cQ$ contains a member with a type (5b) tangency at $v_0$, complemented by a (1a) one (which we call $P'$) along its positive horizontal leg. For this member, we can move $v_0$ towards the right along the unique horizontal leg of $\Gamma$  with left endpoint $v_0$, while preserving the tangency at $P'$. The newly encountered tritangent tuples $(\Lambda', P, P', P'')$ will have slope one edges, and  tangency points $P$ and $P''$ along the two negative legs of $\Lambda'$, of types (3a) horizontal and (1a) vertical, respectively.

    \autoref{lm:MoveRightFrom3a1P} (i) shows that this right-oriented movement produces at most one potentially liftable tritangent, namely the one where $P''$ becomes a type (2a) tangency. Even though such members do not lift since $P'$ has type (1a), they yield two liftable ones, with multiplicity two. We obtain each of them by moving $v_1$ both towards and away from $v_0$. We stop whenever $P$ becomes a type (2a), (4a) or (6a) horizontal tangency, respectively.
\end{proof}

\begin{remark} The proof of~\autoref{pr:type4b1} confirms that the presence of a member of $|\bclassns|$ with tangencies of type (4b) and (1a) determines the support of this complex. Indeed, it is a union of two adjacent coplanar polygons recording the location of the top-vertex of each member, namely, $\cQ$ and  an extra polygon for members with a type (3a) tangency.
  \end{remark}

The remainder of this section discusses the non-special tritangencies  from~\autoref{tab:combClassificationGenericPerDim} corresponding to tritangent classes with lifting partition $(0,4,0,0)$. 

\begin{proposition}\label{pr:3a11Dim2}
  Let $\Lambda$ be a generic member in the relative interior of a $2$-dimensional cell of $\bclassns$ with tangencies of  type (3a) and (1a)  on legs adjacent to the same vertex. Assume no member of $\bclassns$ has a tangency point of type (4b). Then, the lifting partition of $\tclass$ is $(0,4,0,0)$.
\end{proposition}

\begin{proof} We pick the orbit representative for $\Lambda$ seen on the first row of~\autoref{tab:combClassificationGenericPerDim}. Let $P$ be the type (1a) tangency on the positive horizontal leg of $\Lambda$ complementing the two points mentioned in the statement. The bidegree of $\Gamma$ confirms that the tangencies on legs adjacent to $v_0$ lie in $\overline{(1,2)^{\vee}}$. The result then follows from items (ii) and  (iii) of \autoref{cor:leftAndRightmostLiftableMembers3aAnd1Vertical}: there are precisely two pairs of liftable members $\Lambda'$, one to the left and one to the right of $\Lambda$, each with the desired multiplicity.
\end{proof}

\begin{proposition}\label{pr:13c1}
  Let $\Lambda$ be a generic member in the relative interior of a $2$-dimensional cell of $\bclassns$ with one type (3c) and two type (1a) tangencies, one of which lies in the unique edge of $\Lambda$. Then, the lifting partition of $\tclass$ is $(0,4,0,0)$.\end{proposition}

\begin{proof} We let $\sC$ be the unique cell of $\bclassns$ containing $\Lambda$ in its relative interior. Up to $\Dn{4}$-symmetry, we may assume that the edge of $\Lambda$ has slope one and that the type (3c) tangency (labeled $P'$) lies on the negative horizontal leg of $\Lambda$. We let $e'$ be the horizontal edge of $\Gamma$ containing $P'$. We let $P$ be the diagonal tangency. By construction, it lies on an edge of $\Gamma$, which we label as $e$. The remaining tangency (1a) point, which we call $P''$ must be in either of the legs adjacent to $v_1$.  Both possibilities are depicted in~\autoref{fig:boundedDim2Cases}.

  \begin{figure}
  \includegraphics[scale=0.35]{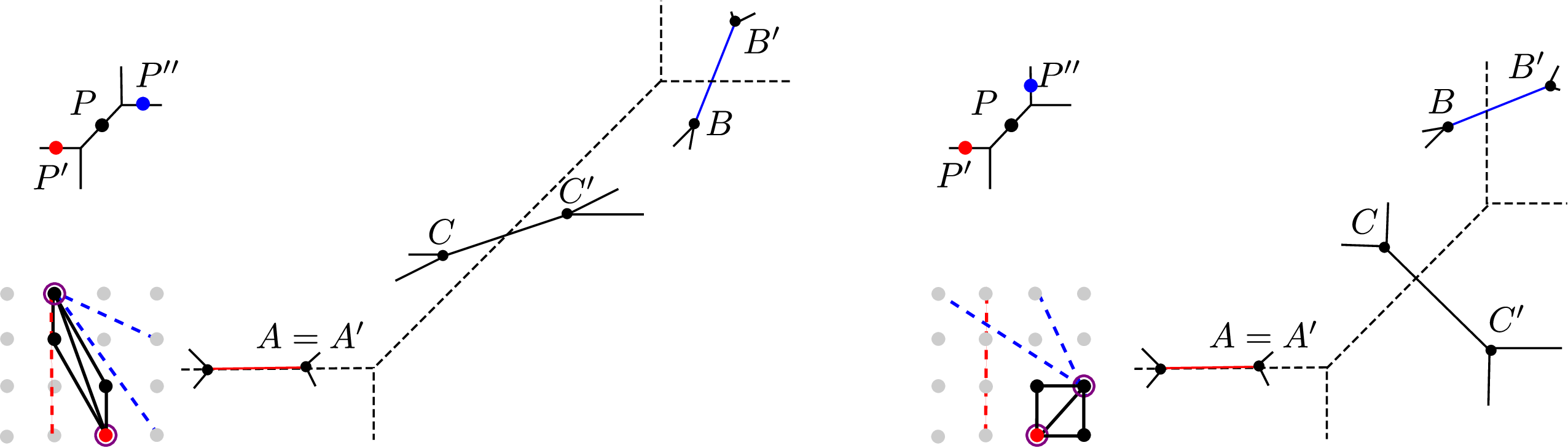}
  \caption{Possible distribution of tangencies between $\Lambda$ and $\Gamma$, the associated partial Newton subdivisions of $\sextic$, and relevant edges of $\Lambda$ and $\Gamma$, where $\Lambda$ is a generic member of a 2-dimensional cell of $\bclassns$ with a type (3c) horizontal tangency and two type (1a) tangencies, one of which is diagonal.\label{fig:boundedDim2Cases}}
\end{figure}

   The proof methods in both cases are similar to the ones used to establish 
   {Lemmas}~\ref{lm:case1Dim3} and~\ref{lm:case2Dim3}, except that in te present case the edge $e'$ is horizontal. 
   In order to carry out the construction of these lemmas in the present setting, we set the vertices $A$ and $A'$ to match  the rightmost endpoint of $e'$.
   We treat the horizontal and vertical nature of the tangency at $P''$ separately. In both cases, since $A=A'$ the parallelogram $\cP_0$ will be a degenerate: it becomes a horizontal segment in $\overline{(2,0)^{\vee}}$ aligned with $e'$.

   First, assume $P''$ lies on the positive vertical leg of $\Lambda$. Then, the slope of $e$ is $-1$. In what follows we construct a polygon $\cQ \subseteq |\bclassns|$ containing $\Lambda$ in its relative interior, with precisely four liftable members, all of which are vertices of $\cQ$ and have lifting  multiplicity two. We do so by using the same methods as in the proof of~\autoref{lm:case2Dim3} with some mild adjustments. In this situation, the inequalities from~\eqref{eq:inequalitiesCase2Dim3}    hold and the construction produces two parallelograms $\cP_0$ and $\cP_1$, analogous to those seen in~\autoref{fig:parallelogramHAndVDim3}, each recording the local moves of the corresponding vertex of $\Lambda$. In addition, there is only one possible relative position for the degenerate edge $e'$ with respect to $\cP_0$,  namely, the one labeled (1). Similarly, the edge $e$ can only have positions (1), (2) or (5) relative to $\cP_0$. As in the aforementioned lemma, the polygon $\cQ$ is obtained by pairing the points of $\cP_0$ and the polytope $\cP_1\cap \overline{(3,1)^{\vee}}$ that are diagonally aligned. By construction, $\cQ \subseteq |\bclassns|$ and contains $\Lambda$ in its relative interior.

In order to produce the four liftable members of $\cQ$, we proceed as in the lemma.  The four position for $v_0$ in $\cP_0$ that produced liftable members in the $3$-dimensional case now become two, and their local lifting multiplicity doubles since $P'$ has  type (3c). Each of these two choices has two possible locations to place the corresponding vertex $v_1$ on $\cP_1\cap \overline{(1,3)^{\vee}}$ and produce liftable members with total lifting multiplicity two. This yields the desired lifting partition for $\tclass$.

To conclude the proof, it remains to treat the case when  $P''$ lies in the positive horizontal leg of $\Lambda$. In this situation, the slope of $e$ becomes $-1/3$. Unlike the horizontal case analyzed earlier, we must argue in different ways depending on whether or not the edges $e$ and $e'$ are adjacent.

If they are not, the situation is similar to the one discussed in the proof of~\autoref{lm:case1Dim3} precisely because the corresponding inequalities listed in~\eqref{eq:orderingDim3} remain valid. Unlike what occured in that lemma,  the edge $e'$ can only be in position (1) relative to $\cP_0$. Since the lower vertex of any member of the polygon $\cQ$ lies in  $\cP_0$,  we can identify $\cQ$ with the location of the top-vertex $v_1'$ of the corresponding tritangents, yielding $\cQ:=\cP_1\cap \overline{(1,3)^{\vee}}$. For each $v_1'$,  the corresponding lower-vertex $v_0'$ is the unique point in $\cP_0$ diagonally aligned with it.

By construction,  we have that  $\cQ\subseteq |\bclassns|$ and its relative interior contains $\Lambda$. Furthermore, there are precisely   two locations for $v_0'$ in the  parallelogram $\cP_0$, namely, its endpoints, that produce liftable tritangents: each of them is paired with two possible points $v_1' \in \cQ$, both of which are vertices of $\cQ$.  The lifting multiplicity of each of these four tritangents equals  two due to the common  type (3c) tangency point $P$ and the type (2a), (4a) or (6a) nature of the remaining two tangencies, dependent on the  position of $e''$ relative to  $\cP_1$.

  On the contrary, if $e$ and $e'$ are adjacent, we have $A=C$. The conditions $C\prec_d P$ and $C\prec_v P$ imposed by $\Lambda$ ensure that we can move $v_0$ to the left preserving  $P''$ as a tangency point until we reach the point $A$. At this point, we acquire a type horizontal (3bb) tangency.  Moving $v_0$ past $A$, we obtain a trivalent member $\Lambda'$ on a $2$-dimensional cell of $\bclassns$ with a slope one edge, the same horizontal tangency at $P''$, a horizontal type (3a) tangency and a vertical type (1a) one on the boundary of $(1,2)^{\vee}$.
  The partial Newton subdivision of $\sextic$ imposed by $\Lambda'$ matches the one labeled by (I) or (II) in~\autoref{fig:3aAnd1VerticalLeftMove}. Both guarantee that we will not have a member of $\bclassns$ with a type (4b) tangency. \autoref{pr:3a11Dim2} confirms that $\tclass$ has the desired lifting partition.
\end{proof}

\begin{remark} In the absence of members with a type (4b) tangency, the proof of~\autoref{pr:13c1} confirms that if $|\bclassns|$ contains a tritangent with a type (3c) and two type (1a) tangencies, one of which is diagonal, the support of $\bclassns$ is obtained by attaching a (potentially empty) chain of segments to a $2$-dimensional polygon (namely, $\cQ$) along a vertes. Thus, such complexes admit a coarsest polyhedral complex structure.
\end{remark}

In the remainder, we discuss the three members in the first row of~\autoref{tab:combClassificationGenericPerDim} that have not yet been analyzed, namely,  the first one in entry $(2,(i))$ and those in entry $(2,(ii))$. To simplify the exposition, we treat together the first one and those among the second group that have a diagonal (3aa) tangency. For convenience, we choose alternative $\Dn{4}$-representatives to those in the table.
The next combinatorial lemma will be central to certify that the lifting partition of any
tritangent class containing such members is  $(0,4,0,0)$.

\begin{lemma}\label{lm:dim2Parallelogram}
  Assume that the boundary of the chamber $(1,1)^{\vee}$ in $\RR^2\smallsetminus \Gamma$  has two edges of slopes $\pm 2$ and $\pm 1/2$, labeled $e'$ and $e''$ respectively, as seen in~\autoref{fig:type1-1Parallelogram}. Then, for every pair of points $(P',P'')$ with $P'\in e'$, $P''\in e''$ satisfying  $P''\preceq_h P'$ and $P'\preceq_v P''$ we can find a unique $(1,1)$-curve $\Lambda$ with vertices in  $\overline{(1,1)^{\vee}}$ and containing the points $P'$ and $P''$ on its positive horizontal and vertical legs that is tritangent to $\Gamma$. 
  Furthermore,  $\Lambda$ has lifting multiplicity  $2\mult(\Lambda, P')\mult(\Lambda, P'')$.
\end{lemma}

\begin{proof} We consider the horizontal and vertical lines through $P'$ and $P''$, respectively, and let $Q$ be the unique intersection point between them. The construction of $\Lambda$ depends solely on the location of $Q$ relative to $\Gamma$. Uniqueness will follow by construction and the combinatorics of the chamber $(1,1)^{\vee}$.   The claim regarding the lifting multiplicity of $\Lambda$ will be settled with the aid of~\autoref{tab:LiftingMultiplicities} once the type of the tangency point in $\Lambda$ complementing $P'$ and $P''$ is determined.

  First, assume that $Q$ lies in the aforementioned chamber. The condition that $P'$ and $P''$ belong on the positive legs of $\Lambda$ forces the vertex $v_1$ to agree with $Q$ and $\Lambda$ to be trivalent with a slope one edge. In turn, the vertex $v_0$ of $\Lambda$ must be placed at   the intersection of the boundary of $(1,1)^{\vee}$  with the  ray $R$ with direction  $(-1,-1)$ emanating from the vertex $Q$. The bidegree of $\Gamma$ restricts the type of the remaining  tangency point of $\Lambda$ to (5a), (4a) or (6a) diagonal, or (3a).

  Next, suppose that $Q$ lies in the boundary of $(1,1)^{\vee}$ but outside the edges $e'$ and $e''$. Once again, the required location of $P'$ and $P''$ within $\Lambda$  ensures that  $Q$  agrees with the vertex $v_1$ of $\Lambda$. The condition that $v_0, v_1\in \overline{(1,1)^{\vee}}$ implies that $\Lambda$ is  $4$-valent. Indeed, our hypothesis on the slopes of the edges $e'$ and $e''$ combined with the bidegree of $\Gamma$ ensures that the stable intersection multiplicity  at the component of $\Lambda\cap \Gamma$ containing  $Q$ equals two. Thus, the extra tangency between $\Lambda$ and $\Gamma$ has  type (4a') or (6a'). A type (3a') tangecy is not allowed for convexity reasons.

  On the contrary, if $Q\in e\cup e'$, then $P'$ and $P''$ are either vertically or horizontally aligned. Since $e'$ and $e''$ are not adjacent in $\Gamma$ by~\autoref{fig:type1-1Parallelogram}, we know that $P'\neq P''$. Thus, up to  $\Dn{4}$-symmetry, we may assume that $P''=_v P'$ and $P''\prec_h P'$. In such situations, we have $v_1 = Q=P''$. In addition,   the edge $e''$ must have slope $1/2$ and, hence, the triangle $B^{\vee}$ has vertices $(1,1)$, $(1,2)$ and  $(0,3)$. From this if follows that  $A\prec_v B$, so  we know that $P''\neq B$.

  \begin{figure}
  \includegraphics[scale=0.3]{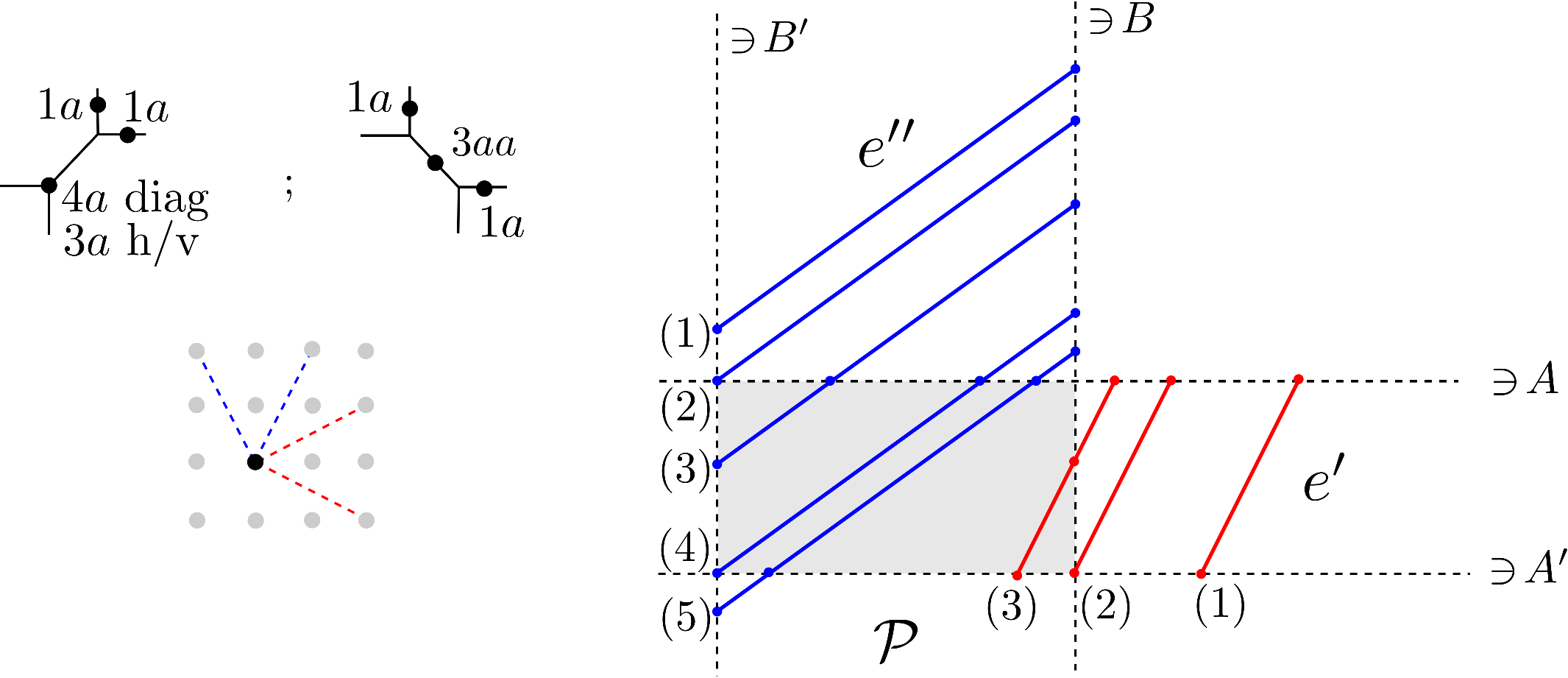}
\caption{From left to right: partial Newton subdivision of $f$ for two possible distribution of tangency points on a trivalent $\Lambda$ with $v_1\in \overline{(1,1)^{\vee}}$ and two type (1a) tangencies on its positive legs,  and positions of the (red) edge $e':=\overline{AA'}$ and (blue) edge $e'':=\overline{BB'}$ relative to the parallelogram  $\cP$ determined by horizontal and vertical lines through the endpoints of these edges.\label{fig:type1-1Parallelogram}}
  \end{figure}
  
  Two situations arise. If  $P''$ lies in the interior of $e''$, then $\Lambda$ will necessarily be  trivalent, with a slope one edge, and  the location of $v_0$ in  $\overline{(1,1)^{\vee}}$ is  uniquely determined as in the case when $Q\in (1,1)^{\vee}$.
  On the contrary, if $P''=B'$, then, the characterization of $B^{\vee}$ given above combined with the inequality  $A'\preceq_v P'$ and the fact that $A'$ is not adjacent to a horizontal edge, ensure that $(1,0)$ is not a vertex of the cell $(B')^{\vee}$.  Hence, the vertex $B'$ is adjacent to an edge of $\Gamma$ with slope one.
We let $B''$ to be other endpoint of this edge. Since $P''=B'$, there is only one option to obtain a valid tritangent with the restrictions in the statement, namely, to place $v_0$ at $B''$. The resulting third tangency point between $\Lambda$ and $\Gamma$ together with $P''$ will constitute a tangency of type (3d) or (3h), depending on whether the triangle $(B'')^{\vee}$ has $(0,1)$ or $(1,0)$ among its vertices.

Finally, assume that  $Q$ lies outside  $\overline{(1,1)^{\vee}}$. The convexity of $(1,1)^{\vee}$ and the bidegree of $\Gamma$ force the Newton subdivision of $\sextic$ to have an edge join $(1,1)$ and $(0,0)$. Its dual edge in $\Gamma$, which we label by $e$,  intersects both the vertical line through $P''$ and the horizontal one through $P'$. Since the intersection multiplicity at these two points equals one, there is only one possible location for the vertices of $\Lambda$ that makes this curve tritangent, namely, we must place them precisely at these two intersection points. The type of the remaining tangency point depends on whether or not any of its vertices lies in $e'\cup e''$.

If none of them is in this union, then this point has type (3aa) along the edge $e$. If only one of them does, say $v_0$, it follows that $v_0=P'=A'$ and $\Lambda$ has a type (3bb2) tangency containing  $P'$ but not $P''$. Finally, if both vertices are in $e'\cup e''$, we have that $v_0=A'=P'$, $v_1=B'=P''$ and all three tangency points on $\Lambda$ are part of a type (7) tangency.
\end{proof}

\begin{remark} The proof of the lemma shows that for generic choices of pairs $(P', P'')$ subject to the hypothesis required in the statement, $\Lambda$ is trivalent and the input tangency points  have type (1a). However, the slope of the unique edge of $\Lambda$ and the type of the complementing tangency varies. If $\Lambda$ has slope one, then $v_0 \in \Gamma$, $v_1\notin \Gamma$,  and we obtain the rightmost element in the $(2,(i))$-entry of~\autoref{tab:combClassificationGenericPerDim}. If $\Lambda$ has slope $-1$, then  $v_0,v_1\in \Gamma$ and we have a type (3aa) tangency. This distribution of tangencies is symmetric to the  one featured in the $(2,(ii))$-entry of the  table.
\end{remark}

The following result proves a partial description of the complexes $\bclassns$ with members satisfying the conditions of~\autoref{lm:dim2Parallelogram}.

\begin{corollary}\label{cor:tritangentClasses}
  All tritangents obtained in~\autoref{lm:dim2Parallelogram} belong to the same tritangent class of $\Gamma$.  Furthermore, the union of the supports of the cells of $\bclassns$ containing these members is a polytope in $\RR^3$,  isomorphic to the intersection of the rectangle $\cP$ seen in~\autoref{fig:type1-1Parallelogram}
and the halfspaces determined by the lines spanned by the edges $e$ and $e'$, respectively, that  contain the top-right vertex of $\cP$.
\end{corollary}

\begin{proof}
    Each tritangent $\Lambda$ obtained in the lemma  is uniquely determined by the point $Q$ defined in the proof of that result. The relative order restrictions between $P''$ and $P'$ fix the two halfspaces indicated in the statement.
  
  Moving the points $P'$ and $P''$ along the edges $e'$ and $e''$ subject to the  prescribed relative orders between them is equivalent to moving the point $Q$ continuously within $\cP$ and these two halfspaces. Thus, all tritangents built in this way lie in the same tritangent class of $\Gamma$ and yield elements of $|\bclassns|$. The point $v_1$ is in  $\cP$, and must be placed   below (or on) $e''$ and to the left of (or on) $e'$. 

  Conversely, any point $Q'$ in this intersection produces a valid tritangent in this class with $v_1=Q'$. Here, the points $P'$ and $P''$ are the intersections of $e'$ and $e''$ with the horizontal and vertical lines through $Q'$. By construction, both $P$ and $P''$  satisfy the required relative order restrictions. 
\end{proof}

\begin{remark} Unlike what happened in the presence of a type (4b) tangency, which was covered in~\autoref{lm:type4b_1}, the region $\cP\cap \overline{(1,1)^{\vee}}$ does not determine the subset of $|\bclassns|$ containing any of the tritangents $\Lambda$ described in~\autoref{cor:tritangentClasses}. The issue arises whenever the boundary of ${(1,1)^{\vee}}$ has a slope one edge intersecting the interior of $\cP$: the bottom left vertex of $\cP$ is not part of $\overline{(1,1)^{\vee}}$ but determines a tritangent in $|\bclassns|$ with a non-transverse diagonal tangency.
\end{remark}

Equipped with the findings of~\autoref{lm:dim2Parallelogram} and~\autoref{cor:tritangentClasses}, we determine the lifting partition for the corresponding tritangent classes:

\begin{proposition}\label{pr:dim2Parallelogram}
Let $\Lambda$ be a generic member in the relative interior of a $2$-dimensional cell of $\bclassns$ with either a type (3aa) diagonal tangency or with two type (1a) tangencies on different legs adjacent to the same vertex of $\Lambda$. Then, the lifting partition of $\tclass$ equals $(0,4,0,0)$. 
\end{proposition}

\begin{proof} We pick an orbit representative for $\Lambda$ that agrees with the ones seen on the left of~\autoref{fig:type1-1Parallelogram}. By \autoref{cor:tritangentClasses} we know that the union of the supporting cells of $\bclassns$ such tritangents in their relative interiors is (isomorphic to) the polygon $\cQ$  determined by  $\cP$ and the halfspaces induced by the edges $e'$ and $e''$ carrying the two type (1a) tangencies.

  By construction, any point of $\cQ$ which is not a vertex determines a member with a type (1a) tangency, which hence cannot lift.  This leaves us with between four and six potential liftable members in $\cQ$, namely, a subset of its vertices. The exact number depending on the relative position of $e'$ and $e''$ with respect to $\cP$. Up to symmetry, there are precisely three options for $e'$ and five for $e''$, leaving a total of 15 combinations to consider. Note that the slopes of these edges are uniquely determined when the positions are not (1): it will be $2$ for $e'$ and $1/2$ for $e''$.

  The polygon will have four vertices unless one of these edges is in position (3). An extra vertex (either $A'$ or $B'$) will be added for each vertex in position (3). We claim that none of the latter yields a liftable tritangent.  Indeed,   if $e'$ is in position (3), the point $A' \in \cQ$ will not produce a liftable member since the corresponding point $P''$ will have type (1a). The argument for the edge $e''$ is symmetric. Thus, we conclude that  at most four vertices of $\cQ$ will yield liftable members for all 15 combinations of relative positions. They correspond to the top-right vertex of $\cQ$, its two adjacent vertices in the boundary of the polygon, and  the bottom-left one.

  In what follows, we analyze each of these vertices and determine whether or not they lift. The proof of~\autoref{lm:dim2Parallelogram} confirms that the  lifting multiplicity of the corresponding tritangent equals 2 whenever the vertex lies in $\overline{(1,1)^{\vee}}$. In particular, we conclude that the top-right vertex of $\cQ$ always lift (and with the desired multiplicity), namely the top-right of $\cQ$  (corresponding to the choice $P'=A$ and $P''=B$ in the lemma).
  
  To finish our proof, we must analyze what happens with the remaining three. Fix any of these points, corresponding to picking $(P', P'') = (A',B), (A, B')$ or $(A',B')$.
  We let $\Lambda'$ be the corresponding tritangent. If this point lies in $\overline{(1,1)^{\vee}}$ it will also lift (and with multiplicity two) since $P'$ and $P''$ will have tangency types (2a), (4a), (6a) or be part of a type (3d) or (3h) one.

  On the contrary, if the chosen vertex of $\cQ$
  does not belong to $\overline{(1,1)^{\vee}}$, then the proof of the lemma confirms that the boundary of this set contains a slope -1 edge (labeled $e$) and both vertices of $\Lambda'$ must lie in $e$. Furthermore, $\Lambda'$ lifts (and with multiplicity 2) unless one of the points $P'$ or $P''$ lies in $e$, since $\Lambda'$ would acquire either a type (3bb2) diagonal or (7) tangency. In these cases, we can use the moving results from~\autoref{sec:technicalLemmasDim1-2} 
  to find an extra point in $|\bclassns|$ that produces a liftable tritangent, and with the desired multiplicity. By symmetry, it suffices to treat two cases.

  First, assume that only one of these vertices lies in $e$, say $P'$. By construction, we must have $P'=A'$. Furthermore, $A'$ is a vertex of $\Lambda'$ and it must be  part of a type (3bb2) tangency in $\Lambda'$ along $e$. Since $P'' \notin e$, we conclude that its tangency type is (2a). Then, moving $v_0'$ towards  $v_1'$ while fixing $P''$ produces a member of $\tclass$ with a (3aa) tangency along $e$, with a type (1a) tangency on the horizontal positive leg. \autoref{lm:MoveDown3-aAnd1H} then ensures that moving the vertex $v_0'$  of $\Lambda'$ away from $v_1'$ yields at most one liftable member, with multiplicity $2\mult(\Lambda', P'')=2$, as desired. Such members lie in a segment contained in  $|\bclassns|$ and adjacent to $\cQ$ at the starting vertex.

  Finally, assume both $P',P''\in e$. In this situation, $P'=A'$, $P''=B'$ and both vertices  of $\Lambda'$ are part of a type (7) tangency. Moving them  away from each other (as in the right of~\autoref{fig:3-aAnd1Horizontal}) produces exact one liftable tritangent $\Lambda''$ with a type (3cc) diagonal tangency, and two type (2a) ones, on its negative horizontal and vertical legs. Thus, its lifting multiplicity equals 2. By construction, the tritangents obtained in this way lie in a $2$-dimensional polytope in $|\bclassns|$ that meets $\cQ$ along the input vertex. 
\end{proof}

\begin{remark} The proof of~\autoref{pr:dim2Parallelogram} confirms that in the presence of tritangents with the configurations depicted in~\autoref{fig:type1-1Parallelogram}, the corresponding complex $\bclassns$ is supported in the union of at most three polygons, one of which is $\cQ$. The remaining ones occur only when $(0,0)$ and $(1,1)$ are adjacent in the Newton subdivision of $\sextic$, yielding a slope -1 edge $e$ in $\Gamma$. The extra polygons will be present whenever the set $e\cap \{A',B'\}$ is non-empty. The corresponding tritangents will be  trivalent with a slope $-1$ edge intersecting non-trivially with $e$ but having at least one vertex outside $e$. Each of them is adjacent to $\cQ$ along a fixed edge, namely, its left vertical or bottom horizontal one.
\end{remark}

Our last statement addresses the remaining $2$-dimensional tritangent classes, namely, those having a generic trivalent member in $|\bclassns|$ with a non-transverse tangency whose type is not (3aa):

\begin{proposition}\label{pr:13acOr3cc1sameOrNot} Let $\Lambda$ be a generic member in the relative interior of a $2$-dimensional cell of $\bclassns$ with a non-transverse diagonal tangency. Suppose at least one of the vertices of $\Lambda$ is not in $\Gamma$. Then, the lifting partition of $\tclass$ equals $(0,4,0,0)$.
\end{proposition}

\begin{proof} Looking at the sample generic members of a $2$-dimensional cell of $\bclass$ listed in~\autoref{tab:combClassificationGenericPerDim}, we know that the two remaining tangencies of $\Lambda$ are of type (1a) and each of them lies in a leg adjacent to a different vertex of $\Lambda$. We let $P'$ be the tangency point along the diagonal.
  
  In order to use the results of~\autoref{sec:technicalLemmasDim1-2}, we choose a representative of $\Lambda$ that has an edge of slope $-1$, a tangency on the positive horizontal leg (labeled $P''$), and  $v_1\notin \Gamma$. The remaining tangency, which we call $P$, can be either horizontal or vertical, but it must lie in  a leg  of $\Lambda$  adjacent to $v_1$.

  The assumptions on $\Lambda$ ensure that  $P$ has type either (3ac) or (3cc). We prove the statement by discussing both cases separately, starting with (3ac).   The partial Newton subdivision of $\sextic$ imposed by the tuple $(\Lambda, P, P', P'')$ matches those from~\autoref{fig:3caAnd1Horizontal}, with the caveat that if  subdivision  (I) occurs,  then $(2,3)$ is a vertex of the  triangle $(C')^{\vee}$, because the  tangency at $P''$ is  horizontal.

  Using the notations established in the figure, we see that if $C'\prec_v B'$ this subdivision is incompatible with the existence of a trivalent tritangent equivalent to $\Lambda$ with a slope -1 edge and a type (4b) tangency. In particular, this ensures that  \autoref{lm:MoveUp3caAnd1H} (ii) produces a member of $|\bclassns|$ satisfying the conditions of~\autoref{pr:3a11Dim2}, thus confirming the expected lifting partition for $\tclass$.

  On the contrary, assume  $B'\preceq_v C'$. We argue in two ways, depending on the nature of $P$. If the tangency at $P$ is vertical,~\autoref{fig:3caAnd1Horizontal} (II)  confirms that the edges of $\Gamma$ containing $P$ and $P'$ are not adjacent. Since $P$ has type (1a), it follows from this observation  that we can move $v_0$ and $v_1$  independently away or towards each other to produce four liftable members, each with multiplicity two. Each vertex has two  valid positions,  determined by 
  {Lemmas}~\ref{lm:MoveDownUpRight3-cAnd1V} (i) and (iii), \ref{lm:MoveDown3-aAnd1H} and~\ref{lm:MoveUp3caAnd1H} (i).

  Next, assume the tangency at $P$ is horizontal, as in~\autoref{fig:3caAnd1Horizontal} (I). If the edges of $\Gamma$ containing $P$ and $P'$ are adjacent,~\autoref{lm:MoveDownUpRight3-cAnd1V} (ii) applied to $v_1$ and $P$  ensures we can move $v_1$ towards $v_0$ to find a generic member of a $2$-dimensional cell of $\bclass$ with type (3aa) tangency. Thus, \autoref{pr:dim2Parallelogram} confirms that we obtain the desired lifting partition.

  On the contrary, if the edges are not adjacent,~\autoref{lm:MoveDownUpRight3-cAnd1V} (i) and (iii) will ensure precisely two positions for $v_1$ (both in $(2,3)^{\vee}$) that can produce potentially liftable tritangents (where the local lifting multiplicity of $P$ is one). To guarantee they lift, we must determine valid positions for $v_0$, either moving it towards or away from $v_1$. The information collected so far on the Newton subdivision of $\sextic$ ensures that there are two such choices for $v_0$ thanks to 
  {Lemmas}~\ref{lm:MoveDown3-aAnd1H} and~\ref{lm:MoveUp3caAnd1H} (i). Thus, the lifting partition of $\tclass$ is as expected.

  It remains to treat the case when $P'$ has type (3cc). Once again, we invoke~\autoref{lm:MoveDownUpRight3-cAnd1V}. First, we assume that the edge of $\Gamma$ containing $P'$ is adjacent to at least one of the edges containing $P$ or $P''$.  In this situation,   the corresponding vertex of $\Lambda$ can be moved toward the other one (as in the right of~\autoref{fig:3-cAnd1Vertical}) following item (ii) of the lemma, to produce a member in a $2$-dimensional cell of $\bclassns$ with a type (3ac) tangency. Our earlier discussion then confirms the statement.

 Finally, if no adjacency occurs, item (i) and (iii) of the same lemma ensures  that we can move $v_0$ and $v_1$ independently towards and away from each other to produce precisely four pairs of locations for the pair $(v_0,v_1)$ that yield liftable members. In all cases, $P'$ remains a type (3cc) tangency, hence confirming that they all have lifting  multiplicity two. This concludes our proof.
\end{proof}

\section{Proof of the refined lifting partition theorem when $\dim \bclassns \leq 1$}\label{sec:proof-lift-part-1d-or-less}

In this section we complete the proof of~\autoref{thm:refinedLiftingPartition} by discussing the lower dimensional cases. We start with the simplest setting, namely, when $\bclassns$ is discrete. In these situations, the path-connectivity established in~\autoref{thm:pathConnectednessOfNSComplexes} ensures that this set is  a point.

\begin{theorem}\label{thm:partitionDim0}
  If $\dim \bclassns = 0$, then the lifting partition of $\tclass$ equals $(0,0,0,1)$.
\end{theorem}

\begin{proof}  We let $\Lambda$ be the  tritangent corresponding to the single point supporting $\bclassns$. Note that $\Lambda$ can be trivalent or $4$-valent, so we must analyze both cases separately.

  \autoref{tab:combClassificationGenericPerDim} lists all options, up to $\Dn{4}$-symmetry. Note that  if  $\Lambda$ is trivalent with a (4a)/(6a) diagonal tangency $P$ at a vertex, the two adjacent legs contain no  tangency points, so the quantity $\mu$ from \eqref{eq:mu} becomes 2. Hence, we have  $\mult(\Lambda,P)=2$ by~\autoref{tab:LiftingMultiplicities}. In turn, the same table confirms that  the remaining tangency types present in a trivalent $\Lambda$ have local lifting multiplicity $2^{r/2}$, where $r$ denotes the (stable) intersection multiplicity of the corresponding connected component of $\Lambda\cap \Gamma$. Combining these two facts confirms that in the trivalent case, $\Lambda$ lifts with  multiplicity eight.

On the contrary, if $\Lambda$ is $4$-valent,~\autoref{lm:4valentSingleton} yields only one  possible distribution of tangencies up to $\Dn{4}$-symmetry, depicted as the last entry in~\autoref{thm:dim210}. Using~\autoref{tab:LiftingMultiplicities} we conclude that the lifting multiplicity for this member  is also eight since the quantity $\mu$ has value 2 as well.
\end{proof}

We devote the rest of this section to discuss the one-dimensional case. As it occured in the $2$-dimensional case, if $\dim \bclassns = 1$, the corresponding tritangent classes have two possible lifting partitions. This is the content of the following statement:

 \begin{theorem}\label{thm:partitionDim1}
  If $\dim \bclassns = 1$, then the lifting partition of $\tclass$ equals $(0,2,1,0)$ or $(0,0,2,0)$. Furthermore, the first partition occurs if, and only if, $|\bclassns|$ has a member with a type (4b) tangency.
\end{theorem}

\begin{proof}
The middle row of~\autoref{tab:combClassificationGenericPerDim} shows the distribution of tangencies on a generic member of a maximal  cell of $\bclassns$. There are precisely 28 possible non-special representatives. We group these options into clusters, as in the proof of~\autoref{thm:partitionDim2}.  The result is a direct consequence of 
{Propositions}~\ref{pr:3fOr3a3c}  through~\ref{pr:3c3acOr3cc13csameOrNot} below.
\end{proof}

Several of the sample members seen in the first and middle row of~\autoref{tab:combClassificationGenericPerDim} impose similar partial Newton subdivisions for $\sextic$, allowing us to adapt the arguments from various Propositions presented in~\autoref{sec:proof-lift-part-2d-or-less} to the $1$-dimensional case. Before we do so, we treat those with different behavior, namely, the first two members in the $(1,(iv))$-entry of the table:

\begin{proposition}\label{pr:3fOr3a3c} Let $\Lambda$ be a generic member in the relative interior of a $1$-dimensional cell of $\bclassns$ with either a type (3f) tangency or a type (3c) and a type (3a) tangency on different legs of $\Lambda$ in the boundary of the same chamber of  $\RR^2\smallsetminus \Gamma$. Then, the lifting partition of $\tclass$ is $(0,0,2,0)$.
\end{proposition}

\begin{proof} As usual, we assume that $\Lambda$ has a slope one edge and $v_0$ is part of either the type (3f) tangency or the (3a) one, which we place on the negative horizontal leg of $\Lambda$. Accordingly, we let $C$ be either the vertex of $\Gamma$ adjacent to $v_0$ along a slope one edge or the rightmost endpoint of the horizontal edge of $\Gamma$ containing $v_0$.

In both cases, the tangency point of $\Lambda$ complementing those in the component of $\Gamma\cap \Lambda$ containing $v_0$ is necessarily of type (1a). We label it $P$. The bidegree of $\Gamma$ restricts the location of $P$ to  the positive horizontal leg of $\Lambda$. We let $e$ be the edge of $\Lambda$ containing $P$. Its slope has two possible values, namely, 2 or 2/3.  We let $B$ and $B'$ be the bottom and top vertices of $e$.

The remaining tangencies on $\Lambda$ fix the position of the vertex $v_0$. Moving $v_1$ away from $v_0$ produces a unique liftable member in $|\bclassns|$, where $P$ becomes a tangency of type (2a) (if $v_0\prec_d B'$), (6a) (if $v_0=_d B'$) or (4a) (if $B'\prec_d v_0$). In all cases, its lifting multiplicity is $4$. Similarly, since  $C\prec_v B$ and $C\preceq_d B$, moving $v_1$ towards $v_0$ produces a unique liftable member,  with a type (2a) tangency at $B$. Once more, the lifting multiplicity of this member equals $4$. 
\end{proof}

Our first lemma is analogous to~\autoref{lm:type4b_1} and determines lifting conditions for  members of a cell of $\bclassns$  with a  type (4b) tangency in its relative interior.

\begin{lemma}\label{lm:type4b3c}
  Let $\Lambda$ be a trivalent tritangent containing two tangency points, one of type (4b) and one of type (3c).  Then, there exists a unique segment in $|\bclassns|$, whose relative interior contains $\Lambda$, with only two liftable members, namely, its endpoints. Furthermore, these tritangents have two lifts each and a vertex carrying a tangency of  type (4b'), (6b'), (5b) or (6b).
\end{lemma}

\begin{proof} The result follows by applying the same strategy as in the proof of~\autoref{lm:type4b_1}, where the role of the points $A$ and $A'$ is replaced by the leftmost vertex of the horizontal edge of $\Gamma$ carrying the type (3c) tangency.  Thus, the parallelogram $\cP$ degenerates to a segment, and the four liftable members are paired to become  the two endpoints of the segment. Their lifting multiplicity is two, due to the presence of a common  type (3c) tangency. The  type for the tangency at $v_0$ of each such member agrees with those predicted by the aforementioned result in the $2$-dimensional setting.
  \end{proof}

The next proposition follows naturally from the previous lemma.

\begin{proposition}\label{pr:type4b3c} If $\dim \bclassns=1$ and contains a member with a (4b) and a (3c) tangency, then, the lifting partition of $\tclass$ equals $(0,2,1,0)$.
\end{proposition}

\begin{proof} The proof is similar to that of~\autoref{pr:type4b1}. We pick the orbit representative for $\Lambda$ with a slope one edge, and the type $(4b)$ tangency at $v_0$. The relevant cells in the Newton subdivision of $\sextic$ are collected in 
  {Figures}~\ref{fig:3aAnd1VerticalLeftMove} (III) and~\ref{fig:3aAnd1VerticalRightMove} (IV).

Combining~\autoref{cor:leftAndRightmostLiftableMembers3aAnd1Vertical} (iii) and~\autoref{lm:type4b3c}  we see that  $|\bclassns|$ is a union of two adjacent segments recording the location of $v_0$. Furthermore, $|\bclassns|$ is contained in the  union of  two adjacent segments of $\Gamma$: namely the ones with slopes $-1/3$ and $0$ in the boundary of $(1,3)^{\vee}$. The tangency point $P$ appears in all members of $|\bclassns|$ and its type is always (3c). Thus, the lifting partition $(4,2,0,0)$ from~\autoref{pr:type4b1} turns into $(0,2,1,0)$.
\end{proof}

In what follows, we discuss the remaining non-special tritangencies seen in the middle row of~\autoref{tab:combClassificationGenericPerDim}. All of them will yield tritangent classes with lifting partition $(0,0,2,0)$.

\begin{proposition}\label{pr:3a13cDim1}
  Let $\Lambda$ be a generic member in the relative interior of a $1$-dimensional cell of $\bclassns$ with tangencies of  type (3a) and  (1a)  on legs adjacent to the same vertex.  Assume no member of $\bclassns$ has a type (4b) tangency. Then, the
  the lifting partition of $\tclass$ is $(0,0,2,0)$.
\end{proposition}

\begin{proof} As in the proof of~\autoref{pr:3a11Dim2}, we pick the second tritangent in the $(1,(iii))$-entry of the table as our orbit representative for $\Lambda$. Let $P$ be the type (3c) tangency on the positive horizontal leg of $\Lambda$ complementing the two tangencies from the statement. 
  {Figures}~\ref{fig:3aAnd1VerticalLeftMove} (excluding (III)) and~\ref{fig:3aAnd1VerticalRightMove} depicts all possible  partial Newton subdivisions for $\sextic$ imposed by the existence of such $\Lambda$. The proof is a direct consequence of~\autoref{cor:leftAndRightmostLiftableMembers3aAnd1Vertical} (ii) and (iii) since the type (3a) and (3c) tangencies are not on the boundary of the same chamber of $\RR^2\smallsetminus \Gamma$. Indeed, there are precisely two liftable members on $|\bclassns|$, i.e., one to the left and one to the right of $\Lambda$. Each of them has four lifts.
\end{proof}

\begin{proposition}\label{pr:3c3c1}
  Let $\Lambda$ be a generic member in the relative interior of a $1$-dimensional cell of $\bclassns$ with two type (3c) tangencies on legs of $\Lambda$ and one of type (1a) in the unique edge of $\Lambda$. Then, the lifting partition of $\tclass$ is $(0,0,2,0)$.\end{proposition}

\begin{proof} 
  We pick the two possible orbit representatives from $\Lambda$ seen in the bottom row of the entry labeled $(1,(iii))$ in~\autoref{tab:combClassificationGenericPerDim}. The result follows by using the same strategy as in the proof of~\autoref{pr:13c1}, setting $B=B'$ in~\autoref{fig:boundedDim2Cases}. In this case, both parallelograms $\cP_0$ and $\cP_1$ will be degenerate. Whenever the edge of $\Gamma$ containing the type (1a) tangency is not adjacent to any of the edges containing the type (3c) ones, we conclude that $|\bclassns|$ can be identified with $\cP_0$, via the location of the vertex $v_0$ of each member.  Thus, the four liftable members from the $2$-dimensional case become two liftable ones, each with  multiplicity four. They are located at the endpoints of $\cP_0$.

  On the contrary, if some adjacency occurs, then the edge containing the diagonal tangency has slope $1/3$. Moving the vertices of $\Lambda$ to the left or  right past the common vertex of the adjacent edges produces a member satisfying the same conditions as in~\autoref{pr:3a13cDim1}. Note that  the absence of a member of $|\bclassns|$ with a type (4b) tangency is ensured because the partial Newton subdivision of $\sextic$ obtained by modifying the one in the left of~\autoref{fig:boundedDim2Cases}
  is incompatible with those in the $\Dn{4}$-orbit of the subdivision from~\autoref{fig:3aAnd1VerticalLeftMove} (III).
Thus, we obtain the desired lifting partition.
\end{proof}

The next two lemmas are central to determine the lifting partitions of tritangent classes with $\dim \bclassns = 1$ containing any of the remaining members from  the middle row of~\autoref{tab:combClassificationGenericPerDim}.  The first one is needed to treat the last tritangent listed in entry $(1,(iii))$ of the table. The second one includes an analog of~\autoref{lm:dim2Parallelogram} in dimension one, and reflects the increased combinatorial complexity of the underlying tritangent classes in low dimension.

\begin{lemma}\label{lm:1aAnd4aetcx2} 
  Assume that the points $(1,1)$ and $(2,2)$ are joined by an edge in the Newton subdivision of $\sextic$, and let $e$ be the corresponding dual edge in $\Gamma$. Then:
  \begin{enumerate}[(i)]
  \item Each point $P$ in  $e$ is contained in a unique trivalent tritangent with a slope one edge,  having a diagonal tangency at $P$, whose lower and upper vertex lie in the boundaries of the chambers $(1,1)^{\vee}$ and $(2,2)^{\vee}$, respectively, and are outside the relative interior of $e$.
\item All such tritangents belong to  the same tritangent class of $\Gamma$. Furthermore, they all lie in a segment contained in the support of $\bclassns$, which is isomorphic to $e$. No member corresponding to a point in the interior of $e$ can lift.
\item When $P$ is an endpoint of $e$, the corresponding tritangent from (i) lifts (and with  multiplicity 4) if, and only if, the corresponding dual cell is a right triangle. In this situation, $P$ is contained in a single $1$-dimensional cell of $\bclassns$.
  \item If a given endpoint $P$ of $e$ corresponds to a non-liftable tritangent, then $P$ lies in the boundary of precisely one more $1$-dimensional cell of $\bclassns$. Furthermore, the latter contains a trivalent member with two tangencies on adjacent legs of $\Lambda$  and of types (3c) and (1a).
  \end{enumerate}
\end{lemma}

\begin{proof} We let $C$ and $C'$ be the left- and rightmost endpoints of $e$, respectively. In order to prove (i), we fix a point $P\in e$ and consider the slope one line $L$ through $P$. This line meets the union of the boundaries  of $(1,1)^{\vee}$ and $(2,2)^{\vee}$ at three points, counted with multiplicity, namely, the point $P$, and two extra points $Q, Q'$ outside the interior of $e$, with  $Q$ in $\overline{(1,1)^{\vee}}$ and  $Q'\in  \overline{(2,2)^{\vee}}$. 

  The convexity of the chambers $(1,1)^{\vee}$ and $(2,2)^{\vee}$ and the bidegree of $\Gamma$ imply that the points $Q$ and $Q'$ must lie on an edge that is either horizontal, vertical or has slope $-1$. Thus, setting $v_0=Q$ and $v_1=Q'$ will produce the desired tritangent member of $|\bclassns|$, establishing the existence statement in (i). Uniqueness follows by construction. 

  Next, we prove (ii). Moving the point $P$ along $e$ will shift both points $Q$ and $Q'$ in a continuous fashion along the cycles of the graph $\Gamma$ containing them. Thus, all members obtained in this way lie in the same tritangent class. This family of members is parameterized by $e$ and is contained in the support of $\bclassns$.
  
  Note that, when $P$ lies in the interior of $e$, the points $P,Q, Q'$ are all distinct and
  the tangencies complementing $P$ will be of types (5a), (4a)/(6a) diagonal or (3a). Furthermore, the tangency at $P$ will be of type (1a). In particular, none of these members lift.

  To conclude, we address (iii) and (iv) together. We let $\Lambda$ be the tritangent corresponding to $P$. Note that  if $P$ is an endpoint of $e$ we may have either $P=Q$ or $P=Q'$ (but not both) depending on the nature of the dual cell to $P$ in the Newton subdivision of $\sextic$, as we now explain. Exploiting the action of $\Dn{4}$,  it suffices to discuss the case when $P=C$ and the cell $C^{\vee}$ is a triangle with vertices $(1,1)$, $(2,2)$ and $(i,i+1)$ for some  $i\in \{1,2\}$.

  If $i=1$ then $C^{\vee}$ is a right triangle. By construction, the points $P,Q$ and $Q'$ are all distinct, $P$ is a tangency of type (2a) and the ones at $Q, Q'$ are as in the case when $P\in e^{\circ}$. Thus, $\Lambda$ has  lifting multiplicity 4. In turn, the uniqueness of the $\Lambda$  guaranteed by (i), and the local movements of type (2a) tangencies seen in~\autoref{fig:localMoves},  confirm that $P$ lies in a single $1$-dimensional cell of $\bclassns$.

  On the contrary, if $i=2$, then we have $P=Q'$ and we obtain a type (3bb)  horizontal tangency containing  $P$, whereas $Q\in \overline{(1,1)^{\vee}}$ and $Q\notin e$.  This tritangent does not lift, per our findings from~\autoref{tab:LiftingMultiplicities}. Thus, our claim (iii) follows.

  Notice that the type (3bb) tangency  forces $v_1$ to move to the left (and $v_0$ accordingly along the boundary of $(1,1)^{\vee}$) if we were to obtain a member of $|\bclassns|$ outside $e$. This new tritangent, and the $1$-cell containing it  satisfies the properties listed in (iv).
\end{proof}

\begin{lemma}\label{lm:dim1Parallelogram}
  Suppose that  the vertex $(1,1)$ in the Newton subdivision of $\sextic$ is adjacent to either $(0,3)$ or $(2,3)$, and that it forms a triangle with two vertices of the form $(2,i)$ and $(2,i+1)$ for some $i\in \{0,1,2\}$, as in~\autoref{fig:type1-3cParallelogramDim1}. Let $e''$ and $A$ be the edge and vertex of $\Gamma$ dual to these two cells. Let $e'$ be the unique horizontal edge of $\Gamma$ adjacent to $A$ and fix $P'$ as the midpoint of $e'$.  Then:

  \begin{enumerate}[(i)]
    \item For every point $P'' \in e''$   satisfying  $P''\preceq_h A$ and $A\preceq_v P''$ we can find a unique tritangent $\Lambda$ to $\Gamma$ with vertices in  $\overline{(1,1)^{\vee}}$,  having $P''$ and $P'$ as tangencies points, located on the positive vertical and horizontal legs of $\Lambda$, respectively. Furthermore,  $\Lambda$ has lifting multiplicity  $4\mult(\Lambda, P'')$.
    \item   All these members belong to the same tritangent class. Moreover, they lie in a unique segment $\cQ$ contained in the support of  $\bclassns$ which is isomorphic to a horizontal segment.
    \item No interior point of $\cQ$ corresponds to a liftable tritangent.
    \item The rightmost vertex of $\cQ$  lifts (and with multiplicity four) whenever $A\notin e''$. 
      \item Whenever the rightmost vertex of $\cQ$ does not lift, then this point lies in another segment contained in $|\bclassns|$, adjacent to $\cQ$ along this vertex, whose relative interior contains a trivalent member with a type (1a) diagonal tangency and two extra tangencies (one of which has type (3a)) on legs  adjacent to different vertices. In addition, $v_0$ lies in the boundary of the chamber $(1,1)^{\vee}$.
      \item The leftmost vertex of $\cQ$ does not lift if, and only if, the edge $(e'')^{\vee}$ forms a triangle with the vertex $(0,0)$ in the Newton subdivision of $\sextic$.
      \item Whenever the leftmost vertex of $\cQ$ does not lift, we can find precisely one more segment in the support of $\bclassns$ that is 
adjacent to $\cQ$ along this vertex. Furthermore, this segment has exactly one liftable member (and with  multiplicity four), namely its second endpoint. 
      \item Whenever a vertex of $\cQ$ corresponds to a liftable member, then this point lies in precisely one top-dimensional cell of  $\bclassns$.
  \end{enumerate}
\end{lemma}

\begin{proof}
  We let $B'$ and $B$ be the left- and rightmost endpoints of $e''$ and set $A=A'$ as the leftmost vertex of the horizontal edge $e'$.     We must consider two situations, depending on whether or not $A$ and $B$ agree. Note that $B\preceq_h A$, $A\preceq_v B$, and both inequalities are strict if $A\neq B$.

  To establish (i) and (ii), we proceed as in the proofs of~\autoref{lm:dim2Parallelogram} and~\autoref{cor:tritangentClasses}, where the midpoint of $e'$ (called $P'$) plays the role that the top-vertex $A$ of $e'$ had in the aformentioned statements. We building a degenerate parallelogram $\cP$ from the points $B, B'$ and $A=A'$ as in~\autoref{fig:type1-3cParallelogramDim1}. Let $\cQ$ be the segment obtained as the intersection of  $\cP$ and the halfspace defined by the affine span of $e''$ that intersects nontriviality with the chamber $(1,1)^{\vee}$. Each point of $\cQ$ determines a unique member $\Lambda$ in $|\bclassns|$  subject to the imposed restrictions on its tangencies and vertices. Since $\mult(\Lambda,P')=2$, and the tangency point in $\Lambda$ complementing  $P'$ and $P''$ has local lifting multiplicity 2, $\Lambda$ has the expected lifting multiplicity.

\begin{figure}[t]
  \includegraphics[scale=0.3]{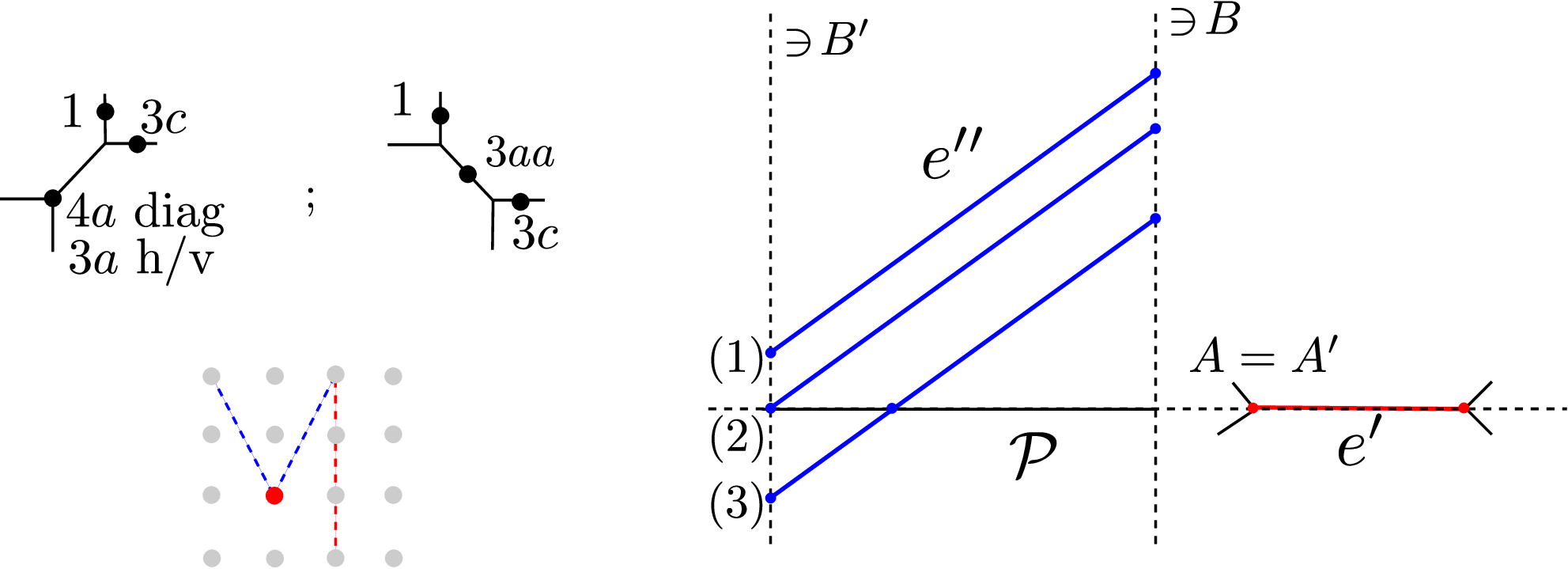}
  \caption{From left to right: partial Newton subdivision of $\sextic$ for two possible distribution of tangency points on a trivalent $\Lambda$ with $v_0,v_1\in \overline{(1,1)^{\vee}}$, having a type (1a) vertical tangency and a type (3c) horizontal one on its positive legs,  and positions of the (red) edge $e'$ and (blue) edge $e'':=\overline{BB'}$ relative to the segment  $\cP$ determined by horizontal and vertical lines through the endpoints of these edges.\label{fig:type1-3cParallelogramDim1}}\end{figure}

The same construction confirms (iii) since for any interior point of $\cQ$, $P''$ will be a type (1a) tangency.
The proof of both (iv) and (v) follows from our previous discussion. Indeed, if $A\neq B$, the rightmost vertex of $\cP$  is contained in $\cQ$ and produces a liftable member with  multiplicity 4 since $P''=B$ becomes a type (2a) tangency, $P'$  remains of type (3c), and the complementing tangency has local lifting multiplicity two. In turn, if $A=B$, we will obtain a type (3bb) horizontal tangency along $e$, which cannot lift. In addition, the point $v_0$ will remain in the boundary of the chamber $(1,1)^{\vee}$. Moving $v_1$ horizontally past $A$ (while keeping $v_0$ in the same cycle of $\Gamma$ dual to $(1,1)$) we obtain a member in a unique top-dimensional cell of $\bclassns$ with the desired distribution of tangencies.

Next, we verify (vi) and (vii). The arguments used in the proof of~\autoref{pr:dim2Parallelogram} confirm that the leftmost vertex of $\cQ$ lifts (and with multiplicity four)  unless  $e''$ is in position (1) relative to the segment $\cP$, its slope equals  $-1/2$ and $e''$ is adjacent to a slope $-1$ edge in the boundary of $(1,1)^{\vee}$. In turn, for this exceptional case, we have that $\cP=\cQ$ and the member corresponding to its leftmost vertex acquires a diagonal type (3bb2) tangency containing $v_1=B'$.~\autoref{tab:LiftingMultiplicities} confirms this tritangent will not lift.

However, by~\autoref{lm:MoveDown3-aAnd1H}, moving $v_1$ away from $v_0$ while preserving the  tangency at $P'$ produces a segment in $|\bclassns|$ adjacent to $\cQ$ along $B'$ and  with only one liftable member (with multiplicity four). Indeed, for all members of this segment, the original (3bb2) tangency becomes two tangencies: one of type (1a) or (2a)  on the negative horizontal leg, and one  of type (3ac) along an edge of $\Gamma$ of slope $-1$ in the boundary of $(1,1)^{\vee}$ with $v_1\notin \Gamma$. The type (2a) tangency corresponds to the second endpoint  of this segment. Note that, by construction, this point does not belong to any other top-dimensional cell of $\bclassns$.

To conclude, we establish item (viii). The claim will follow after determining the possible local movements of liftable members of $\cQ$. We treat both endpoints of $\cQ$ separately.

If the member associated to the rightmost vertex of $\cQ$ lifts then $P''=B$ is a type (2a) vertical tangency. The local moves of $v_1$ imposed by this tangency (seen in~\autoref{fig:localMoves}) and $P'$ prevent this vertex from lying in any other $1$-dimensional cell of $\bclassns$.

Finally if the tritangent $\Lambda'$ associated to the leftmost vertex of $\cQ$ lifts, we can adapt the proof of~\autoref{lm:dim2Parallelogram} to determine  the distribution of tangencies on $\Lambda'$. In this situation,~\autoref{fig:type1-3cParallelogramDim1} confirms that the edge $e''$ has only three possible positions relative to the degenerate parallelogram $\cP$, namely (1), (2) or (3). In cases (1) or (3), then $\Lambda'$ acquires a vertical tangency of  type (2a) or (4a) along $e''$. Once again, the local movements for these two types imply the desired claim.

On the contrary, for position (2), we know that $\Lambda'$ has a type (3c) tangency on its positive  horizontal leg and a diagonal tangency of type (3h) or (3d). The local movements imposed by these two possible diagonal tangencies will bring $\Lambda'$ back to $\cQ$ or turn $\Lambda'$ into a special tangency (with two tangencies on its positive horizontal leg). Thus there is precisely one cell of $\bclassns$ containing the leftmost vertex of $\cQ$ whenever the latter lifts.
\end{proof}

The previous two lemmas allow us to characterize the support of  $\bclassns$ when the partial Newton subdivision of $\sextic$ is the one prescribed by these two statements:

\begin{corollary}\label{cor:1And4ax2Class}  Let $\bclassns$ be 1-dimensional and let $\Lambda$ be a generic member in the relative interior of a top-dimensional cell of $\bclassns$ with a type (1a) diagonal tangency and no type (3c) ones.  Then, $\bclassns$ is supported on  a chain of segments, and its   only liftable members are its two unique $1$-valent vertices. Furthermore, their lifting multiplicities equal four.
\end{corollary}

\begin{proof}
    By symmetry, we may assume the unique edge of $\Lambda$ has slope one. We let $P$ be the type (1a) tangency from the statement and let $e$ be the edge of $\Gamma$ containing $P$. In this situation, the hypotheses on $\Lambda$ ensure the presence of an edge joining $(1,1)$ and $(2,2)$ in the Newton subdivision of $\sextic$ dual to $e$. By construction, the vertices $v_0$ and $v_1$  lie in the boundary of the chambers $(1,1)^{\vee}$ and $(2,2)^{\vee}$, respectively. This is precisely the setting of~\autoref{lm:1aAnd4aetcx2}.

    By item (ii) of the lemma, the support of $\bclassns$ contains a segment that is  isomorphic to $e$, and whose relative interior contains $\Lambda$.
  We let $C$ and $C'$ be its two endpoints, seen from left to right.
  By items (iii) and (iv) of the aforementioned lemma, the structure of $\bclassns$ depends on the nature of the  dual cells $C^{\vee}$ and $(C')^{\vee}$  in the Newton subdivision of $\sextic$.

  In what follows we show that the support of $\bclassns$ is obtained  by attaching to $e$ a (possible empty) chain of segments  along each vertex.
  Exploiting the action of $\Dn{4}$, we need only treat  one of these vertices (say $C$) and two possibilities for its dual cell. 
  If $C^{\vee}$ is a right triangle, then item (iii) of the lemma confirms, both,  that $C$ is a $1$-valent vertex of $|\bclassns|$, and that the corresponding tritangent  has  lifting multiplicity one. 

  On the contrary, if $C^{\vee}$ has no right-angle, we may assume  that $(2,3)$ is one of its vertices.
  Then, by item (iv) of the lemma, we can find a unique  $1$-cell $\sC$  intersecting $e$ at $C$ with a trivalent member $\Lambda'$ of slope one with $v_0$ contained in the boundary of $(1,1)^{\vee}$, having a type (1a) tangency on its positive vertical leg and a type (3c) one on its positive horizontal leg, whose connected in $\Gamma\cap \Lambda$ includes $C$. In this situation,~\autoref{lm:dim1Parallelogram} (vi)-(vii) confirms both, the existence of a chain of  segments of $\bclassns$ that is attached to $e$ along $C$, and that contains exactly one liftable member (with multiplicity four), namely, the  $1$-valent vertex of the chain complementing $C$.

  Note that all tritangents corresponding to elements of this chain other than $C$,  share the same type (3c) tangency along  the horizontal edge of $\Gamma$ containing $C$. Furthermore, this tangency lies in the positive horizontal leg of all its members. In particular, this chain of cells intersect $e$ only at $C$.

  To conclude, we must show that the two (possible empty) chains attached to $e$ along $C$ and $C'$ do not intersect. If $C^{\vee}$ or $(C')^{\vee}$ are right triangles, the corresponding chain is empty and the statement becomes tautological. In turn, when both triangles are obtuse, we may fix $C^{\vee}$ to be the triangle with vertices $(1,1)$, $(2,2)$ and  $(2,3)$ chosen earlier. We have two options for $(C')^{\vee}$ since its undetermined vertex can be $(1,0)$ or $(3,2)$.   In both cases, no member of the chain can have a type (3c) tangency along the horizontal edge of $\Gamma$ adjacent to $C$.  This confirms that the chains attached to $C$ and $C'$ cannot intersect.
\end{proof}

The next result follows directly from~\autoref{cor:1And4ax2Class}:

\begin{proposition}\label{pr:1aAnd4aetcx2} 
  Let $\Lambda$ be a generic member in the relative interior of a $1$-dimensional cell of $\bclassns$ with a type (1a) diagonal tangency and no type (3c) ones.  Then, the lifting partition of $\tclass$ equals $(0,0,2,0)$. 
\end{proposition}

Our next two statements are analogous to 
{Propositions}~\ref{pr:dim2Parallelogram} and~\ref{pr:13acOr3cc1sameOrNot}, respectively.

\begin{proposition}\label{pr:dim1Parallelogram}
  Let $\Lambda$ be a generic member in the relative interior of a $1$-dimensional cell of $\bclassns$ with either a type (3aa) diagonal tangency or two tangencies of types (1a) and (3c), respectively, on adjacent legs  of $\Lambda$. Then, the lifting partition of $\tclass$ equals $(0,0,2,0)$. 
\end{proposition}

\begin{proof} From~\autoref{tab:combClassificationGenericPerDim} we know that if $\Lambda$ has a type (3aa) tangency, its complementary tangencies are of types (1a) and (3c). In particular, we can treat both possibilities listed in the statement simultaneously if we pick alternative $\Dn{4}$-representatives to those listed in the table, namely, where the
  type (3c) and (1a) tangencies (labeled $P'$ and $P''$) lie on the  positive  horizontal and vertical legs of $\Lambda$, respectively.

  We write $e'$ and $e''$ for the edges of $\Gamma$ containing $P'$ and $P''$ in their relative interior, and let $A$ be the leftmost endpoint of $e'$. This is precisely the setting of~\autoref{lm:dim1Parallelogram} and~\autoref{fig:type1-3cParallelogramDim1}. By items (i) and (ii) of that lemma, $\Lambda$ belongs to a segment in the support of  $\bclassns$, which has $A$ as an endpoint.

  We analyze two different situations, depending on the nature of $A$.   If $A\in e''$, by item (v) of the aforementioned lemma, $|\bclassns|$ contains a member satisfying the conditions of~\autoref{pr:1aAnd4aetcx2}. Thus, the lifting partition of $\tclass$ is as expected.

  On the contrary, when $A\notin e''$, the  tritangent associated to $A$ lifts (and with multiplicity four) by item (iv) of the lemma. In turn, item (viii) confirms that $A$ is a $1$-valent vertex of $|\bclassns|$.
  Finally, items (vi) and (vii) ensure that $|\bclassns|$ is obtained from $\sC$ by attaching to it a (possibly empty) chain of segments at its leftmost vertex. If the chain is empty, the leftmost vertex of $\sC$ lifts. In all other cases, this vertex does not lift, but the chain has precisely one liftable member, namely its complementing $1$-valent vertex. In both cases, we obtain the expected lifting multiplicity.
\end{proof}

\begin{proposition}\label{pr:3c3acOr3cc13csameOrNot}
Let $\Lambda$ be a generic member in the relative interior of a $1$-dimensional cell of $\bclassns$ with a non-transverse diagonal tangency. Suppose at least one of the vertices of $\Lambda$ is not in $\Gamma$. Then, the lifting partition of $\tclass$ equals $(0,0,2,0)$.
\end{proposition}

\begin{proof} We let $P'$ be the non-transverse diagonal tangency of $\Lambda$. Looking at the sample members listed in the $(1,(iv))$-entry of~\autoref{tab:combClassificationGenericPerDim}, we know that one of the two remaining tangencies of $\Lambda$ has type (1a). We label it by $P''$. There are two options for the remaining tangency, which we call $P$, depending on whether or not the vertex of $\Lambda$ adjacent to the leg containing $P''$ belongs  to $\Gamma$. In all cases, this is the sole vertex of $\Lambda$ that can move. In addition, a simple inspection confirms that $\mult(\Lambda, P)=2$ in all cases.

  As in the proof of~\autoref{pr:13acOr3cc1sameOrNot} we use the auxiliary results from~\autoref{sec:technicalLemmasDim1-2}. To do so, we assume that the  edge of  $\Lambda$ has slope $-1$ and that $P''$ lies in a  leg adjacent to  the vertex $v_0$. 

  First, assume $v_0\in \Gamma$, and that the tangency at $P''$ is horizontal. Our hypothesis ensures that $P'$ has type (3ac) with $v_1\notin \Gamma$.
  \autoref{lm:MoveDown3-aAnd1H} confirms that moving $v_0$ away from $v_1$ leads to precisely  one liftable member, of multiplicity $4$.

  To treat the movement of $v_0$ towards $v_1$, we use~\autoref{lm:MoveUp3caAnd1H} and the notation established in~\autoref{fig:3caAnd1Horizontal}. Indeed, whenever $B'\preceq_v C'$, item (i) of the lemma confirms that the movement produces exactly one liftable member, with multiplicity $4$. On the contrary, if $C'\prec_v B'$, $P$ is necessarily a type (3c) vertical tangency, as indicated in the subdivision labeled (II) in the figure. Moving $v_0$ past $C'$ and $v_1$ upwards within the chamber $(0,1)^{\vee}$, item (ii) of the lemma produces a member $\Lambda'$ with a type (3a) vertical tangency, and a type (1a) horizontal one in the boundary of  $(1,1)^{\vee}$. Note that the vertical edge containing $C$ and the edge of $\Gamma$ containing $P''$ are not adjacent, so we cannot encounter a subdivision symmetric that of \autoref{fig:3aAnd1VerticalLeftMove} (III). This fact ensures no member of $|\bclassns|$ has a type (4b) tangency. The existence of $\Lambda'$ falls into the framework of~\autoref{pr:3a13cDim1} after acting by $\Dn{4}$. Thus,  $\tclass$ has the expected lifting partition.

  To conclude, we treat the case when $v_0\notin \Gamma$. Then, up to symmetry we may assume that $P''$ belongs to the negative vertical leg of $\Lambda$, so $v_0\in (3,2)^{\vee}$.  We prove the statement by exploiting our findings from~\autoref{lm:MoveDownUpRight3-cAnd1V}.  First, item (i) of the aforementioned lemma ensures that when moving $v_0$ away from $v_1$ we encounter a unique liftable member, with multiplicity four. In turn, moving $v_0$ towards $v_1$ yields two possible situations.

  If the edges of $\Gamma$ containing $P'$ and $P''$ are not adjacent, item (iii) of the lemma produces a single liftable member, again with multiplicity four. On the contrary, if these two edges are adjacent, we encounter a member $\Lambda'$ with its lower vertex in $\Gamma$, a type (1a) tangency on its positive horizontal leg, the same tangency at $P$ and either a type (3ca) or (3aa) along the edge of $\Lambda'$. In the first case, the result follows from our previous arguments, whereas the second one is a consequence of~\autoref{pr:dim1Parallelogram}.
\end{proof}

  \section{The dimension of the complexes $\tclass$, $\bclass$ and $\bclassns$}\label{sec:comp-dimens-tclass}

  In 
  {Sections}~\ref{sec:polyhedra-bounded-complex} and~\ref{sec:non-special-bounded} we introduced two subcomplexes for each tritangent class $\tclass$ to $\Gamma$, namely $\bclass$ and $\bclassns$, by successively peeling off cells that did not contribute to the lifting partition of $\tclass$ while remaining connected.  The natural inclusions $\bclassns\subseteq \bclass \subseteq \tclass$ yield the inequalities
  \begin{equation}\label{eq:dimComparison}
    \dim \bclassns \leq \dim \bclass \leq \dim \tclass \leq 3
  \end{equation}
  \autoref{thm:comparingDimensions} provides instances where equalities between these quantities are attained. It is natural to asks which triples of values for the dimension of each complex can occur. Our main theorem in this section provides an answer and confirms that the invariant restricting the lifting partition of a tritangent class is the dimension of its associated bounded non-special complex.

  \begin{theorem}\label{thm:dimComparison} The tuple recording the dimensions of the complexes $\bclassns$, $\bclass$ and $\tclass$, respectively, has one of the following 11 possible values:
    \begin{equation*}\label{eq:tupleDimValues}   (0,0,0), (0,0,1), (0,0,2), (0,1,1), (1,1,1), (1,1,2), (1,1,3),  (1,2,2), (2,2,2), (2,2,3) \text{ or }(3,3,3).
    \end{equation*}
    All tuples are realized for appropriate choices of the curve $\Gamma$. 
  \end{theorem}

    \begin{figure}[t]
      \includegraphics[scale=0.85]{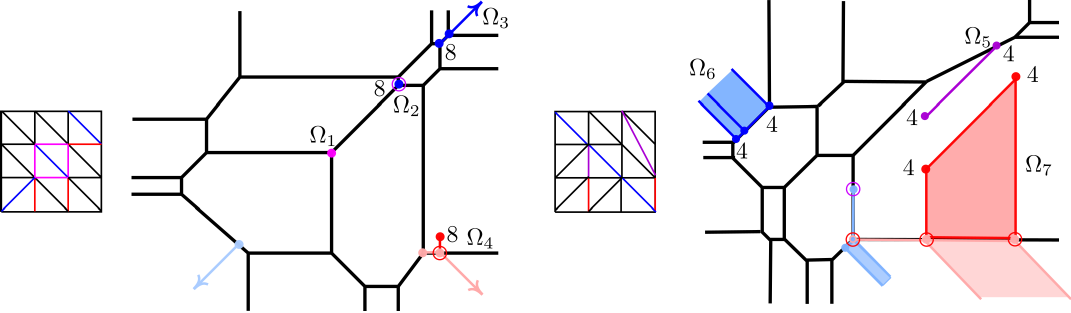}
      \caption{Examples realizing the tuples of dimensions with first entry either 0 or 1.\label{fig:dim01Examples}}
    \end{figure}
    
  The remainder of this section is devoted to the proof of~\autoref{thm:dimComparison}. We start by providing examples of tritangents with tuples of dimensions  $(\dim \bclassns, \dim \bclass, \dim \tclass)$ realizing all values listed in the statement, grouped according to the value of $\dim \bclassns$. We use the same conventions from~\autoref{ex:realizableLifts} to record the relevant tritangent classes associated in all four examples below.
  
  \begin{example}\label{ex:0etc}
    Consider the curve $\Gamma$ and the four tritangent classes $\tclass_1$, $\tclass_2$, $\tclass_3$ and $\tclass_4$ depicted in the left of~\autoref{fig:dim01Examples}. The associated bounded non-special complexes are $0$-dimensional, so~\autoref{thm:partitionDim0} ensures the lifting partition of all of them  equals $(0,0,0,1)$.

    The combinatorial type of $\Gamma$ and the lengths of all edges in its skeleton ensures that, $(\tclass_i)^b_{ns} = \bclass_i$  for $i\in \{1,2,3\}$, whereas $\bclass_4$ is supported on a segment with endpoint $|(\Omega_4)^b_{ns}|$. We identify $|\bclass_4|$ with the vertical red segment recording the location of the vertex $v_1$ paired with the pink horizontal one adjacent to it recording the position of $v_0$). In turn, $\tclass_1=\bclass_1$, whereas $\tclass_2$ is obtained from $\bclass_2$ by adding a ray with endpoint $|\bclass_2|$  recording the position of the vertex $v_0$, seen as a light-blue slope one ray in the figure. Similarly,  the complex $\tclass_3$ is 2-dimensional, and its support is  the cone with apex $|\bclass_3|$, generated by two rays with 
    directions $(-1,-1,0)$ and $(0,0,1)$, respectively.
    
Finally, $\tclass_4$ is obtained from $\bclass_4$ by attaching a ray with direction $(-1,-1,-1)$ along a vertex of this complex. 
It follows from this analysis that the tuple of dimensions for each $\tclass_i$ equals $(0,0,0)$, $(0,0,1)$, $(0,0,2)$  and $(0,1,1)$,  for $i=1,2,3,4$, respectively.
    \end{example}

  \begin{example}\label{ex:1etc} We consider the curve $\Gamma$ seen on the left of~\autoref{fig:examplesLift} and the tritangent $\tclass_1$. We have $(\Omega_1)^b_{ns}=\bclass_1$, and $|\tclass_1|$ is obtained from $|\bclass_1|$ by attaching a flap with direction $(-1,-1,1)$ along the edge of $\bclass_1$ recording those members with lower vertex $v_0$ lying in the unique slope $-1/3$ edge of $\Gamma$. Thus, the tuple of dimensions for this class equals $(1,1,2)$.

    Examples realizing the remaining tuples of dimensions with first entry $1$ can be found on the right of~\autoref{fig:dim01Examples}: they correspond to the tritangent classes $\tclass_5$, $\tclass_6$ and $\tclass_7$ to a tropical curve, which we label $\Gamma'$. In what follows, we show that the tuple of dimensions for each of them equals $(1,1,1)$, $(1,1,3)$ and $(1,2,2)$, respectively.

    By construction, we have $(\Omega_5)^b_{ns} = \bclass_5= \tclass_5$. Similarly we have that $(\Omega_6)^b_{ns} = \bclass_6$. All members in the support of $\bclass_6$ are trivalent, with an edge of slope $-1$. The position of the top-vertex of all its members is restricted to the slope 1 edge in the boundary of $(0,3)^{\vee}$, subdivided into two consecutive segments, while the  lower vertices lie in adjacent segments in the boundary of $(2,1)^{\vee}$ that are adjacent to a positive horizontal leg of $\Gamma'$. The support of the complex $\tclass_6$ is obtained from this by adding  a flap with direction $(0,0,-1)$ attached along one of the two segments in $\bclass_6$ and a cone with directions $(0,0,-1)$ and $(-1,-1,0)$ along the remaining one.

Finally, we describe $\tclass_7$ and its two subcomplexes. The set $(\Omega_7)^b_{ns}$ is a segment. All members in its support are trivalent, with a slope one edge. In addition, the  lower vertex $v_0$ of each of them agrees with the vertex of $\Gamma'$ dual to the triangle with vertices $(3,0)$, $(2,1)$ and $(3,1)$ in the Newton subdivision of $\sextic$.  In turn, the support of the complex $\bclass_7$ can be identified with the trapezoid seen in the picture, which records the location of the top vertex $v_1$ of each member. The segment supporting $(\Omega_7)^b_{ns}$ becomes the slope one edge in the boundary of this polytope. Finally, the support of the complex $\tclass_7$ is obtained by attaching to $|\tclass_7|$ a flap with direction $(1,-1,-1)$ along the unique edge of corresponding to the $4$-valent members in its support.
  \end{example}

      \begin{example}\label{ex:23etc} Consider the two tropical $(3,3)$-curves seen on each side of~\autoref{fig:examplesLift} and the tritangent classes labeled $\tclass_2$, $\tclass_3$ and $\tclass_4$. In all three cases, $(\Omega_i)^b_{ns} = \bclass_i$ for $i=2,3,4$,  whereas $\tclass_i=\bclass_i$ for $i=3,4$ and $|\tclass_2|$ is obtained from $|\bclass_2|$ by attaching a flap with direction $(-1,-1,1)$ along the edge of $|\bclass_2|$ corresponding to those trivalent members of with slope one edge and lower vertex  $v_0$ located on the unique edge of $\Gamma$ with slope $-3$. Thus, the tuple of dimensions for each $\tclass_i$ equal $(2,2,3)$, $(2,2,2)$ and $(3,3,3)$, for $i=2,3,4$, respectively.
  \end{example}

      \begin{remark} The examples seen on the left~\autoref{fig:examplesLift} witness a more general fact. Let $\Lambda$ be  tritangent containing a  (4b) tangency and let $\eta$ be the tuple of  tuple of dimensions for the tritangent class $\tclass$ containing $\Lambda$. If $\Lambda$ contains a (1a) or a (2a) tangency, then $\eta =(2,2,3)$. In turn, if it contains a (3c) tangency, then $\eta=(1,1,2)$ or $(1,2,2)$. As~\autoref{fig:specialConfigurations} (I) confirms, the latter occurs if, and only if,  $|\tclass|$ has a member with a type  (6b') tangency complemented by a (3c) one.
      \end{remark}
      
      \begin{proof}[Proof of~\autoref{thm:dimComparison}] We fix a curve $\Gamma$ and a tritangent class $\tclass$ on it. We let $\eta = (\eta_1,\eta_2, \eta_3)$ be the tuple of dimensions associated to $\tclass$. We argue by contradiction, and assume $\eta$ is outside the set of 11 tuples listed in the statement.
        This fact combined with~\eqref{eq:dimComparison} forces  $\eta$ to be on the set
        \[
        \{ (0,0,3), (0,1,3), (0,2,3), (0,2,2), (0,1,2), (1,2,3), (2,3,3)\}.
        \]
          {Lemmas}~\ref{lm:topDimensionalBoundedProper} and~\ref{lm:ineqBoundedDimensions}   below discard the first four elements of this set. The next two cannot occur by 
          {Lemmas}~\ref{lm:no012} and~\ref{lm:no123}. The last case is voided by~\autoref{pr:dimSpecialCells}.
      \end{proof}
      
      The following result allows us to prove the first part of~\autoref{thm:dimComparison} by imposing restrictions on the tuples of dimensions with last entry value 3.

      \begin{lemma}\label{lm:topDimensionalBoundedProper} Let  $\tclass$ be a full-dimensional tritangent class of $\Gamma$. If $\dim \bclass <\dim \tclass$, then $\dim \bclassns\geq 1$. Furthermore, if equality holds, then the support of $\tclass$ contains a trivalent member with three type (1a) tangencies on its unique edge.
      \end{lemma}

      \begin{proof} The bidegree of $\Gamma$ combined with 
        {Propositions}~\ref{pr:3dCells} and~\ref{pr:unbounded} ensure that, up to $\Dn{4}$-symmetry, $|\tclass|$ admits a trivalent member $\Lambda$ with one of three configurations depicted in~\autoref{fig:3optionsNonGeneric}. For each of them, we show that local moves of $\Lambda$ produce a member in a cell of $\bclassns$ of dimension at least one. We analyze each case separately, which we label from left to right as $\Lambda_1$, $\Lambda_2$ and $\Lambda_3$, respectively.

      \begin{figure}[t]
        \includegraphics[scale=0.35]{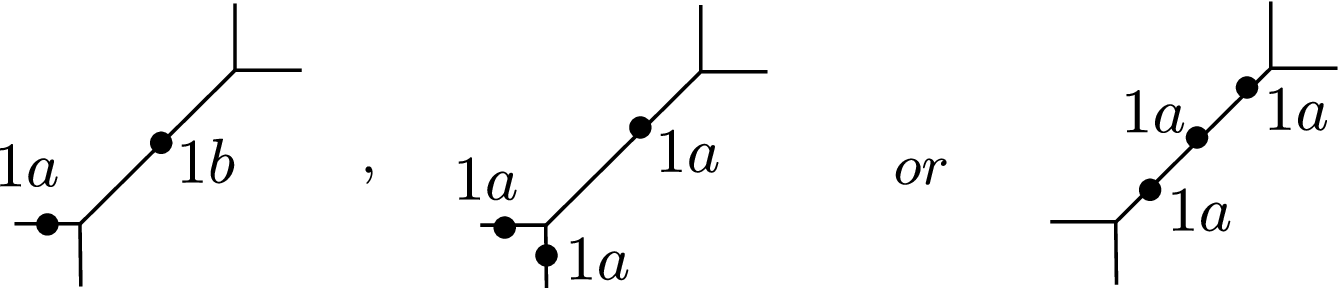}
        \caption{Possible members featured on $|\tclass|$ with $\dim \bclass < \dim \tclass = 3$.\label{fig:3optionsNonGeneric}}
      \end{figure}

        We start with $\Lambda_1$. By moving the vertex $v_1$ towards $v_0$, the type (1b) tangency becomes a (4b) one at $v_1$. The resulting tritangent lies in the support of a cell of $\bclassns$ of dimension two by 
        {Lemmas}~\ref{lm:OpenCellsTrivalent} and~\ref{lem:Type4And4pCells}, so $\dim \bclassns= 2$.

        In turn, moving the vertex $v_1$ of $\Lambda_2$ towards $v_0$ replaces the (1a) diagonal tangency by a (4a) one at $v_1$. The partial Newton subdivision of $\sextic$ imposed by this member is symmetric to the one seen in~\autoref{fig:type1-1Parallelogram}. \autoref{cor:tritangentClasses} confirms that $\bclassns$ contains a 2-dimensional polygon, so $\dim \bclassns =2$.

        Finally, moving both vertices of $\Lambda_3$ toward each other, produces a member with three diagonal tangencies: one of type (1a) and two of type (4a). Thus, $\dim \bclassns \geq 1$ by the same two lemmas used for $\Lambda_1$. Comparing this lower bound with the value of $\dim \bclassns$ obtained for $\Lambda_1$ and $\Lambda_2$ confirms the second part of the statement.        
      \end{proof}

      Our next lemma restricts the dimension of $\bclassns$ and $\bclass$ whenever their dimensions do not agree:

      \begin{lemma}\label{lm:ineqBoundedDimensions} 
        Let $\tclass$ be a tritangent class of $\Gamma$ with $\dim \bclassns <\dim \bclass$. Then,  $\dim \bclassns \leq 1$ and $\dim \bclass = \dim \bclassns +1$.
      \end{lemma}

      \begin{proof} \autoref{pr:dimSpecialCells} combined with the dimension assumptions ensures that  $\dim \bclassns < \dim \bclass < 3$, so $\dim \bclassns \leq 1$. Thus, to confirm the second identity in the statement we must show that $(\dim \bclassns, \dim \bclass)= (0,2)$ cannot occur. We argue by contradiction.

        Let  $(\Lambda, P,P,P'')$ be a tritangent tuple corresponding to a generic member $\Lambda$ in the relative interior of a two-dimensional special cell in $\bclass$. Up to $\Dn{4}$-symmetry, its local moves within $|\bclass|$ are described in~\autoref{fig:specialConfigurations}.
        Following the notation given there, we see that the vertical tangency $P''$ must be of type (1a). In turn, the dimension of $\bclassns$ confirms that the partial Newton subdivision of $\sextic$ must be the one labeled (III) in the figure. However, the type of $P''$ confirms that $\dim \bclassns \geq 1$, since we can fix $v_0$ at $C$ and move $v_1$ towards and away from $v_0$ while remaining in $|\bclassns|$.         
      \end{proof}
      
      Our last  lemmas discard  $(0,1,2)$ and $(1,2,3)$ as valid tuple of dimensions of tritangent classes:
      
\begin{lemma}\label{lm:no012} No tritangent class can have $(0,1,2)$ as a tuple of dimensions.
      \end{lemma}

\begin{proof} We argue by contradiction and assume that $\eta = (0,1,2)$ is the tuple of dimensions of a tritangent class $\tclass$ on a fixed curve $\Gamma$. Since $\dim \bclassns <\dim \bclass$, we know  that $\bclass$ contains a special cell of dimension one. We let $(\Lambda, P, P',P'')$ be a tritangent tuple corresponding to a member in the relative interior of such a cell. In the notation of~\autoref{fig:specialConfigurations}, we know that $P''$ has type (3c). Thus, the support of $\bclassns$ must include one of the last three members described in   \autoref{lm:specialConfigurations}.

  Among them, the only one that is compatible with the condition $\dim \bclassns = 0$ is the one with a type (5a) tangency at $v_0$, and a type (3c) tangency at $P$, corresponding to the partial Newton subdivision of $\sextic$  seen in~\autoref{fig:specialConfigurations} (III). In this situation, $\bclass$ is a segment, and $|\tclass|$ is obtained by attaching to its endpoint outside $|\bclassns|$ (recording the unique $4$-valent member of $|\bclass|$) a ray with direction $(-1,-1,-1)$. Thus, $\dim \tclass = 1$, which contradicts our choice of $\eta$.
      \end{proof}

        \begin{figure}[t]
          \includegraphics[scale=0.35]{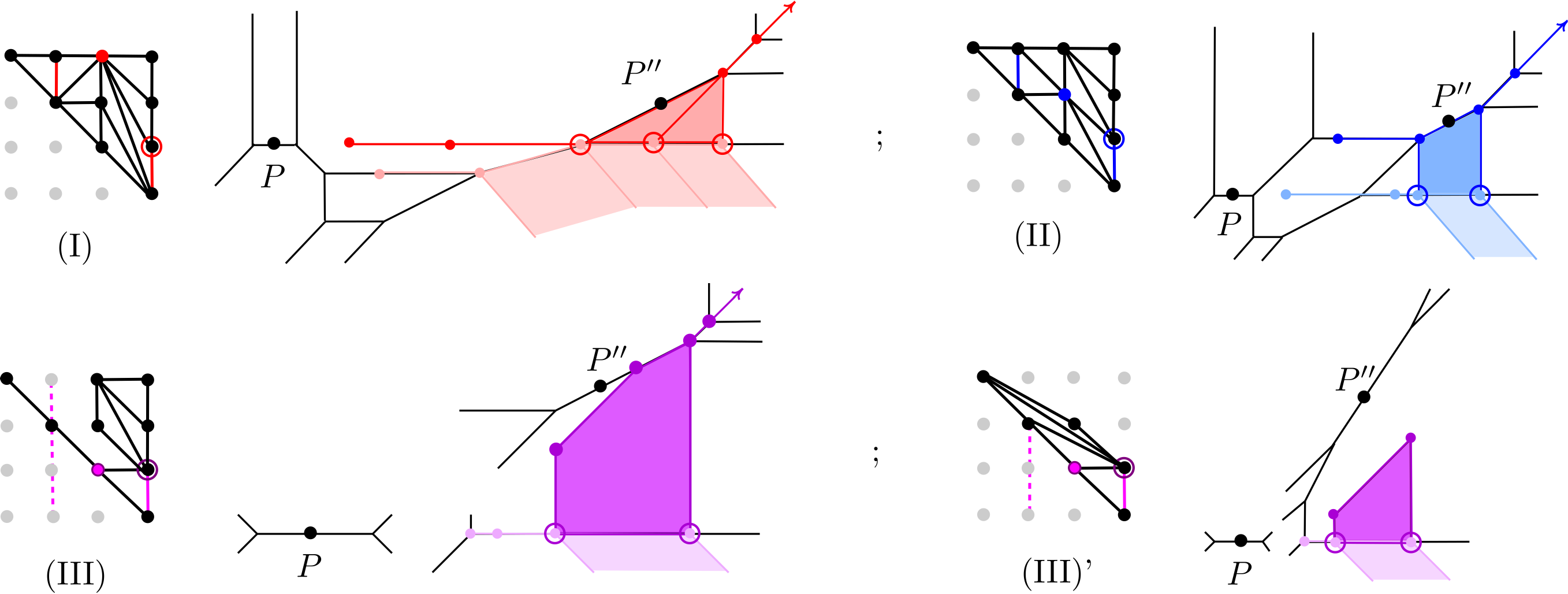}
          \caption{Partial Newton subdivisions of $\sextic$ when $\Gamma$ admits a tritangent class with dimension tuple $(1,2,3)$ and the corresponding classes for a tritangent tuple $(\Lambda, P,P',P'')$ in a top-dimensional special cell of the given class.   \label{fig:123NewtonSubdivision}}
                  \end{figure}

      \begin{lemma}\label{lm:no123} The tuple of dimensions for a tritangent class cannot have value $(1,2,3)$.
      \end{lemma}
      \begin{proof} We argue by contradiction. Fix a curve $\Gamma$ and a tritangent class $\tclass$ with tuple of dimensions $\eta = (1,2,3)$.   By~\autoref{lm:topDimensionalBoundedProper} we know that the 
        complex $\bclassns$ contains a trivalent member $\Lambda'$ with three diagonal tangencies, one of type (1a) and two of type (4a). Thus, the Newton subdivision of $\sextic$ includes the three edges on one of the two main diagonals of the square of side length three.

        On the other hand, the inequality $\dim \bclassns <\dim \bclass=2$ confirms that $\bclass$ has a 2-dimensional cell with a special tritangent tuple $(\Lambda, P, P', P'')$ with $P''$ of type (1a) tangency, in the notation of~\autoref{fig:specialConfigurations}. In turn, one of the first four members seen in~\autoref{lm:specialConfigurations} must lie in the support of $\bclassns$. The dimension of $\bclassns$ confirms that $P$ has type (3c).

        These facts combined with the existence of $\Lambda'$ in $|\bclass|$ confirms that the three antidiagonal edges in the square of side length three are cells in  the Newton subdivision of $\sextic$. Letting  $v$ be the vertex of the subdivision forming a triangle with the vertices $(3,0)$ and $(3,1)$, we conclude that $v$ also forms a triangle with the vertical edge $P^{\vee}$. These two observations restrict the partial Newton subdivision of $\sextic$ to the ones depicted in~\autoref{fig:123NewtonSubdivision}. The picture shows the corresponding tritangent classes for each input $(\Lambda, P,P',P'')$ following our drawing conventions from~\autoref{ex:realizableLifts}. In all cases, we see that $\dim \tclass = 2$, which contradicts our choice of $\eta$.
              \end{proof}

\normalsize

  \bibliographystyle{plain}

  \vspace{2ex}
  

  \noindent
\textbf{\small{Authors' addresses:}}
\smallskip
\

\noindent
\small{M.A.\ Cueto,  Mathematics Department, The Ohio State University, 231 W 18th Ave, Columbus, OH 43210, USA.
\\
\noindent \emph{Email address:} \texttt{cueto.5@osu.edu}}
\vspace{2ex}

\noindent
\small{H.~Markwig, Eberhard Karls Universit\"at T\"ubingen, Fachbereich Mathematik, Auf der Morgenstelle 10, 72108 T\"ubingen, Germany.
  \\
  \noindent \emph{Email address:} \texttt{hannah@math.uni-tuebingen.de}}
\vspace{2ex}

\noindent
\small{Y.~Ren, Mathematical Sciences \& Computer Science Building, Upper Mountjoy Campus, Stockton Road, Durham University, DH1 3LE, UK.
 \\
  \noindent \emph{Email address:} \texttt{yue.ren2@durham.ac.uk}}

\end{document}